\documentclass [10pt] {article}
\usepackage[dvipdfm]{graphicx}
\usepackage{amsfonts}
\usepackage{amsmath}
\usepackage{amssymb}
\usepackage{bm}
\usepackage[alphabetic]{amsrefs}
\usepackage{fullpage}
\usepackage{theorem}
\usepackage{tensor}
\usepackage{easybmat}
\usepackage{multirow}
\usepackage{enumerate}
\usepackage{textcomp}
\usepackage{etex}
\usepackage[all]{xy}

\newtheorem{thm}{Theorem}[section]

\newtheorem{lem}[thm]{Lemma}
\newtheorem{prop}[thm]{Proposition}

\theorembodyfont{\rmfamily}
\newtheorem{defn}[thm]{Definition}

\newtheorem{rem}[thm]{Remark}

\numberwithin{equation}{section}
\newenvironment{proof}{\noindent \emph{Proof.}}{\hspace{\stretch{1}}$\Box$}

\newcommand{\mcA} {\mathcal{A}}
\newcommand{\mcB} {\mathcal{B}}
\newcommand{\mcC} {\mathcal{C}}

\newcommand{\mcF} {\mathcal{F}}

\newcommand{\mcK} {\mathcal{K}}

\newcommand{\mcM} {\mathcal{M}}
\newcommand{\mcN} {\mathcal{N}}

\newcommand{\mcR} {\mathcal{R}}

\newcommand{\mcU} {\mathcal{U}}
\newcommand{\mcV} {\mathcal{V}}

\newcommand{\dd} {\mathrm{d}}

\newcommand{\ii} {\mathrm{i}}

\newcommand{\gr} {\mathrm{gr}}

\newcommand{\diag} {\mathop{\mathrm{diag}}}
\newcommand{\sspan} {\mathop{\mathrm{span}}}

\newcommand{\ctr}[1] {#1 \lrcorner}

\newcommand{\ind} {\indices}

\newcommand{\lb} [1] {{\left[ #1 \right. }}
\newcommand{\rb} [1] {{\left. #1 \right] }}

\newcommand{\Rho} {\mathrm{P}}

\newcommand{\Ad} {\mathrm{Ad}}

\newcommand{\Mat} {\mathrm{Mat}}

\newcommand{\GL} {\mathrm{GL}}

\newcommand{\SO} {\mathrm{SO}}

\newcommand{\UU} {\mathrm{U}}

\newcommand{\Tgt} {\mathrm{T}}

\newcommand{\Sim} {\mathrm{Sim}}

\newcommand{\so} {\mathfrak{so}}
\newcommand{\slie} {\mathfrak{sl}}
\newcommand{\glie} {\mathfrak{gl}}
\newcommand{\g} {\mathfrak{g}}

\newcommand{\prb} {\mathfrak{p}}

\newcommand{\uu} {\mathfrak{u}}

\newcommand{\mfz} {\mathfrak{z}}
\newcommand{\simalg} {\mathfrak{sim}}

\newcommand{\R} {\mathbb{R}}
\newcommand{\C} {\mathbb{C}}
\newcommand{\Z} {\mathbb{Z}}

\begin{document}

\title{The complex Goldberg-Sachs theorem in higher dimensions}
\author{Arman Taghavi-Chabert\\
{\small Masaryk University, Faculty of Science, Department of Mathematics and Statistics,}\\
 {\small Kotl\'{a}\v{r}sk\'{a} 2, 611 37 Brno,
Czech Republic } }
\date{}

\maketitle

\begin{abstract}
We study the geometric properties of holomorphic distributions of totally null $m$-planes on a $(2m+\epsilon)$-dimensional complex Riemannian manifold $(\mathcal{M} , \bm{g})$, where $\epsilon \in \{ 0 ,1 \}$ and $m \geq 2$. In particular, given such a distribution $\mathcal{N}$, say, we obtain algebraic conditions on the Weyl tensor and the Cotton-York tensor which guarrantee the integrability of $\mathcal{N}$, and in odd dimensions, of its orthogonal complement. These results generalise the Petrov classification of the  (anti-)self-dual part of the complex Weyl tensor, and the complex Goldberg-Sachs theorem from four to higher dimensions.

Higher-dimensional analogues of the Petrov type D condition are defined, and we show that these lead to the integrability of up to $2^m$ holomorphic distributions of totally null $m$-planes. Finally, we adapt these findings to the category of real smooth pseudo-Riemannian manifolds, commenting notably on the applications to Hermitian geometry and Robinson (or optical) geometry.
\end{abstract}

\section{Introduction and motivation}
One of the milestones in the development of general relativity, the Goldberg-Sachs theorem, first formulated in 1962, states \cite{Goldberg2009} that a four-dimensional Einstein Lorentzian manifold admits a shearfree congruence of null geodesics if and only if its Weyl tensor is algebraically special. It has proved invaluable in the discovery of solutions to Einstein's field equations, and the Kerr metric is a prime example of its application \cite{Kerr1963}.

A number of versions of the Goldberg-Sachs theorem subsequently appeared, and revealed a far deeper insight into the geometry of pseudo-Riemannian manifolds. To start with, the Einstein condition can be weakened to a condition on the Cotton-York tensor \cites{Kundt1962,Robinson1963}, whereby the conformal invariance of the theorem is made manifest. Further, the theorem turns out to admit a complex holomorphic counterpart \cites{Pleba'nski1975,Penrose1986}, and other variants on real pseudo-Riemannian manifolds of arbitrary metric signatures \cites{Przanowski1983,Apostolov1998,Apostolov1997,Ivanov2005,Gover2010}. In all these versions, real or complex, the underlying geometric structure is a \emph{null structure}, i.e. an integrable distribution of totally null complex $2$-planes. In the real category, the metric signature induces an additional reality structure on the complexified tangent bundle, which adds a particular `flavour' to this null geometry. Thus, in Lorentzian signature, a null structure is equivalent to a \emph{Robinson structure} (also known as an \emph{optical structure}), i.e. a congruence of null geodesics along each of which a complex structure on its screenspace is preserved \cites{Nurowski2002,Trautman2002}. In particular, such a congruence is shearfree. Similarly, a Hermitian structure on a proper Riemannian manifold can be identified with a null structure.

That distributions of totally null complex $2$-planes on pseudo-Riemannian manifolds represent fundamental geometric objects forms the backbone of twistor theory, or more generally spinor geometry, and a number of geometric properties of spacetimes can be nicely formulated in this setting \cites{Penrose1967,Penrose1986}. These ideas generalise to higher dimensions: in even dimensions, a null structure is now an integrable distribution of maximal totally null planes; in odd dimensions, the definition is identical except that the orthogonal complement to the null distribution is also required to be integrable. Applications of higher-dimensional twistor geometry can be seen in the work of Hughston and Mason \cite{Hughston1988}, who give an even-dimensional generalisation of the Kerr theorem as a means to generating null structures on open subsets of the conformal complex sphere. More recently, it was noted by Mason and the present author \cite{Mason2010} that the higher-dimensional Kerr-NUT-AdS metric \cite{Chen2006} is characterised by a discrete set of Hermitian structures, and its Weyl tensor satisfies an algebraic condition generalising the four-dimensional Petrov type D condition. As in four dimensions \cite{Walker1970}, these results were shown to arise from the existence of a conformal Killing-Yano $2$-form.

Such findings suggest that a higher-dimensional Goldberg-Sachs theorem should be formulated in the context of null structures, and to this end, an invariant classication of the curvature tensors with respect to an almost null structure appears to be the most natural framework. Such a classification already exists in almost Hermitian geometry \cites{Falcitelli1994,Tricerri1981}, but curvature prescriptions that are sufficient for the integrability of an almost Hermitian structure do not appear to have been investigated. In Lorentzian geometry, the Weyl tensor has also been subject to a classification \cites{Coley2004,Coley2004a,Milson2005,Pravda2004,Pravda2007,Ortaggio2007} which has mostly focused on the properties of null geodesics. In fact, according to this approach, the geodesic part of the Goldberg-Sachs theorem admits a generalisation to higher dimensions \cite{Durkee2009}, but its shearfree part does not. In fact, shearfree congruences of null geodesics in more than four dimensions, which, as remarked in \cite{Trautman2002a}, are no longer equivalent to Robinson structures, have not featured so prominently  in the solutions to Einstein's field equations \cites{Frolov2003,Pravda2004}.

On the other hand, the present author \cite{Taghavi-Chabert2011} put forward a higher-dimensional generalisation of the Petrov type II condition, which, together with a genericity assumption on the Weyl tensor and a degeneracy condition on the Cotton-York tensor, guarantees the existence of a Robinson structure on a five-dimensional Lorentzian manifold. A counterexample to the converse is given: the black ring solution \cite{Emparan2002} admits pairs of null structures, but the Weyl tensor fails to be `algebraically special relative to it' in the sense of Theorem \ref{thm-GS_intro} below.
In the same reference, it is also conjectured that these results are also true in arbitrary dimensions, and in the holomorphic category. It is the aim of the paper to turn this conjecture into a theorem. To be precise, we shall prove
\begin{thm}\label{thm-GS_intro}
 Let $\mcN$ be a holomorphic distribution of totally null $m$-planes on a $(2m+\epsilon)$-dimensional complex Riemannian manifold $(\mcM, \bm{g})$, where $\epsilon \in \{ 0 , 1\}$ and $2m+\epsilon \geq 5$, and let $\mcN^\perp$ denote its orthogonal complement with respect to $\bm{g}$. Suppose the Weyl tensor and the Cotton-York tensor (locally) satisfy
\begin{align*}
 \bm{C} ( \bm{X} , \bm{Y}, \bm{Z} , \cdot ) & = 0 \, , & \bm{A} (\bm{Z} , \bm{X} , \bm{Y} ) & = 0 \, ,
\end{align*}
respectively, for all vector fields $\bm{X}, \bm{Y} \in \Gamma (\mcN^\perp)$, and $\bm{Z} \in \Gamma (\mcN)$. Suppose further that the Weyl tensor is otherwise generic. Then, the distributions $\mcN$ and $\mcN^\perp$ are (locally) integrable.
\end{thm}
In fact, we shall demonstrate more than this. We shall define further degeneracy classes of the Weyl tensor and Cotton-York tensor with respect to $\mcN$, and show that these also imply the integrability of $\mcN$ and $\mcN^\perp$.
We shall also be able to weaken the genericity assumption on the Weyl tensor in Theorem \ref{thm-GS_intro} to such an extent as to guarantee the integrability of up to $2^m$ canonical null distributions and their orthogonal complements. Consequently, Theorem \ref{thm-GS_intro} will be generalised to the category of smooth pseudo-Riemannian manifolds of arbitrary metric signature.

The structure of the paper is as follows. In section \ref{sec-algebra}, we lay bare the algebraic properties of null structures by means of their stabiliser $\prb$, say, which is well-known to be a parabolic Lie subalgebra of the complex special orthogonal group. Their properties are already well-documented in \cites{Baston1989,vCap2009}, and we use these sources to set up the algebraic background and the notation used throughout the paper.

These algebraic considerations are then translated into the language of vector bundles in section \ref{sec-geometry}. In particular, algebraic classes of Weyl tensors and Cotton-York tensors are defined in terms of $\prb$-invariant filtered vector bundles. We also examine the geometric characteristics of almost null structures such as integrability conditions and geodetic property.

In section \ref{sec-GS}, we present the main results of this paper. It begins with a restatement of the complex four-dimensional Goldberg-Sachs theorem in the notation introduced in section \ref{sec-geometry}. We discuss to which extent it may be generalised to higher dimensions. Buildling on \cite{Taghavi-Chabert2011}, we argue that the existence of a null structure together with a degenerate Cotton-York tensor does not necessarily lead to further, i.e. `special', degeneracy of the Weyl tensor, in the sense of Theorem \ref{thm-GS_intro}. On the other hand, we show that certain algebraic classes of the Weyl tensor, which generalise the Petrov type II and more degenerate, guarrantee the integrability of an almost null structure, provided that the Weyl tensor satisfies a genericity assumption, and the Cotton-York tensor is sufficiently degenerate. We then prove the conformal invariance of these results.

In section \ref{sec-GS_degenerate}, after a heuristic discussion on the genericity assumption on the Weyl tensor, we extend Theorem \ref{thm-GS_intro} to the case of multiple null structures, which may be viewed as a generalisation of the Petrov type D condition. This allows us to show how it also applies to real pseudo-Riemannian smooth manifolds of arbitrary metric signature, giving special attention to proper Riemannian, split signature and Lorentzian manifolds.

We end the paper with some remarks on the relation between the Goldberg-Sachs theorem and parabolic geometry.

We have collected the complex Bianchi identity in component form in an appendix.

\paragraph{Acknowlegments} The author would like to thank Jan Slovak for useful discussions. This work is funded by a SoMoPro (South Moravian Programme) Fellowship: it has received a financial contribution from the European
Union within the Seventh Framework Programme (FP/2007-2013) under Grant Agreement No. 229603, and is also co-financed by the South Moravian Region.

\section{Algebraic preliminaries}\label{sec-algebra}
This section is largely a down-to-earth application of the theory of parabolic Lie algebras given in \cite{vCap2009}. Other useful references on parabolic geometry and representation theory are \cites{Baston1989,Fulton1991}.

Let $(V , \bm{g} )$ be a $(2m+\epsilon)$-dimensional complex vector space, where $\epsilon \in \{ 0,1\}$, equipped with a non-degenerate symmetric bilinear form $\bm{g} : V \times V\rightarrow \C$.
If $U$ is a vector subspace of $V$, we shall denote its orthogonal complement with respect to $\bm{g}$ by $U^\perp$, and its dual by $U^*$. Fix an orientation for $(V, \bm{g})$, and denote by $*$ the Hodge duality operator on the exterior algebras $\bigwedge^\bullet V$ and $\bigwedge^\bullet V^*$. The group of automorphisms of $V$ preserving $\bm{g}$ of determinant $1$ is the complex special orthogonal group $\SO(2m+\epsilon,\C)$. It will be denoted $G$ for short, and its Lie algebra $\so(2m+\epsilon,\C)$ by $\g$.

\begin{defn}
A \emph{null structure on $(V , \bm{g})$} is a maximal totally null subspace $N$ of $V$, i.e.
\begin{align}\label{eq-null_structure}
 \{ 0 \} \subset N \subset N^\perp \subset V \, ,
\end{align}
where $\dim N = m$.
\end{defn}
There are notable differences between the even- and odd-dimensional cases, which we state as a lemma.
\begin{lem} \label{lem-even/odd}
 Let $N$ be a null structure. Then
\begin{enumerate}
 \item when $\epsilon = 0$, $N=N^\perp$, and $N$ is either self-dual or anti-self-dual, i.e. for any $\bm{\omega} \in \bigwedge^m N$, either $* \bm{\omega} = \bm{\omega}$ or $* \bm{\omega} = - \bm{\omega}$;
 \item when $\epsilon = 1$, $N$ is a proper subspace of $N^\perp$, and $N^\perp / N$ is one-dimensional.
\end{enumerate}
\end{lem}
In what follows, we describe the Lie algebra of the stabiliser of the null structure.

\subsection{Graded Lie algebras and parabolic subalgebras}
\paragraph{Even dimensions greater than four ($\epsilon=0$, $m>2$)}
By Lemma \ref{lem-even/odd}, we can rewrite filtration \eqref{eq-null_structure} in the form
\begin{align}\label{eq-null_structure_even}
 V^{\tfrac{3}{2}} \subset V^{\tfrac{1}{2}} \subset V^{-\tfrac{1}{2}} \, .
\end{align}
By convention, one take $V^{\tfrac{k}{2}} = \{ 0 \}$ for all $k \geq 3$, and $V^{\tfrac{k}{2}} = V$ for all $k \leq - 1$. The meaning of this notation will become apparent in a moment. For definiteness, we also assume, with no loss of generality, that $N$ is self-dual.
Choose a subspace $V_{-\tfrac{1}{2}} \subset V^{-\tfrac{1}{2}}$ complementary to $V^{\tfrac{1}{2}}$, so that setting $V_{\tfrac{1}{2}} := V^{\tfrac{1}{2}}$, the vector space $V$ can be expressed as the direct sum
\begin{align} \label{eq-null_grading_even}
 V & = V_{\tfrac{1}{2}} \oplus V_{-\tfrac{1}{2}} \, .
\end{align}
We can then adopt the following arrangement of basis for $V$ and of symmetric bilinear form $\bm{g}$
\begin{align*}
V_{\tfrac{1}{2}} & = \left\{ \begin{pmatrix}
  \mathbf{u} \\
  \mathbf{0} 
 \end{pmatrix}
: \mathbf{u} \in \C^m \right\} \, , & 
V_{-\tfrac{1}{2}} & = \left\{ \begin{pmatrix}
  \mathbf{0} \\
  \mathbf{w} \\
 \end{pmatrix}
: \mathbf{w} \in \C^m \right\} \, , &  
 \bm{g} & =
\begin{pmatrix}
 \mathbf{0} & \mathbf{1} \\
 \mathbf{1} & \mathbf{0} 
\end{pmatrix} \, .
\end{align*}
Assuming $m>2$, the Lie algebra $\g=\so(2m,\C)$ can now be expressed as the graded Lie algebra
\begin{align} \label{eq-graded_algebra_even}
 \g & = \g_{-1} \oplus \g_0 \oplus \g_1 \, ,
\end{align}
where
\begin{align*}
\g_{-1} & = \left\{
 \begin{pmatrix}
  \mathbf{0} &  \mathbf{0} \\
  \mathbf{Z} &  \mathbf{0} \\
 \end{pmatrix} : \mathbf{Z} \in \Mat(m, \C ) \, , \, \mathbf{Z} = - \mathbf{Z}^t \right\} \, , &
\g_1 & = \left\{ \begin{pmatrix}
  \mathbf{0} &  \mathbf{Y} \\
  \mathbf{0} &  \mathbf{0} \\
 \end{pmatrix} : \mathbf{Y} \in \Mat(m, \C ) \, , \, \mathbf{Y} = - \mathbf{Y}^t \right\} \, , \\
\g_0 & = \left\{ \begin{pmatrix}
  \mathbf{X} &  \mathbf{0} \\
  \mathbf{0} &  -\mathbf{X}^t \\
 \end{pmatrix} : \mathbf{X} \in \glie(m,\C) \right\}
\, .
\end{align*}
Here, $\glie(m,\C)$ denotes the Lie algebra of the complex general linear group $\GL(m,\C)$, $\Mat(m, \C )$ the ring of all $m \times m$ matrices over $\C$, and $\cdot^t$ matrix transposition. The Lie bracket is compatible with the grading of $\g$, i.e. $[\g_i , \g_j ] \subset \g_{i+j}$ for all $i,j$, with the convention that $\g_i = \{ 0 \}$ for all $i > |1|$. Further, being a reductive Lie algebra, $\g_0$ decomposes as $\g_0 = \g_0^{ss} \oplus \mfz(\g_0)$, where $\g_0^{ss}$ is semi-simple and isomorphic to $\slie (m,\C)$, and $\mfz(\g_0)$ is the centre of $\g_0$ and is one-dimensional. In particular,  $\mfz(\g_0)$ contains the element
\begin{align*}
 \mathbf{E} & := \frac{1}{2}
\begin{pmatrix}
 \mathbf{1}_m & \mathbf{0} \\
 \mathbf{0} & - \mathbf{1}_m 
\end{pmatrix} \, ,
\end{align*}
and we see that the adjoint action of $\mathbf{E}$ on $\g$ is given by $\Ad (\mathbf{E} ) (\mathbf{X}) = i \mathbf{X}$ for all $\mathbf{X} \in \g_i$, and any $i \in \{ -1, 0 , 1 \}$. For this reason, $\mathbf{E}$ is referred to as the \emph{grading element of $\g$}. The grading on $\g$ induces a filtration $\g^1 \subset \g^0 \subset \g^{-1} = \g$ on $\g$, where $\g^i := \g_i \oplus \ldots \oplus \g_1$. Setting $\prb := \g^0$, we see that $\prb$ preserves the filtered Lie algebra $(\g, \{ \g^i \} )$. The Lie algebra $\prb$ is an example of a \emph{parabolic Lie subalgebra} of $\so(2m,\C)$. The above description is also referred to as a \emph{standard} parabolic Lie subalgebra, and any parabolic Lie subalgebra preserving a self-dual null structure must be $\SO(2m,\C)$-conjugate to it.\footnote{These definitions are usually given in terms of the root system of a semi-simple Lie algebra. This is not needed for the purpose of the present article, and we refer the reader to \cites{Baston1989,vCap2009} for a more thorough treatment.}

It is now apparent that our choice of notation for the filtration \eqref{eq-null_structure} is justified by the fact that $\mathbf{E}$ also induces the grading \eqref{eq-null_grading_even} on $V$, since for any element $\mathbf{v} \in V_i$, one has $\mathbf{E} \cdot \mathbf{v} = i \mathbf{v}$. In particular, the filtration \eqref{eq-null_structure} is $\prb$-invariant.

On the other hand the grading \eqref{eq-null_grading_even} is only invariant under $\g_0$, not $\prb$. We can nonetheless define the \emph{associated graded vector space $\gr(V)$ to $(V , \{ V^i \})$} by $\gr(V) := \bigoplus_i ( \gr_i (V) )$ where  $\gr_i (V) := V^i / V^{i+1}$, which is clearly $\prb$-invariant. Restricting the natural projections $\pi_i: V^i \rightarrow \gr_i (V)$ to $V_i$, one then obtains isomorphisms $V_i \cong \gr_i (V)$, and thus an isomorphism $\gr (V) \cong V$.

\begin{rem}\label{rem-ASD_parabolic}
 The stabiliser of an \emph{anti-self-dual} null structure can also be described in terms of a standard parabolic Lie subalgebra of $\so(2m,\C)$. It is however not $\SO(2m,\C)$-conjugate to the parabolic Lie subalgebra preserving a self-dual null structure as given above. Nonetheless, they enjoy the same properties, and the distinction between these two Lie algebras will not be crucial to the applications covered in this paper -- the notation in the anti-self-dual case mirrors that introduced above for the self-dual case. The situation in four dimensions is slightly different as we shall see presently.
\end{rem}

\paragraph{Four dimensions ($\epsilon=0$, $m=2$)}
The Lie algebra $\so(4,\C)$ can also be described in terms of the grading \eqref{eq-graded_algebra_even}. However, unlike $\so(2m,\C)$ for $m >2$, $\so(4,\C)$ is not simple, but splits into a self-dual part and an anti-dual part, each isomorphic to $\slie(2,\C)$, and which we shall denote by ${}^+ \g$ and ${}^- \g$ respectively\footnote{Here, self-duality is defined via the standard identification $\so(2m,\C) \cong \bigwedge^2 V$.}. The stabiliser of a self-dual, respectively, anti-self-dual null structure will then be a parabolic subalgebra of ${}^+\g$, respectively, ${}^-\g$. Assuming that $N$ is self-dual as above, and using the setting of the previous section, the Lie algebras ${}^+\g$ and ${}^-\g$ are given by ${}^+ \g = {}^+ \g_{-1} \oplus {}^+ \g_0 \oplus {}^+ \g_1$ and ${}^- \g = \g_0^{ss}$ respectively. Here, we have set ${}^+ \g_1 := \g_1$, ${}^+ \g_{-1} := \g_{-1}$ and ${}^+ \g_0 := \mfz(\g_0)$. Setting ${}^+ \g^i := {}^+ \g_i \oplus \ldots \oplus {}^+ \g_1$ for each $i$, we obtain the induced filtration ${}^+ \g^1 \subset {}^+ \g^0 \subset {}^+ \g^{-1}$. Then, we see that both $({}^+ \g, \{ {}^+ \g^i \} )$ and the filtration \eqref{eq-null_structure} are preserved by the parabolic Lie subalgebra $\prb := {}^+ \g^0$. A similar filtration can be derived on ${}^- \g$ with respect to the parabolic Lie algebra preserving an anti-self-dual null structure.

\paragraph{Odd dimensions ($\epsilon=1$)}
By Lemma \ref{lem-even/odd}, the filtration \eqref{eq-null_structure} can be rewritten in the form
\begin{align}\label{eq-null_structure_odd}
 V^2 \subset V^1 \subset V^0 \subset V^{-1} \, ,
\end{align}
and we set $V^k = \{ 0 \}$ for all $k \geq 2$, and $V^k = V$ for all $k \leq -1$ for convenience. This notation will be justified in the same way as in the even-dimensional case. As before, to describe the Lie algebra preserving this filtration, we introduce subspaces $V_i \subset V^i$ complementary to $V_{i+1}$, for $i=-1,0$, with $V_1 = V^1$, so that 
\begin{align}\label{eq-null_grading_odd}
  V & = V_1 \oplus V_0 \oplus V_{-1} \, . 
\end{align}
If one adopts the following arrangement of basis for $V$ and of symmetric bilinear form $\bm{g}$, adapted to this direct sum
\begin{align*}
V_1 & = \left\{
\begin{pmatrix}
  \mathbf{u} \\
  \mathbf{0} \\
  0 
 \end{pmatrix} : \mathbf{u} \in \C^m
\right\} \, , & 
V_0 & =
\left\{
\begin{pmatrix}
  \mathbf{0} \\
  \mathbf{0} \\
  v 
 \end{pmatrix} : v \in \C
\right\} \, , &
V_{-1} & = 
\left\{
\begin{pmatrix}
  \mathbf{0} \\
  \mathbf{w} \\
  0 
 \end{pmatrix} : \mathbf{w} \in \C^m
\right\} \, , & 
\bm{g} & =
\begin{pmatrix}
 \mathbf{0} & \mathbf{1} & \mathbf{0} \\
 \mathbf{1} & \mathbf{0} & \mathbf{0} \\
 \mathbf{0} & \mathbf{0} & 1 
\end{pmatrix} \, ,
\end{align*}
the Lie algebra $\g = \so(2m+1,\C)$ can be expressed as the graded Lie algebra
\begin{align*}
 \g & = \g_{-2} \oplus \g_{-1} \oplus \g_0 \oplus \g_1 \oplus \g_2 \, ,
\end{align*}
where
\begin{align*}
\g_{2} & =
\left\{
 \begin{pmatrix}
  \mathbf{0} &  \mathbf{Y} & \mathbf{0} \\
  \mathbf{0} &  \mathbf{0} & \mathbf{0} \\
  \mathbf{0} &  \mathbf{0} & 0 
 \end{pmatrix} : \mathbf{Y} \in \Mat(m,\C) \, , \, \mathbf{Y} = -\mathbf{Y}^t  \right\} \, ,
&
\g_{1} & =
\left\{
 \begin{pmatrix}
  \mathbf{0} &  \mathbf{0} & \mathbf{U} \\
  \mathbf{0} &  \mathbf{0} & \mathbf{0} \\
  \mathbf{0} & -\mathbf{U}^t & 0 
 \end{pmatrix} : \mathbf{U} \in \C^m \right\} \, ,
\\
\g_{0} & =
\left\{
 \begin{pmatrix}
  \mathbf{X} &  \mathbf{0} & \mathbf{0} \\
  \mathbf{0} &  -\mathbf{X}^t & \mathbf{0} \\
  \mathbf{0} & \mathbf{0} & 0 
 \end{pmatrix} : \mathbf{X} \in \glie(m,\C) \right\} \, ,
\\
\g_{-2} & =
\left\{
 \begin{pmatrix}
  \mathbf{0} &  \mathbf{0} & \mathbf{0} \\
  \mathbf{Z} &  \mathbf{0} & \mathbf{0} \\
  \mathbf{0} & \mathbf{0} & 0 
 \end{pmatrix} : \mathbf{Z} \in \Mat(m,\C) \, , \, \mathbf{Z} = -\mathbf{Z}^t  \right\} \, ,
&
\g_{-1} & =
\left\{
 \begin{pmatrix}
  \mathbf{0} &  \mathbf{0} & \mathbf{0} \\
  \mathbf{0} &  \mathbf{0} & \mathbf{V} \\
  - \mathbf{V}^t & \mathbf{0} & 0 
 \end{pmatrix} : \mathbf{V} \in \C^m\right\} \, .
\end{align*}
The Lie bracket is compatible with the grading of $\g$, with the convention that $\g_i = \{ 0 \}$ for all $i > |2|$. Again, $\g_0$ decomposes as $\g_0 = \g_0^{ss} \oplus \mfz(\g_0)$, where $\g_0^{ss}$ is semi-simple and isomorphic to $\slie (m,\C)$, and $\mfz(\g_0)$ is the centre of $\g_0$ and is one-dimensional. In particular,  $\mfz(\g_0)$ contains the grading element
\begin{align*}
 \mathbf{E} & := 
\begin{pmatrix}
 \mathbf{1}_m & \mathbf{0} & \mathbf{0} \\
 \mathbf{0} & - \mathbf{1}_m & \mathbf{0} \\
 \mathbf{0} & \mathbf{0} & 0
\end{pmatrix}
\end{align*}
of $\g$ since $\Ad (\mathbf{E} ) (\mathbf{X}) = i \mathbf{X}$ for all $\mathbf{X} \in \g_i$, and any $i$. The grading on $\g$ induces a filtration of Lie algebra $\g^2 \subset \g^1 \subset \g^0 \subset \g^{-1} \subset \g^{-2} = \g$, where $\g^i := \g_i \oplus \ldots \oplus \g_1$. Again, setting $\prb := \g^0$, we see that $\prb$ preserves the filtered Lie algebra $(\g, \{ \g^i \} )$. It is a standard parabolic Lie subalgebra of $\so(2m+1,\C)$, and the stabiliser of any null structure is $\SO(2m+1,\C)$-conjugate to it. We also note that our choice of notation for the filtration \eqref{eq-null_structure_odd} reflects the grading \eqref{eq-null_grading_odd} of $\mathbf{E}$ on $V$. It is then straightforward to show that the filtration \eqref{eq-null_structure_odd} is invariant under $\prb$.

As in the even-dimensional case, one can define \emph{associated graded vector space $\gr(V)$ to $(V , \{ V^i \})$} by $\gr(V) := \bigoplus_i ( \gr_i (V) )$ where  $\gr_i (V) := V^i / V^{i+1}$. A choice of grading on $V$ then allows one to establish an isomorphism $\gr (V) \cong V$.

\subsection{Induced filtered vector spaces}\label{sec-alg_dual_tensor}
Any filtration $\{ V^i \}$ on a vector space $V$ induces a filtration $\{ (V^*)^i \}$ on its dual $V^*$, whereby each vector subspace $(V^*)^i$ is the annihilator of $V^{1-i}$. Further, the associated graded vector space $\gr(V^*)$ is then such that $\gr_i(V^*) = (\gr_{-i}(V))^*$. Thus, the filtrations dual to filtrations \eqref{eq-null_structure_even} and \eqref{eq-null_structure_odd} are
\begin{align*}
 (V^* )^{\tfrac{3}{2}} \subset (V^* )^{\tfrac{1}{2}} \subset (V^* )^{-\tfrac{1}{2}} & = V^* \, , &
 (V^* )^2 \subset (V^* )^1 \subset (V^* )^0 \subset (V^* )^{-1} & = V^* \, ,
\end{align*}
in even and odd dimensions respectively, and $V^i \cong (V^*)^i$ for each $i$, by means of $\bm{g}$.

Similarly, given two filtered vector spaces $(V, \{ V^i \})$ and $( W, \{ W^i \} )$, one can naturally define a filtration $\{ (V \otimes W )^k \}$ on their tensor product $V \otimes W$ by setting
\begin{align*}
( V \otimes W )^k & := \bigoplus_{i+j=k} V^i \otimes W^j \, ,
\end{align*}
and the associated graded vector space $\gr( V \otimes W )$ is such that $\gr_k( V \otimes W ) = \bigoplus_{i+j=k} \gr_i (V) \otimes \gr_j ( W )$.

Another useful property of filtered vector spaces is that if $U$ is a vector subspace of a filtered vector space $(V , \{ V^i \})$, then $U$ inherits the filtration $\{ V^i \}$ of $V$ by setting $U^i := U \cap V^i$.

In all of these constructions, the filtrations and the associated graded vector spaces induced from a given $\prb$-invariant filtered vector space $(V , \{ V^i \})$ are also $\prb$-invariant. Further, the choice of a grading on $V$ compatible with its filtration will also induce gradings on the dual vector space and tensor products.

\begin{rem}\label{rem-rep}
It is often more convenient to view the filtrations \eqref{eq-null_structure_even} and \eqref{eq-null_structure_odd} as \emph{representations} of $\prb$ in even and odd dimensions respectively. Typically, one starts with a representation $V$ of $\g$, which for simplicity we may assume to be irreducible. In the case at hand, $V$ is simply the standard representation of $\g$. Then, one can obtain a filtration $\{ V^i \}$ on $V$ where each subspace $V^i$ is a $\prb$-invariant subspace of $V$. It turns out that the associated graded vector space $\gr(V)$ can be viewed as a refinement of the filtration $\{ V^i \}$ in the sense that each $\gr_i (V) := V^i / V^{i+1}$ is a completely reducible $\prb$-module, and each irreducible component can be described in terms of an irreducible $\g_0^{ss}$-module. This analysis clearly extends to dual and tensor representations.
\end{rem}

\subsection{Parabolic subgroups}
The passage from the Lie algebra $\g$ and its parabolic Lie algebra $\prb$ to their respective Lie groups $G$ and $P$ is explained in details in \cite{vCap2009}. In general, having fixed a complex Lie algebra $\g$ and a parabolic Lie subalgebra $\prb$, there will be some choice of possible Lie groups with Lie algebras $\g$ and $\prb$. For our purpose, it suffices to choose $G$ to be the connected Lie group $\SO(2m+\epsilon,\C)$, in which case there is only one possible choice for $P$ obtained by exponentiating $\prb$. It can also be described as follows. We first conveniently define a group $G_0$ with Lie algebra $\g_0$, which will be $\GL(m,\C)$ in the case at hand. Then, writing $\prb_+ := \g_1$, respectively, $\prb_+ := \g_1 \oplus \g_2$, when $\epsilon = 0$, respectively, $\epsilon = 1$, one has a diffeomorphism $G_0 \times \prb_+ \rightarrow P:(\mathbf{g}_0,\mathbf{Z}) \mapsto \mathbf{g}_0 \exp (\mathbf{Z})$. The Lie subgroup $P$ is appropriately called a \emph{parabolic subgroup} of $G$, and the Lie subgroup $G_0$ is referred to as the \emph{Levy subgroup of $P$}.

Finally, in our case, the $\prb$-invariant filtrations and associated graded vector spaces will all give rise to $P$-modules, i.e. irreducible representations of $\prb$ will exponentiate\footnote{This is not true in general: there is a condition on the coefficients of the highest weight vector of an irreducible representation of a parabolic subalgebra to be satisfied \cite{Baston1989}.} to irreducible representations of $P$.

\section{The geometry of almost null structures}\label{sec-geometry}
Throughout $\mcM$ will denote a $(2m+\epsilon)$-dimensional complex manifold $\mcM$, where $\epsilon \in \{ 0 ,1 \}$ and $m \geq 2$. We shall essentially be working in the holomorphic category. Thus, $\Tgt \mcM$ and $\Tgt^* \mcM$ will denote the holomorphic tangent bundle and the holomorphic cotangent bundle of $\mcM$ respectively. If $E \rightarrow \mcM$ is a vector bundle over $\mcM$, the sheaf of holomorphic sections of $E$ will be denoted $\Gamma (E)$.
If $E$ and $F$ are vector bundles, $E \otimes F$ will denote the tensor product of $E$ and $F$, $\bigwedge^k E$, the $k$-th exterior power of $E$, $\bigodot^k E$,  the $k$-th symmetric power of $E$.
The Lie bracket of (holomorphic) vector fields will be denoted by $[ \cdot , \cdot ]$. We shall also assume that $\mcM$ is orientable, and the Hodge operator on differential forms will be denoted by $*$. When $\epsilon=0$, its restriction to $\Gamma ( \bigwedge^m \Tgt^* \mcM )$ is an involution, i.e. $*^2 = 1$, and the $+1$- and $-1$-eigenforms of $*$ will be referred to as self-dual and anti-self-dual respectively.\footnote{This choice of eigenvalues is always possible in the complex category.}

We shall equip $\mcM$ with a holomorphic metric $\bm{g}$, i.e. a non-degenerate global holomorphic section of $\bigodot^2 \Tgt^* \mcM$, and the pair $(\mcM , \bm{g} )$ will be referred to as a \emph{complex Riemannian manifold}. Equivalently, the structure group of the frame bundle $\mcF$ over $\mcM$ is reduced to $G := \SO(2m+\epsilon,\C)$, and the tangent bundle can be constructed as the standard representation of $G$, i.e. $\Tgt \mcM := \mcF \times_{G} V$ where $V$ is the standard representation of $G$. The $k$-th tracefree symmetric power of the tangent bundle and the cotangent bundle will be denoted by $\bigodot_{\circ}^k \Tgt \mcM$ and $\bigodot_{\circ}^k \Tgt^* \mcM$ respectively.

The holomorphic tangent bundle admits a unique torsion-free connection, the (holomorphic) Levi-Civita connection, which preserves the holomorphic metric; it will be identified with its associated covariant derivative $\nabla : \Gamma ( \Tgt \mcM ) \otimes \Gamma ( \Tgt \mcM ) \rightarrow \Gamma ( \Tgt \mcM )$, and it extends to a connection on sheaves of holomorphic sections of tensor products of $\Tgt \mcM$ and $\Tgt^* \mcM$.

The (holomorphic) Riemann curvature tensor $\bm{R} : \Gamma ( \bigwedge^2 \Tgt \mcM ) \otimes \Gamma ( \Tgt \mcM ) \rightarrow \Gamma ( \Tgt \mcM )$ associated to $\nabla$ is given by
\begin{align*}
 \bm{R}_{\bm{X} \wedge\bm{Y}} \cdot \bm{Z} & := \nabla_{\bm{X}} \nabla_{\bm{Y}} \bm{Z} - \nabla_{\bm{Y}} \nabla_{\bm{X}} \bm{Z} - \nabla_{[\bm{X},\bm{Y}]} \bm{Z}  \, ,
\end{align*}
for all $\bm{X}, \bm{Y}, \bm{Z} \in \Gamma (\Tgt \mcM )$, and extends to sheaves of holomorphic sections of tensor products of $\Tgt \mcM$ and $\Tgt^* \mcM$. This induces a section of $\Gamma(\bigodot^2 (\bigwedge^2 \Tgt^* \mcM))$, also denoted $\bm{R}$, via
\begin{align*}
 \bm{R} ( \bm{X}, \bm{Y}, \bm{Z}, \bm{W} ) & = \bm{g} ( \bm{R}_{ \bm{X} \wedge \bm{Y} } \cdot \bm{Z} , \bm{W} ) \, ,
\end{align*}
for all $\bm{X}, \bm{Y}, \bm{Z}, \bm{W} \in \Gamma (\Tgt \mcM)$, which satisfies the Riemann symmetry
\begin{align*}
 \bm{R} ( \bm{X} , \bm{Y} , \bm{Z}, \bm{W} ) + \bm{R} ( \bm{Y} , \bm{Z} , \bm{X}, \bm{W} ) + \bm{R} ( \bm{Z} , \bm{X} , \bm{Y}, \bm{W} ) & = 0 \, .
\end{align*}
The Riemann tensor naturally splits as 
\begin{multline}\label{eq-R=C+P}
 \bm{R} (\bm{X}, \bm{Y} ,\bm{Z} , \bm{W} ) = \bm{C} (\bm{X}, \bm{Y} ,\bm{Z} , \bm{W} ) \\ - \bm{g} ( \bm{X}, \bm{Z} ) \bm{\Rho} ( \bm{Y} , \bm{W} ) + \bm{g} ( \bm{X}, \bm{W} ) \bm{\Rho} ( \bm{Y} , \bm{Z} ) + \bm{g} ( \bm{Y}, \bm{Z} ) \bm{\Rho} ( \bm{X} , \bm{W} ) - \bm{g} ( \bm{Y}, \bm{W} ) \bm{\Rho} ( \bm{X} , \bm{Z} ) \, ,
\end{multline}
where the Weyl tensor $\bm{C}$ is the tracefree part of $\bm{R}$, and the Rho tensor $\bm{\Rho}$ is a trace-adjusted Ricci tensor. The Cotton-York tensor is the $2$-form valued $1$-form $\bm{A}$ defined by
\begin{align} \label{eq-CY}
\bm{A} ( \bm{X} , \bm{Y}, \bm{Z} ) & := \nabla_{\bm{Y}} \bm{\Rho} ( \bm{Z}, \bm{X} ) - \nabla_{\bm{Z}} \bm{\Rho} ( \bm{Y}, \bm{X} ) 
\end{align}
for all $\bm{X}, \bm{Y}, \bm{Z} \in \Gamma ( \Tgt \mcM )$. Since $\bm{\Rho}$ is symmetric, $\bm{A}$ is in the kernel of $\wedge:\Tgt^* \mcM \otimes \bigwedge^2 \Tgt^* \mcM \rightarrow \bigwedge^3 \Tgt^* \mcM$.

Finally, we shall express the Bianchi identity in terms of the covariant derivative of the Weyl tensor and the Cotton-York tensor as
\begin{multline}\label{eq-Bianchi_identity}
 ( \nabla_{\bm{X}} \bm{C} ) ( \bm{Y}, \bm{Z} , \bm{S} , \bm{T} ) + ( \nabla_{\bm{Y}} \bm{C} ) ( \bm{Z}, \bm{X} , \bm{S} , \bm{T} ) + ( \nabla_{\bm{Z}} \bm{C} ) ( \bm{X}, \bm{Y} , \bm{S} , \bm{T} ) = \\
- \bm{g} (\bm{X}, \bm{S} ) \bm{A} ( \bm{T} , \bm{Y}, \bm{Z} )
- \bm{g} (\bm{Y}, \bm{S} ) \bm{A} ( \bm{T} , \bm{Z}, \bm{X} )
- \bm{g} (\bm{Z}, \bm{S} ) \bm{A} ( \bm{T} , \bm{X}, \bm{Y} ) \\
+ \bm{g} (\bm{X}, \bm{T} ) \bm{A} ( \bm{S} , \bm{Y}, \bm{Z} )
+ \bm{g} (\bm{Y}, \bm{T} ) \bm{A} ( \bm{S} , \bm{Z}, \bm{X} )
+ \bm{g} (\bm{Z}, \bm{T} ) \bm{A} ( \bm{S} , \bm{X}, \bm{Y} ) \, ,
\end{multline}
for all $\bm{X}, \bm{Y}, \bm{Z}, \bm{S}, \bm{T} \in \Gamma ( \Tgt \mcM )$. Taking the trace of equation \eqref{eq-Bianchi_identity} yields the contracted Bianchi identity, from which one can deduce that the Cotton-York is the divergence of the Weyl tensor, and thus must be tracefree.

It will be convenient to view the Weyl tensor and the Cotton-York tensor as sections of the bundles
\begin{align}\label{eq-vector_bundle_Weyl_CY}
 \mcC & := {\bigodot}_\circ ^2 ( {\bigwedge}^2 \Tgt^* \mcM ) \, , & \mcA & := \Tgt^* \mcM \odot_\circ  {\bigwedge}^2 \Tgt^* \mcM \, ,
\end{align}
where $\odot_\circ$ should be understood as reflecting the symmetry properties of the Weyl tensor and Cotton-York tensor. For this reason we may refer to $\mcC$ and $\mcA$ as the bundles of tensors with Weyl symmetries and Cotton-York symmetries respectively. When $2m + \epsilon \geq 5$, these bundles are irreducible $G$-modules, i.e. $\mcC = \mcF \times_G C$ and $\mcA = \mcF \times_G A$, where $C$ and $A$ are irreducible $G$-modules. When $m=2$, $\epsilon = 0$, under $\so(4,\C) \cong \slie(2,\C) \times \slie(2,\C)$, the bundle of $2$-forms splits into a self-dual part and an anti-self-dual part, and accordingly the bundles $\mcC$ and $\mcA$ split into self-dual parts ${}^+\mcC$ and ${}^+\mcA$, respectively, and anti-self-dual parts ${}^-\mcC$ and ${}^-\mcA$, respectively. 

\subsection{Almost null structures and classifications of the Weyl tensor and Cotton-York tensor}
This section is a translation of the algebraic setup of section \ref{sec-algebra} into the language of vector bundles. More detailed background information can be found in \cites{vCap2009} although their approach focuses essentially on Cartan geometries. 

\begin{defn} An \emph{almost null structure} on $(\mcM, \bm{g})$ is a holomorphic distribution $\mcN$ of maximal totally null planes on $\mcM$, i.e. a holomorphic subbundle of $\Tgt \mcM$ such that at every point $p$ of $\mcM$, the fiber $\mcN_p$ is a maximal totally null subspace of the tangent space $\Tgt_p \mcM$ of $\mcM$ at $p$, and $\mcN_p$ is spanned by holomorphic vector fields in a neighbourhood of $p$.

We say that the almost null structure $\mcN$ is \emph{integrable} in an open subset $\mcU$ of $\mcM$ if the distribution $\mcN$, and in odd dimensions, its orthogonal complement $\mcN^\perp$ are integrable in $\mcU$, i.e. at every point $p \in \mcU$, the fibers $\mcN_p$, and in odd dimensions, $\mcN^\perp_p$ are tangent to leaves of foliations of dimensions $m$ and $m+1$ respectively. An integrable almost null structure will be referred to as a \emph{null structure}.
\end{defn}

From the above definition, we shall essentially regard an almost null structure as a filtration of holomorphic vector subbundles
\begin{align}\label{eq-null_bundle_filtration}
 \mcM \subset \mcN \subset \mcN^\perp \subset \Tgt \mcM \, ,
\end{align}
where $\mcM$ should be regarded as the zero vector bundle. The structure group of the frame bundle $\mcF \rightarrow \mcM$ is then reduced to $P$, the parabolic subgroup preserving the filtration \eqref{eq-null_bundle_filtration}, as described in section \ref{sec-alg_dual_tensor}. One can in fact think of the almost null structure as being modeled on the filtration of vector spaces \eqref{eq-null_structure}. For this reason, we can apply the notation of section \ref{sec-algebra} to vector bundles. In particular, using the constructions of section \ref{sec-alg_dual_tensor}, this time in terms of vector bundles, we will give $P$-invariant classifications of the Weyl tensor and the Cotton-York tensor, generalising the four-dimensional Petrov classification.

Before delving into this, we restate Lemma \ref{lem-even/odd} in the vector bundle format.
\begin{lem} \label{lem-bundle_even/odd}
 Let $\mcN$ be an almost null structure on $(\mcM, \bm{g})$. Then
\begin{enumerate}
 \item when $\epsilon = 0$, $\mcN=\mcN^\perp$, and $\mcN$ is either self-dual or anti-self-dual;
 \item when $\epsilon = 1$, $\mcN$ is a proper subbundle of $\mcN^\perp$, and $\mcN^\perp / \mcN$ is a rank-one vector bundle.
\end{enumerate}
\end{lem}

\begin{rem}
 As already noted in Remark \ref{rem-ASD_parabolic}, whether an almost null structure is self-dual or anti-self-dual will not be of major significance except in four dimensions, and for the remainder of the article, we shall in general make no assumption regarding the self- or anti-self-duality of the almost null structure.
\end{rem}

\paragraph{Even dimensions greater than four ($\epsilon=0$, $m>2$)}
By Lemma \ref{lem-bundle_even/odd}, $\mcN = \mcN^\perp$, and one can rewrite the filtration \eqref{eq-null_bundle_filtration} as
\begin{align}\label{eq-TM_filtration_even}
 \mcV^{\tfrac{3}{2}} \subset \mcV^{\tfrac{1}{2}} \subset \mcV^{-\tfrac{1}{2}} \, ,
\end{align}
and we set $\mcV^{\tfrac{k}{2}} = \mcM$ for $k \geq 3$, and $\mcV^{\tfrac{k}{2}} = \Tgt \mcM$ for $k \leq - 1$ for convenience. The associated graded vector bundle is $\gr (\Tgt \mcM ) = \gr_{\tfrac{1}{2}} (\Tgt \mcM ) \oplus \gr_{-\tfrac{1}{2}} (\Tgt \mcM )$ where $\gr_{i} (\Tgt \mcM ) := \mcV^i / \mcV^{i+1}$. One can assign a grading on $\Tgt \mcM$ adapted to $\mcN$,
\begin{align}\label{eq-TM_splitting_even}
 \Tgt \mcM & = \mcV_{\tfrac{1}{2}} \oplus \mcV_{-\tfrac{1}{2}} \, ,
\end{align}
by choosing a vector subbundle $\mcV_{-\tfrac{1}{2}} \subset \mcV^{-\tfrac{1}{2}}$ complementary to $\mcV_{\tfrac{1}{2}} := \mcV^{\tfrac{1}{2}}$. This can be viewed as making a choice of frame adapted to the almost null structure. The natural projection $\mcV^i \rightarrow \gr_i ( \Tgt \mcM )$ establishes isomorphisms $\mcV_i \cong \gr_i ( \Tgt \mcM )$, and thus $\Tgt \mcM \cong \gr (\Tgt \mcM)$.

It is now a simple matter to apply the discussion of section \ref{sec-alg_dual_tensor} in the context of the filtration \eqref{eq-TM_filtration_even}, based on the remark that the bundles $\mcC$ and $\mcA$ defined in \eqref{eq-vector_bundle_Weyl_CY} are subbundles of $\bigotimes^4 \Tgt^* \mcM$ and $\bigotimes^3 \Tgt^* \mcM$ respectively. Thus,  when $m>2$, they admit the respective filtrations,
\begin{align}
 \mcM = \mcC^3 \subset \mcC^2 \subset \mcC^1 \subset \mcC^0 \subset \mcC^{-1} \subset \mcC^{-2} & = \mcC \, , \label{eq-Weyl_filtration_even} \\
 \mcM = \mcA^{\tfrac{5}{2}} \subset \mcA^{\tfrac{3}{2}} \subset \mcA^{\tfrac{1}{2}} \subset \mcA^{-\tfrac{1}{2}} \subset \mcA^{-\tfrac{3}{2}} & = \mcA \, , \label{eq-CY_filtration_even}
\end{align}
with respective associated graded vector bundles
\begin{align} 
 \gr (\mcC) & = \gr_2 (\mcC) \oplus \gr_1 (\mcC) \oplus \gr_0 (\mcC) \oplus \gr_{-1} ( \mcC ) \oplus \gr_{-2} ( \mcC )\, , \label{eq-Weyl_gr_even} \\
 \gr ( \mcA ) & := \gr_{\tfrac{3}{2}} ( \mcA ) \oplus \gr_{\tfrac{1}{2}} ( \mcA ) \oplus \gr_{-\tfrac{1}{2}} ( \mcA ) \oplus \gr_{-\tfrac{3}{2}} ( \mcA ) \, , \label{eq-CY_gr_even}
\end{align}
where $\gr_i (\mcC) := \mcC^i / \mcC^{i+1}$, $\gr_j (\mcA) := \mcA^j / \mcA^{j+1}$ for each $i,j$. A choice of frame adapted to $\mcN$ induces gradings on $\mcC$ and $\mcA$,
\begin{align}
 \mcC & = \mcC_2 \oplus \mcC_1 \oplus \mcC_0 \oplus \mcC_{-1} \oplus \mcC_{-2} \, , \label{eq-Weyl_grading_even} \\
\mcA & = \mcA_{\tfrac{3}{2}} \oplus \mcA_{\tfrac{1}{2}} \oplus \mcA_{-\tfrac{1}{2}} \oplus \mcA_{-\tfrac{3}{2}} \, , \label{eq-CY_grading_even}
\end{align}
respectively. With this choice, the natural projections $\mcC^i \rightarrow \mcC^i / \mcC^{i+1}$ and $\mcA^i \rightarrow \mcA^i / \mcA^{i+1}$ establish isomorphisms $\mcC_i \cong \gr_i (\mcC )$ and $\mcA_i \cong \gr_i (\mcA )$ for each $i$, and thus $\mcC \cong \gr (\mcC )$ and $\mcA \cong \gr (\mcA )$.

\paragraph{Four dimensions ($\epsilon=0$, $m=2$)}
In four dimensions, and assuming the almost null structure to be self-dual, one obtains filtrations on ${}^+\mcC$ and ${}^+\mcA$
\begin{align} 
 \mcM = {}^+ \mcC^3 \subset {}^+ \mcC^2 \subset {}^+ \mcC^1 \subset {}^+ \mcC^0 \subset {}^+ \mcC^{-1} \subset {}^+ \mcC^{-2} & = {}^+ \mcC \, , \label{eq-Weyl_filtration_4} \\
\mcM = {}^+ \mcA^{\tfrac{5}{2}} \subset {}^+ \mcA^{\tfrac{3}{2}} \subset {}^+ \mcA^{\tfrac{1}{2}} \subset {}^+ \mcA^{-\tfrac{1}{2}} \subset {}^+ \mcA^{-\tfrac{3}{2}} & = {}^+ \mcA \, , \label{eq-CY_filtration_4}
\end{align}
respectively.\footnote{One also gets a filtration $\mcM = {}^- \mcA^{\tfrac{3}{2}} \subset {}^- \mcA^{\tfrac{1}{2}} \subset {}^- \mcA^{-\tfrac{1}{2}} = {}^- \mcA$ on ${}^- \mcA$, which we shall not need however.}
As in the higher dimensions, one also defines associated graded vector bundles $\gr ( {}^+ \mcC )$ and $\gr ( {}^+ \mcA )$, which, on choosing a particular grading \eqref{eq-TM_splitting_even}, become isomorphic to ${}^+ \mcC$ and ${}^+ \mcA$ respectively. Similar results can be obtained on $\mcC^-$ and $\mcA^-$, when the almost null structure is taken to be anti-self-dual.

\begin{rem}\label{rem-Petrov_types}
 In four dimensions, it is well-known that at every point $p$ of $\mcM$, one can always find a maximal totally null subspace $\mcN_p$ of $\Tgt_p \mcM$ such that the self-dual part of the Weyl tensor at that point degenerates to ${}^+ \mcC^{-1}$, and this maximal totally null subspace can be extended to an almost null structure in a neighbourhood of $p$. For this reason, if the self-dual part of the Weyl tensor degenerates further to a section of ${}^+ \mcC^0$, it is referred\footnote{This is usually formulated in terms of a spinor field $\bm{\xi}$, say, which defines $\mcN$, and the terminology `with respect to $\mcN$' is then replaced by `along $\bm{\xi}$'.} to as \emph{algebraically special with respect to $\mcN$}. In fact, the (complex self-dual) Petrov types I, II, III and N can easily be defined in terms of the bundles ${}^+ \mcC^{-1}$, ${}^+ \mcC^0$, ${}^+ \mcC^1$ and ${}^+ \mcC^2$ respectively, and similarly for the anti-self-dual case.
\end{rem}

\paragraph{Odd dimensions ($\epsilon=1$)}
This is very similar to the previous case except that now, $\mcN$ is a proper holomorphic subbundle of $\mcN^\perp$. Thus the filtration \eqref{eq-null_bundle_filtration} can be rewritten as
\begin{align} \label{eq-TM_filtration_odd}
 \mcV^2 \subset \mcV^1 \subset \mcV^0 \subset \mcV^{-1} \, ,
\end{align}
and we set $\mcV^{k} = \mcM$ for $k \geq 2$, and $\mcV^{k} = \Tgt \mcM$ for $k \leq - 1$ for convenience. The associated graded vector bundle is $\gr (\Tgt \mcM ) = \gr_{1} (\Tgt \mcM ) \oplus \gr_{0} (\Tgt \mcM ) \oplus \gr_{-1} (\Tgt \mcM )$ where $\gr_{i} (\Tgt \mcM ) := \mcV^i / \mcV^{i+1}$. One can assign a grading on $\Tgt \mcM$ adapted to $\mcN$,
\begin{align} \label{eq-TM_splitting_odd}
 \Tgt \mcM & = \mcV_{1} \oplus \mcV_0 \oplus \mcV_{-1} \, , 
\end{align}
by choosing vector subbundles $\mcV_{i} \subset \mcV^{i}$ complementary to $\mcV_{i+1}$ with $\mcV_{1} : = \mcV^{1}$. This can be viewed as making a choice of frame adapted to the almost null structure. The natural projection $\mcV^i \rightarrow \gr_i ( \Tgt \mcM )$ establishes isomorphisms $\mcV_i \cong \gr_i ( \Tgt \mcM )$, and thus $\Tgt \mcM \cong \gr (\Tgt \mcM)$.

Again, from section \ref{sec-alg_dual_tensor}, the filtration \eqref{eq-TM_filtration_even} induces filtrations on the vector bundles $\mcC$ and $\mcA$,
\begin{align}
 \mcM = \mcC^5 \subset \mcC^4 \subset \mcC^3 \subset \ldots \subset \mcC^{-3} \subset \mcC^{-4} & = \mcC \, , \label{eq-Weyl_filtration_odd} \\
 \mcM = \mcA^4 \subset \mcA^3 \subset \mcA^2 \subset \ldots \subset \mcA^{-2} \subset \mcA^{-3} & = \mcA \, , \label{eq-CY_filtration_odd}
\end{align}
respectively,
with associated graded vector bundles
\begin{align} 
 \gr (\mcC) & = \gr_4 (\mcC) \oplus \gr_3 (\mcC) \ldots \oplus \gr_{-3} ( \mcC ) \oplus \gr_{-4} ( \mcC )\, , \label{eq-Weyl_gr_odd} \\
\gr ( \mcA ) & = \gr_3 ( \mcA ) \oplus \gr_2 ( \mcA ) \oplus \ldots \oplus \gr_{-2} ( \mcA ) \oplus \gr_{-3} ( \mcA ) \, , \label{eq-CY_gr_odd}
\end{align}
respectively, where $\gr_i (\mcC) := \mcC^i / \mcC^{i+1}$, $\gr_j (\mcA) := \mcA^j / \mcA^{j+1}$ for each $i,j$. A choice of frame adapted to $\mcN$ induces gradings on $\mcC$ and $\mcA$,
\begin{align}
 \mcC & = \mcC_4 \oplus \mcC_3 \oplus \ldots \oplus \mcC_{-3} \oplus \mcC_{-4} \, , \label{eq-Weyl_grading_odd} \\
 \mcA & = \mcA_3 \oplus \mcA_2 \oplus \ldots \oplus \mcA_{-2} \oplus \mcA_{-3} \, , \label{eq-CY_grading_odd}
\end{align}
respectively, which allow one to establish isomorphisms $\mcC_i \cong \gr_i (\mcC )$ and $\mcA_i \cong \gr_i (\mcA )$ for each $i$, and thus $\mcC \cong \gr (\mcC )$ and $\mcA \cong \gr (\mcA )$.

\begin{rem}
  In even and odd dimensions greater than four, the (pointwise) existence of an almost null structure with respect to which the Weyl tensor degenerates to a section of $\mcC^{-1}$ and $\mcC^{-3}$ respectively, is not guarranteed in general. While the use of the terms `algebraically special' to describe a Weyl tensor degenerating to a section of $\mcC^0$ may then not be entirely appropriate, such a Weyl tensor nonetheless enjoys some `special' status regarding the geometric property of the almost null structure $\mcN$ as will be seen in section \ref{sec-GS}.
\end{rem}

\begin{rem}\label{rem-refine}
  Referring back to Remark \ref{rem-rep}, we can view each of the vector bundles $\mcC^i$ and $\mcA^i$ as $\prb$-modules (or $P$-modules at the Lie group level), and one way to refine the classification is by considering the irreducible $\prb$-modules in each of the quotient bundles $\gr_i ( \mcC )$ and $\gr_i ( \mcA )$. This will not be needed in this paper, but will be covered in a future publication.
\end{rem}

\paragraph{Tensorial characterisation of sections of $\mcC$ and $\mcA$}
When it comes to explicit computations, it is somewhat more convenient to describe sections of the bundles $\mcC^i$ and $\mcA^i$ by means of the following lemma.
\begin{lem}\label{lem-type_characterisation}
Fix $k \in \Z$, $k > - \frac{4}{2-\epsilon}$, $\ell \in \Z + \tfrac{1-\epsilon}{2}$, $\ell > - \tfrac{3}{2-\epsilon}$. When $m=2$, $\epsilon=0$, assume that $\mcN$ is self-dual, and write $\mcC$ and $\mcA$ for ${}^+ \mcC$ and ${}^+ \mcA$ respectively. Let $\bm{C} \in \Gamma( \mcC )$ and $\bm{A} \in \Gamma( \mcA )$. Then,
\begin{align*}
 \bm{C} \in \Gamma ( \mcC^k ) & \Leftrightarrow & \bm{C} ( \bm{X}_{i_1} , \bm{X}_{i_2} , \bm{X}_{i_3}, \bm{X}_{i_4} ) & = 0 \, , & \mbox{for all $\bm{X}_{i_j} \in \Gamma( \mcV^{i_j} )$ such that $\sum_j i_j = 1-k$,} \\
 \bm{A} \in \Gamma ( \mcA^\ell ) & \Leftrightarrow  &  \bm{A} ( \bm{X}_{i_1} , \bm{X}_{i_2} , \bm{X}_{i_3} ) & = 0 \, , & \mbox{for all $\bm{X}_{i_j} \in \Gamma( \mcV^{i_j} )$ such that $\sum_j i_j = 1 - \ell$,}
\end{align*}
where $i_j \in \Z$, $|i_j| \leq \frac{1+\epsilon}{2}$ for all $j=1, \ldots ,4$.
\end{lem}
The above characterisation can be proved immediately from the fact that the bundles $\mcC$ and $\mcA$ are subbundles of $\bigotimes^4 \Tgt^* \mcM$ and $\bigotimes^3 \Tgt^* \mcM$ respectively, and $(\mcV^*)^i$ is the annihilator of $\mcV^{1-i}$ for each $i$.

\subsection{Geometric properties}
Let $(\mcM, \bm{g})$ be a $(2m+\epsilon)$-dimensional complex Riemannian manifold, where $\epsilon \in \{ 0 , 1 \}$ and $m \geq 2$, endowed with an almost null structure $\mcN$. In the following discussion, we shall generally treat both the even- and odd-dimensional cases at once, bearing in mind that in the former case $\mcN^\perp = \mcN$, so that there will be some redundancy in the properties presented. When some distinction needs to be made, the notation of section \ref{sec-geometry} together with $\epsilon$ will be used.

\begin{thm}[Frobenius]
 A necessary and sufficient condition for an almost null structure $\mcN$ to be integrable is that it is involutive, i.e.
\begin{align*}
 [ \bm{X} , \bm{Y} ] & \in \Gamma (\mcN) \, , & [ \bm{S} , \bm{T} ] & \in \Gamma (\mcN^\perp) \, , 
\end{align*}
or equivalently,
\begin{align} \label{eq-integrability_bracket_metric}
 \bm{g} (  \bm{S} , [ \bm{X} , \bm{Y} ] ) & = 0 \, , & \bm{g} ( \bm{X} , [\bm{S} , \bm{T} ] ) & = 0 \, ,  
\end{align}
for all $\bm{X}, \bm{Y} \in \Gamma (\mcN)$ and $\bm{S}, \bm{T} \in \Gamma (\mcN^\perp)$.
\end{thm}

The integrability of $\mcN$ in some open subset $\mcU$ gives rise to a foliation of $\mcU$ by maximal totally null leaves. In odd dimensions, $\mcU$ is also foliated by leaves of dimension $m+1$. In both cases, these leaves are \emph{totally geodetic}, in the sense given by the next lemma.
\begin{lem} \label{lem-integrability_connection}
 The almost null structure $\mcN$ is integrable if and only if
\begin{align} \label{eq-integrability_connection}
 \bm{g} ( \bm{X} , \nabla_{\bm{Y}} \bm{Z} ) & = 0 \, , & \bm{g} ( \bm{Y} , \nabla_{\bm{X}} \bm{Z} ) & = 0 \, ,  
\end{align}
for all $\bm{X} \in \Gamma (\mcN^\perp)$, $\bm{Y}, \bm{Z} \in \Gamma (\mcN)$.
\end{lem}

\begin{proof}
Let $\bm{X} \in \Gamma (\mcN^\perp)$, $\bm{Y}, \bm{Z} \in \Gamma (\mcN)$, and suppose that the distributions $\mcN$ and $\mcN^\perp$ are integrable. Then, using the defining properties of the Levi-Civita connection,
\begin{align*}
 \bm{g} ( \bm{X} , \nabla_{\bm{Y}} \bm{Z} )
& =\frac{1}{2} \left( \bm{g} ( \bm{X} , \nabla_{\bm{Y}} \bm{Z} ) - \bm{g} ( \bm{X} ,\nabla_{\bm{Z}} \bm{Y} ) - \bm{g} ( \bm{Z} , \nabla_{\bm{Y}} \bm{X} ) + \bm{g} ( \bm{Z} , \nabla_{\bm{X}} \bm{Y} ) - \bm{g} ( \bm{Y} , \nabla_{\bm{X}} \bm{Z} ) + \bm{g} ( \bm{Z} , \nabla_{\bm{Z}} \bm{X} ) \right) \\
& =\frac{1}{2} \left( \bm{g} ( \bm{X} , [ \bm{Y} , \bm{Z} ] ) + \bm{g} ( \bm{Z} , [ \bm{X} , \bm{Y} ] ) + \bm{g} ( \bm{Y} , [ \bm{Z} , \bm{X} ] ) \right) = 0 \, ,
\end{align*}
by equations \eqref{eq-integrability_bracket_metric}, and similarly for $\bm{g} ( \bm{Y} , \nabla_{\bm{X}} \bm{Z} )$. The converse is obvious.
\end{proof}

\begin{rem}
There is an alternative way of characterising the integrability of the almost null structure, which mirrors a procedure introduced in \cites{Gray1980,Falcitelli1994} in almost Hermitian geometry. We note that the almost null structure $\mcN$ can be represented\footnote{This can also be formulated spinorially.} by a single, up to scale, tensorial object $\bm{\omega} \in \Gamma (\bigwedge^m \mcN^*)$. It thus makes sense to measure the failure of the Levi-Civita connection to preserve $\bm{\omega}$, or equivalently, to be a $\prb$-valued $1$-form on $\mcM$. In even dimensions, the geometric properties of $\mcN$ can then be encoded by the $P$-invariant differential equations
\begin{align}\label{eq-null_class_even}
\nabla_{\bm{X}} \bm{\omega} & = \bm{\alpha} (\bm{X}) \bm{\omega} \, ,
\end{align}
for some $1$-form $\bm{\alpha}$, and for all $\bm{X} \in \Gamma ( \mcV^i )$ for some  $i \in \{ - \tfrac{1}{2} , \tfrac{1}{2} \}$. In particular, taking $i=\tfrac{1}{2}$ gives the integrability of $\mcN$.
In odd dimensions, the geometric properties of $\mcN$ can be encoded by the $P$-invariant differential equations
\begin{align}\label{eq-null_class_odd}
\nabla_{\bm{X}} \bm{\omega} & = \bm{\alpha} (\bm{X}) \bm{\omega} \, , & \nabla_{\bm{Y}} (* \bm{\omega} )& = \bm{\beta} (\bm{Y}) \bm{\omega} \wedge \bm{\gamma} \, ,
\end{align}
for some $1$-forms $\bm{\alpha}$, $\bm{\beta}$ and $\bm{\gamma}$, and for all $\bm{X} \in \Gamma ( \mcV^i )$, $\bm{Y} \in \Gamma ( \mcV^j )$ for some $i, j \in \{ -1, 0 , 1 \}$. In this case, the integrability of $\mcN$ (and $\mcN^\perp$) is given by taking $i=1$ and $j=0$.
\end{rem}

\paragraph{Integrability condition}
The existence of a null structure $\mcN$ on $\mcM$ is subject to an integrability condition on the Weyl tensor as given by the next proposition.
\begin{prop}\label{prop-integrability_condition}
 Suppose $\mcN$ is a null structure. Then, in dimensions greater than four, the Weyl tensor is a section of $\mcC^{-1-\epsilon}$. In four dimensions, assuming $\mcN$ is self-dual, the self-dual part of the Weyl tensor is a section of ${}^+\mcC^{-1}$.
\end{prop}

\begin{proof}
 Let $\bm{X} \in \Gamma (\mcN^\perp)$, $\bm{Y} , \bm{Z}, \bm{W} \in \Gamma (\mcN)$. We shall show that
\begin{align} \label{eq-integrability_condition}
 \bm{C} ( \bm{W} , \bm{X} , \bm{Y}, \bm{Z} ) & = 0 \, ,
\end{align}
which, by Lemma \ref{lem-type_characterisation}, is equivalent to the claim of the proposition.
We start by differentiating either of relations \eqref{eq-integrability_connection}, so for definiteness, we have
\begin{align} \label{eq-integrability_intermediate}
 0 & = \nabla_{\bm{W}} ( \bm{g} (\bm{Y}, \nabla_{\bm{X}} \bm{Z} ) ) = \bm{g} ( \nabla_{\bm{W}} \bm{Y} , \nabla_{\bm{X}} \bm{Z} ) + \bm{g} ( \bm{Y}, \nabla_{\bm{W}} \nabla_{\bm{X}} \bm{Z} ) = \bm{g} ( \bm{Y}, \nabla_{\bm{W}} \nabla_{\bm{X}} \bm{Z} ) \, ,
\end{align}
by equation \eqref{eq-integrability_connection} again.
Now, from the definition of the Riemann tensor, we have
\begin{align*}
 \bm{R} ( \bm{W}, \bm{X}, \bm{Y}, \bm{Z} ) 
& = \bm{g} (\bm{Y}, \nabla_{\bm{W}} \nabla_{\bm{X}} \bm{Z} - \nabla_{\bm{X}} \nabla_{\bm{W}} \bm{Z} - \nabla_{[\bm{W},\bm{X}]} \bm{Y} ) = 0 \, ,
\end{align*}
by equations \eqref{eq-integrability_intermediate} and \eqref{eq-integrability_connection} together with Lemma \ref{lem-integrability_connection}. The splitting of the Riemann tensor \eqref{eq-R=C+P} now establishes equation \eqref{eq-integrability_condition}. Further, when $2m + \epsilon = 4$ and $\mcN$ is assumed to be self-dual, equation \eqref{eq-integrability_condition} is always trivially satisfied on restriction to the anti-self-dual part of the Weyl tensor, and so must be a condition on the self-dual part of the Weyl tensor
\end{proof}

\subsection{Null basis and its associated canonical almost null structures}\label{sec-null_basis}
So far the discussion has been expressed invariantly, with no reference to any particular frame, but at this stage, it is convenient to introduce some notation tied up to a choice of frame adapted to a null structure. As before $(\mcM, \bm{g})$ will denote a $(2m+\epsilon)$-dimensional complex Riemannian manifold where $\epsilon \in \{ 0 , 1 \}$, and $m \geq 2$, and $\mcN$ an almost null structure on $\mcM$. We first note that choosing a (local) grading \eqref{eq-TM_splitting_even}, respectively \eqref{eq-TM_splitting_odd} of the tangent bundle, when $\epsilon = 0$, respectively $\epsilon=1$, is really tantamount to choosing a frame adapted to $\mcN$. When $\epsilon = 0$, this (local) frame will be denoted
\begin{align*}
\left\{ \bm{\xi}_\mu, \tilde{\bm{\xi}}_{\tilde{\nu}} | \mu, \tilde{\nu} = 1 , \ldots, m \right\}  \, ,
\end{align*}
where $\{\bm{\xi}_\mu\}$ and $\{\tilde{\bm{\xi}}_{\tilde{\mu}}\}$ span $\mcV_{\tfrac{1}{2}}$ and $\mcV_{-\tfrac{1}{2}}$ respectively. When $\epsilon = 1$, it will be denoted
\begin{align*}
\left\{ \bm{\xi}_\mu, \tilde{\bm{\xi}}_{\tilde{\nu}}, \bm{\xi}_0 | \mu, \tilde{\nu} = 1 , \ldots, m \right\}  \, ,
\end{align*}
where $\{\bm{\xi}_\mu\}$, $\{ \bm{\xi}_0 \}$ and $\{\tilde{\bm{\xi}}_{\tilde{\mu}}\}$ span $\mcV_{1}$, $\mcV_{0}$, and $\mcV_{-1}$ respectively.
In both cases, the frame vector fields will be taken to satisfy the normalisation conditions
\begin{align*}
 \bm{g} ( \bm{\xi}_\mu , \tilde{\bm{\xi}}_{\tilde{\nu}} ) & = \delta_{\mu \tilde{\nu}} \, , & \bm{g} ( \bm{\xi}_0 , \bm{\xi}_0 ) & = 1 \, .
\end{align*}
The corresponding coframes will be denoted
\begin{align*}
 \left\{ \bm{\theta}^\mu , \tilde{\bm{\theta}}^{\tilde{\nu}} | \mu, \tilde{\nu} = 1 , \ldots, m \right\} \, , & &
 \left\{ \bm{\theta}^\mu , \tilde{\bm{\theta}}^{\tilde{\nu}} , \bm{\theta}^0 | \mu, \tilde{\nu} = 1 , \ldots, m \right\} \, ,  
\end{align*}
when $\epsilon =0$ and $\epsilon=1$ respectively, and where $\ctr{\bm{\xi}_\mu}\bm{\theta}^\nu = \delta_\mu^\nu$, $\ctr{\tilde{\bm{\xi}}_{\tilde{\mu}}}\tilde{\bm{\theta}}^{\tilde{\nu}} = \delta_{\tilde{\mu}}^{\tilde{\nu}}$, $\ctr{\bm{\xi}_0}\bm{\theta}^0 = 1$, and all other pairings vanish. In particular, the metric takes the canonical form
\begin{align}\label{eq-null_metric}
 \bm{g} & = 2 \sum_{\mu = 1}^m \bm{\theta}^\mu \odot \tilde{\bm{\theta}}^{\tilde{\mu}} + \epsilon \bm{\theta}^0 \otimes \bm{\theta}^0 \, .
\end{align}
With this convention, we shall denote the components of the tensors with respect to these frame and co-frame in the usual way, i.e. if $\bm{A}$ is a tensor field, then its components are given by, e.g.
\begin{align*}
 A \ind{^{\mu 0 \tilde{\nu} \ldots \tilde{\kappa}} _{\lambda \tilde{\rho} \ldots 0 \tau}} & := 
\bm{A} ( \bm{\theta}^{\mu} , \bm{\theta}^0, \tilde{\bm{\theta}}^{\tilde{\nu}},  \ldots , \tilde{\bm{\theta}}^{\tilde{\kappa}} , \bm{\xi}_{\lambda} , \tilde{\bm{\xi}}_{\tilde{\rho}}, \ldots, \bm{\xi}_0, \bm{\xi}_\tau ) \, ,
\end{align*}
and so on.

For future use, we introduce the following notation for the components of the connection $1$-form
\begin{align*}
 \Gamma \ind{_{\kappa \mu \nu}} & := \bm{g} ( \nabla_{\bm{\xi}_\kappa} \bm{\xi}_\mu , \bm{\xi}_\nu ) \, , &
 \Gamma \ind{_{\kappa \mu \tilde{\nu}}} & := \bm{g} ( \nabla_{\bm{\xi}_\kappa} \bm{\xi}_\mu , \tilde{\bm{\xi}}_{\tilde{\nu}} ) \, , &
 \Gamma \ind{_{\kappa \tilde{\mu} \tilde{\nu}}} & := \bm{g} ( \nabla_{\bm{\xi}_\kappa} \tilde{\bm{\xi}}_{\tilde{\mu}} , \tilde{\bm{\xi}}_{\tilde{\nu}} ) \, , \\
 \Gamma \ind{_{{0} \mu \nu}} & := \bm{g} ( \nabla_{\bm{\xi}_{0}} \bm{\xi}_\mu , \bm{\xi}_\nu ) \, , &
 \Gamma \ind{_{{0} \mu \tilde{\nu}}} & := \bm{g} ( \nabla_{\bm{\xi}_{0}} \bm{\xi}_\mu , \tilde{\bm{\xi}}_{\tilde{\nu}} ) \, , &
 \Gamma \ind{_{{0} \tilde{\mu} \tilde{\nu}}} & := \bm{g} ( \nabla_{\bm{\xi}_{0}} \tilde{\bm{\xi}}_{\tilde{\mu}} , \tilde{\bm{\xi}}_{\tilde{\nu}} ) \, ,
\end{align*}
and so on, in the obvious way. Since the Levi-Civita connection preserves the metric, these components are skew-symmetric in their last two indices.

\paragraph{Canonical almost null structures}
For convenience, let $S := \{ 1, 2, \ldots , m \}$, $M \subset S$, and $\widetilde{M} := S \setminus M$. Then, having chosen a null frame as above, for every $2^m$ choice of $M$, one can canonically define almost null structures
\begin{align*}
  \mcN_{M} & := \sspan \left\{ \bm{\xi}_{\mu} , \tilde{\bm{\xi}}_{\tilde{\nu}} : \mbox{for all $\mu \in M$, $\tilde{\nu} \in \widetilde{M}$} \right\} \, .
\end{align*}
That these are maximal totally null is clear from the form of the metric \eqref{eq-null_metric}. In particular, $\mcN = \mcN_{1 \ldots m}$. For future use, we shall denote
\begin{align*}
 \mcB_S & := \left\{ \mcN_M : \mbox{for all $M \subset S$} \right\} \, ,
\end{align*}
the set of all canonical almost null structures on (an open subset of) $\mcM$.

\begin{rem}
 In the above notation, the spinor bundle $\bigwedge^\bullet \mcN$ is locally spanned by the $2^m$ simple $m$-vectors $\bm{\xi}_M := \bm{\xi}_{\mu_1} \wedge \ldots \wedge \bm{\xi}_{\mu_p}$ where $p$ is the cardinality of $M$ -- when $M$ is empty, we write $\bm{\xi}_0$ for the unit scalar field spanning $\bigwedge^0 \mcN$. In fact, each $\bm{\xi}_M$ is a \emph{pure} spinor field, in the sense that it annihilates the corresponding canonical null distribution $\mcN_{M}$ via the Clifford action \cites{Cartan1981,Budinich1989}.
\end{rem}

\section{The Goldberg-Sachs theorem}\label{sec-GS}
We begin by restating the Goldberg-Sachs theorem as generalised by Kundt-Thompson \cite{Kundt1962} and Robinson-Schild \cite{Robinson1963}. The formulation, closely following \cite{Penrose1986}, is adapted to the language of section \ref{sec-geometry} .
\begin{thm}[Generalised Goldberg-Sachs Theorem]\label{thm-GS4}
 Let $(\mcM, \bm{g})$ be a four-dimensional complex Riemannian manifold. Let $\mcN$ be a self-dual almost null structure on $(\mcM , \bm{g})$, and $\mcU$ an open subset of $\mcM$. Consider the following statements
\begin{enumerate}
 \item the self-dual part of the Weyl tensor is a section of ${}^+ \mcC^k$ over $\mcU$ which does not degenerate to a section of ${}^+ \mcC^{k+1}$ over $\mcU$; \label{thm-GS4.1}
 \item the almost null structure $\mcN$ is integrable in $\mcU$; \label{thm-GS4.2}
 \item the self-dual part of the Cotton-York tensor is a section of ${}^+\mcA^{k-\tfrac{1}{2}}$ over $\mcU$. \label{thm-GS4.3}
\end{enumerate}
Then,
\begin{enumerate}[(a)]
 \item \label{thm-GS4a} for $k=0, 1, 2$, \eqref{thm-GS4.1} \& \eqref{thm-GS4.2} $\Rightarrow$ \eqref{thm-GS4.3};
 \item \label{thm-GS4b} for $k=0, 1, 2$, \eqref{thm-GS4.1} \& \eqref{thm-GS4.3} $\Rightarrow$ \eqref{thm-GS4.2};
 \item \label{thm-GS4c} for $k=0$, \eqref{thm-GS4.2} \& \eqref{thm-GS4.3} $\Rightarrow$ \eqref{thm-GS4.1}' $:=$ \eqref{thm-GS4.1} with $k = 0$ or $1$ or $2$. 
\end{enumerate}
\end{thm}

\begin{rem}
 An anti-self-dual version of Theorem \ref{thm-GS4} coexists with it.
\end{rem}

The proof of each of the implications \eqref{thm-GS4a}, \eqref{thm-GS4b}, \eqref{thm-GS4c} of the theorem is essentially based on the (self-dual) contracted Bianchi identity. It is usually carried out as a local computation using a local null frame adapted to $\mcN$, e.g. in the Newman-Penrose formalism, or more invariantly in terms of spinor fields. The assumption on the Cotton-York tensor can also be replaced by an assumption on the Rho tensor in implication \eqref{thm-GS4c}. We also note that implication \eqref{thm-GS4a} is really an integrability condition on the Cotton-York tensor given some algebraic condition on the Weyl tensor.

In higher dimensions, a putative Goldberg-Sachs theorem would take the same form as Theorem \ref{thm-GS4} except for the fact that self-duality has now no place there, and in odd dimensions, one has additional degeneracy classes. Let's examine each implication in turn.
\begin{itemize}
 \item Implication \eqref{thm-GS4a} presents no difficulty, and follows directly from the definition of the Cotton-York tensor, in terms of the contracted Bianchi identity, for which we give an invariant expression in terms of an almost null structure below.
 \item To prove implication \eqref{thm-GS4b} in four dimensions, we first note that each summand ${}^+\mcC^k/{}^+\mcC^{k+1}$ of the graded vector bundle associated to the filtration \eqref{eq-Weyl_filtration_4} is one-dimensional. This means that the property that the Weyl tensor is a (local) section of $\mcC^k$, but does not degenerate to a section of $\mcC^{k+1}$, depends on a single non-vanishing component of the Weyl tensor in a frame adapted to $\mcN$.

 In higher dimensions, the bundles $\mcC^k/\mcC^{k+1}$ are not one-dimensional in general, and it is no longer enough to assume that the Weyl tensor, as a section of $\mcC^k$, does not degenerate to a section of $\mcC^{k+1}$. For this reason, we must introduce a genericity assumption, which must be understood in the sense that there are no additional structures imposed on $\mcM$ beside the almost null structure. As a result, the components of the Weyl tensor, modulo Weyl symmetries, do not satisfy algebraic relations among themselves. It is also worth noting that unlike in four dimensions, the \emph{full} Bianchi identity is now required in the proof of implication \eqref{thm-GS4b}: the contracted Bianchi identity alone does not provide enough constraints on the relevant connection components.
 \item Finally, one can already assert that implication \eqref{thm-GS4c} fails in higher dimensions. Indeed, based on the computations of \cite{Taghavi-Chabert2011}, one can complexify a small region of the Lorentzian black ring solution \cite{Emparan2002}, and show that it locally admits (holomorphic) null structures.\footnote{These are the complexification of the original null structures on a Lorentzian manifold, as explained in section \ref{sec-Robinson_structure}.} However, the Weyl tensor does not degenerate to a section of $\mcC^0$. The author is aware of at least one other counterexample to implication \eqref{thm-GS4c} in higher dimensions, the complexification of the five-dimensional Euclidean black hole metric discovered in \cite{Lu2009}.
\end{itemize}

\begin{rem}\label{lem-GS_conformally_flat}
 Not covered in Theorem \ref{thm-GS} is the case when $(\mcM , \bm{g} )$ is conformally (half-)flat, i.e. the self-dual part of the Weyl tensor and the self-dual part of the Cotton-York tensor are sections of ${}^+ \mcC^3$ and ${}^+ \mcA^{\tfrac{5}{3}}$ respectively. If one is concerned in finding a null structure in this case, the appropriate alternative is to appeal to the Kerr theorem, which states that any (local) null structure on a conformally (half-)flat complex Riemannian manifold $( \mcM , \bm{g} )$ can be prescribed by a holomorphic projective variety in its twistor space \cite{Penrose1986}. Consequently, $( \mcM , \bm{g} )$ admits (locally) infinitely many self-dual null structures. The same remark applies in higher dimensions \cite{Hughston1988}.
\end{rem}

We treat the even- and odd-dimensional cases separately. Before we proceed, we give an expression for the Cotton-York tensor in terms of the contracted Bianchi identity.
\begin{lem}\label{lem-CY_alt}
Let $(\mcM , \bm{g} )$ be a $(2m+\epsilon)$-dimensional complex Riemannian manifold where $\epsilon \in \{ 0, 1 \}$ and $m \geq 2$. Then the defining equation of the Cotton-York tensor \eqref{eq-CY} is equivalent to
\begin{multline}\label{eq-CY_alt}
 ( 3-2m+\epsilon ) \bm{A} (\bm{X} , \bm{Y} , \bm{Z} ) = \\
\sum_\sigma \left( \nabla_{\tilde{\bm{\xi}}_{\tilde{\sigma}}} \bm{C} ( \bm{\xi}_\sigma , \bm{X} , \bm{Y} , \bm{Z} ) - \bm{C} ( \bm{\xi}_\sigma , \nabla_{\tilde{\bm{\xi}}_{\tilde{\sigma}}} \bm{X} , \bm{Y} , \bm{Z} ) - \bm{C} ( \bm{\xi}_\sigma , \bm{X} , \nabla_{\tilde{\bm{\xi}}_{\tilde{\sigma}}} \bm{Y} , \bm{Z} ) - \bm{C} ( \bm{\xi}_\sigma , \bm{X} , \bm{Y} , \nabla_{\tilde{\bm{\xi}}_{\tilde{\sigma}}} \bm{Z} ) \right.  \\
+ \left. \nabla_{\bm{\xi}_\sigma} \bm{C} ( \tilde{\bm{\xi}}_{\tilde{\sigma}} , \bm{X} , \bm{Y} , \bm{Z} ) - \bm{C} ( \tilde{\bm{\xi}}_{\tilde{\sigma}} , \nabla_{\bm{\xi}_\sigma} \bm{X} , \bm{Y} , \bm{Z} ) - \bm{C} ( \tilde{\bm{\xi}}_{\tilde{\sigma}} , \bm{X} , \nabla_{\bm{\xi}_\sigma} \bm{Y} , \bm{Z} ) - \bm{C} ( \tilde{\bm{\xi}}_{\tilde{\sigma}} , \bm{X} , \bm{Y} , \nabla_{\bm{\xi}_\sigma} \bm{Z} ) \right) \\
+ \epsilon \left( \nabla_{\bm{\xi}_0} \bm{C} ( \bm{\xi}_0 , \bm{X} , \bm{Y} , \bm{Z} ) - \bm{C} ( \bm{\xi}_0 , \nabla_{\bm{\xi}_0} \bm{X} , \bm{Y} , \bm{Z} ) - \bm{C} ( \bm{\xi}_0, \bm{X} , \nabla_{\bm{\xi}_0} \bm{Y} , \bm{Z} ) - \bm{C} ( \bm{\xi}_0 , \bm{X} , \bm{Y} , \nabla_{\bm{\xi}_0} \bm{Z} ) \right) \, ,
\end{multline}
for all $\bm{X}, \bm{Y}, \bm{Z} \in \Gamma ( \Tgt \mcM )$, where $\{ \bm{\xi}_\mu , \tilde{\bm{\xi}}_{\tilde{\mu}}, \epsilon \bm{\xi}_0 \}$ is a null basis as described in section \ref{sec-null_basis}.
\end{lem}

\subsection{The complex Goldberg-Sachs theorem in even dimensions}
We start with the even-dimensional generalisation of implication \eqref{thm-GS4b} of Theorem \ref{thm-GS4}, which is an application of Lemma \ref{lem-type_characterisation} to equation \eqref{eq-CY_alt} with $\epsilon=0$, together with the geodesy property \eqref{lem-integrability_connection}.
\begin{prop}\label{prop-GS}
 Let $(\mcM, \bm{g})$ be a $2m$-dimensional complex Riemannian manifold, where $m \geq 3$. Let $\mcN$ be an almost null structure on $\mcM$, and $\mcU$ an open subset of $\mcM$. Let $k \in \{0, 1, 2 \}$. Suppose that the Weyl tensor is a section of $\mcC^k$ over $\mcU$. Then the Cotton-York tensor is a section of $\mcA^{k - \tfrac{3}{2}}$ over $\mcU$. Suppose further that $\mcN$ is integrable in $\mcU$. Then the Cotton-York tensor is a section of $\mcA^{k - \tfrac{1}{2}}$ over $\mcU$.
\end{prop}

Next, the even-dimensional generalisation of implication \eqref{thm-GS4b} of Theorem \ref{thm-GS4} can be expressed as follows.
\begin{thm}\label{thm-GS}
 Let $(\mcM, \bm{g})$ be a $2m$-dimensional complex Riemannian manifold, where $m \geq 3$. Let $\mcN$ be an almost null structure on $\mcM$, and $\mcU$ an open set of $\mcM$. Let $k \in \{ 0, 1, 2 \}$. Suppose that the Weyl tensor is a section of $\mcC^k$ over $\mcU$, and is otherwise generic. Suppose further that the Cotton-York tensor is a section of $\mcA^{k-\tfrac{1}{2}}$ over $\mcU$. Then $\mcN$ is integrable in $\mcU$.
\end{thm}

\begin{proof}
This is essentially a local computation. Choose a local frame $\{ \bm{\xi}_\mu , \tilde{\bm{\xi}}_{\tilde{\mu}} \}$ over $\mcU$ adapted to $\mcN$, as described in section \ref{sec-null_basis}. Such a choice induces the local gradings \eqref{eq-Weyl_grading_even} and \eqref{eq-CY_grading_even} on the bundles $\mcC$ and $\mcA$ respectively. The condition that the Weyl and Cotton-York tensor be sections of $\mcC^k$ and $\mcA^{k-\tfrac{1}{2}}$ respectively is equivalent to their components in $\mcC_i$ and $\mcA_{i-\tfrac{1}{2}}$ vanishing for all $-2 \leq i \leq k-1$.

To show that $\mcN$ is integrable, we shall make use of the equivalent geodesy condition \eqref{eq-integrability_connection}. Locally, this can be expressed as a condition on the $\frac{1}{2}m^2(m-1)$ connection components
\begin{align} \label{eq-integrability_connection_components}
 \Gamma _{\kappa \mu \nu} = 0 \, ,
\end{align}
for all $\kappa$, $\mu$, $\nu$.

The gist of the proof is based on the fact that for each $k \in \{0,1,2\}$, in the local frame, and as a result of the algebraic degeneracy of the Weyl tensor and Cotton-York tensor, some of the differential equations defined by the components of the Bianchi identity \eqref{eq-Bianchi_identity} given in Appendix \ref{sec-Bianchi_identity} become \emph{algebraic} equations, which can be viewed as a homogeneous overdetermined system of linear equations on the unknowns $\Gamma_{\kappa \mu \nu}$. It is however not immediately clear whether these algebraic equations are all linearly independent. Hence, the proof will consist in singling out a subsystem of $\frac{1}{2}m^2(m-1)$ linearly independent equations on $\Gamma_{\kappa \mu \nu}$. In this case, the only possible solution will be the trivial solution
\eqref{eq-integrability_connection_components}.

More specifically, for each $k \in \{0,1,2\}$, we shall be able to choose a subsystem of $\frac{1}{2}m^2(m-1)$ linear equations which takes the matrix form
\begin{align}\label{mat-GS_even}
\left(
  \begin{BMAT}((rc)[2pt,6cm,6cm]{cccccc.cccccc}{cccccc.cccccc}
	K_{12} & * & \cdots & & & & & & & & \cdots & *\\
	* & K_{13} & & & & & & & & & & \vdots \\
	\vdots & & \ddots & & & & & & & & & \\
	& & & K_{\mu \nu} & & & & & & & & \\
	& & & & \ddots & & & & & & & \\
	& & & & & K_{m,m-1} & & & & & & \\
	& & & & & & \mathbf{K}_{1 2 3} & & & & & \\
	& & & & & & & \mathbf{K}_{1 2 4} & & & & \\
	& & & & & & & & \ddots & & & \\
	& & & & & & & & & \mathbf{K}_{\mu \nu \lambda} & & \vdots \\
	 \vdots & & & & & & & & & & \ddots & * \\
	* & \cdots & & & & & & & & \cdots & * & \mathbf{K}_{m-2,m-1, m} 
\end{BMAT}
\right)
\quad
\left(
\begin{BMAT}(rc)[2pt,1cm,6.5cm]{c}{cccccc.cccccc}
\Gamma_{1 1 2} \\
\Gamma_{1 1 3}  \\
\vdots \\
\Gamma_{\mu \mu \nu}   \\
\vdots \\
\Gamma_{m, m, m-1}  \\
\mathbf{u}_{1 2 3} \\
\mathbf{u}_{1 2 4} \\
\vdots \\
\mathbf{u}_{\mu \nu \lambda} \\
\vdots \\
\mathbf{u}_{m-2,m-1, m} 
\end{BMAT}
\right)
& =
\left(
\begin{BMAT}(rc)[2pt,.5cm,6.5cm]{c}{cccccc.cccccc}
0 \\
0 \\
\vdots \\
0 \\
\vdots \\
0 \\
\mathbf{0}_3 \\
\mathbf{0}_3 \\
\vdots \\
\mathbf{0}_3 \\
\vdots \\
\mathbf{0}_3
\end{BMAT}
\right) \, ,
\end{align}
or $\mathbf{K} \mathbf{u} = \mathbf{0}$ for short. Here, each entry of the $\frac{1}{2}m^2(m-1) \times 1$ vector $\mathbf{u}$ corresponds to a connection component $\Gamma _{\kappa \mu \nu}$. Some of these have been arranged in triples in the column vectors
\begin{align*}
 \mathbf{u}_{\mu \nu \lambda} & :=
\begin{pmatrix}
 \Gamma_{\mu \nu \lambda} \\ 
 \Gamma_{\nu \lambda \mu} \\
 \Gamma_{\lambda \mu \nu}
\end{pmatrix} \, ,
\end{align*}
for all $\kappa < \mu < \nu$.

On the other hand, each entry of the $\frac{1}{2}m^2(m-1) \times \frac{1}{2}m^2(m-1)$ matrix $\mathbf{K}$ will consist of a (constant) linear combination of components of the Weyl tensor. Corresponding to the arrangement of the entries of $\mathbf{u}$, we have also singled out the matrices $K_{\mu \nu}$ and $\mathbf{K}_{\mu \nu \kappa}$ of dimensions $1 \times 1$ and $3 \times 3$ respectively, each acting on $\Gamma_{\mu \mu \nu}$ and $\mathbf{u}_{\kappa \mu \nu}$ respectively. The remaining entries of $\mathbf{K}$ have been marked with an asterix $*$, the meaning of which will be clarified in a moment. In fact, from the structure of the matrix $\mathbf{K}$, we will be able to show that $\mathbf{K}$ is non-singular. This is made clear by the following lemma.
\begin{lem}\label{lem-det_matrix}
 Let $A$, $B$ be two distinct index sets, i.e. $A \cap B = \{ \emptyset \}$, and let $\{ f^{\alpha} \}_{\alpha \in A}$, $\{ g^{\beta} \}_{\beta \in B}$ be two collections of functions over $\mcU$. Consider the field of square matrices over $\mcU$ of the form
\begin{align}\label{eq-lem_mat}
\mathbf{K} :=
 \begin{pmatrix}
  \mathbf{D}_1 & \mathbf{A}_{12} & \ldots & \ldots & \mathbf{A}_{1,p} \\
 \mathbf{A}_{21} & \mathbf{D}_2 & & & \vdots \\
  \vdots &       &  \ddots & & \vdots \\
 \vdots &        &                 & \mathbf{D}_{p-1} & \mathbf{A}_{p-1,p}  \\
 \mathbf{A}_{p,1} & \ldots & \ldots & \mathbf{A}_{p,p-1} & \mathbf{D}_p 
 \end{pmatrix} \, , 
\end{align}
where for each $i$, the entries of the block square matrix $\mathbf{D}_i$ are polynomials in $f^\alpha$ with constant coefficients, and for each $i \neq j$, the entries of $\mathbf{A}_{ij}$ are polynomials in $g^{\beta}$ with constant coefficients. Then, the determinant of $\mathbf{K}$ is given by
\begin{align}\label{eq-det_matrix}
 \det \mathbf{K} & = G + D \, ,
\end{align}
where $G=G(f^{\alpha},g^\beta)$ is a polynomial in $f^\alpha$ and $g^\beta$ such that $G(f^\alpha,0)=0$, and $D=\prod_i ( \det \mathbf{D}_i )$.

In particular, assuming that
\begin{enumerate}
 \item for each $i$, $\mathbf{D}_i$ is non-singular, \label{lem-det_matrix1}
 \item the collections $\{ f^\alpha \}$ and $\{g^\beta\}$ are generically unrelated, in the sense that there are no algebraic relations between the functions $f^\alpha$ and $g^\beta$ for all $\alpha \in A$, $\beta \in B$, \label{lem-det_matrix2}
\end{enumerate}
then $\mathbf{M}$ is non-singular.
\end{lem}

\begin{proof}
Clearly, by definition, the determinant of $\mathbf{K}$ is a polynomial in $f^\alpha$ and $g^\beta$. Hence, we can always write
$\det \mathbf{K} = G + D$, where $G=G(f^\alpha,g^\beta)$ is a polynomial in $f^\alpha$ and $g^\beta$ such that $G(f^\alpha,0)=0$, and $D=D(f^\alpha)$ is a polynomial in $f^\alpha$. Setting $g^\beta = 0$ for all $\beta \in B$ yields $\det \mathbf{K} = D$. But $\mathbf{K}$ is now the block diagonal matrix $\diag ( \mathbf{D}_1 , \mathbf{D}_2 , \ldots , \mathbf{D}_p )$, which has determinant $\prod_i ( \det \mathbf{D}_i )$. Hence, equation \eqref{eq-det_matrix} is established.

Next, from part \eqref{lem-det_matrix1}, we have $\prod_i ( \det \mathbf{D}_i ) \neq 0$. Further, from the genericity assumption \eqref{lem-det_matrix2}, we have $G + \prod_i ( \det \mathbf{D}_i ) \neq 0$, i.e. $\mathbf{K}$ is non-singular.
\end{proof}
\vspace{+10pt}

Thus, to prove the theorem for each $k \in \{0,1,2\}$, it suffices to check whether the matrix $\mathbf{K}$ of the system of equations \eqref{mat-GS_even} satisfies the hypotheses of Lemma \ref{lem-det_matrix}. But it turns out that this is precisely the case: it will be seen that the index structure of the components of the Weyl tensor, modulo the Weyl symmetries, may be split into two distinct sets $A$ and $B$ such that the hypotheses of Lemma \ref{lem-det_matrix} hold, with
\begin{align}\label{eq-special_term}
 D & = \left( \prod_{\mu \neq \nu} K_{\mu \nu} \right) \cdot \left( \prod_{\kappa < \lambda < \rho} \det (\mathbf{K}_{\kappa \lambda \rho} ) \right) \, .
\end{align}
Moreover, one can simply invoke the genericity assumption on the Weyl tensor to deduce the additional requirements \eqref{lem-det_matrix1} and \eqref{lem-det_matrix2} of Lemma \ref{lem-det_matrix}. It will then follow that the system \eqref{mat-GS_even} is non-singular, and must therefore have trivial solution \eqref{eq-integrability_connection_components}.

\begin{rem}
 By `components of the Weyl tensor, modulo the Weyl symmetries', we mean that the components of the Weyl tensor are subject to the Riemman symmetries
\begin{align*}
 \bm{C} ( \bm{X} , \bm{Y} , \bm{Z}, \bm{W} ) + \bm{C} ( \bm{Y} , \bm{Z} , \bm{Y}, \bm{W} ) + \bm{C} ( \bm{Z} , \bm{X} , \bm{Y}, \bm{W} ) & = 0 \, ,
\end{align*}
together with the tracefree condition
\begin{align*}
 \sum_{\sigma} \left( \bm{C} ( \bm{\xi}_\sigma , \bm{X} , \tilde{\bm{\xi}}_{\tilde{\sigma}}, \bm{Y} ) + \bm{C} ( \tilde{\bm{\xi}}_{\tilde{\sigma}}, \bm{X} , \bm{\xi}_\sigma , \bm{Y} ) \right) + \epsilon \bm{C} ( \bm{\xi}_0, \bm{X} , \bm{\xi}_0 , \bm{Y} ) & = 0 \, ,
\end{align*}
for all vector fields $\bm{X},\bm{Y},\bm{Z}, \bm{W} \in \Gamma ( \Tgt \mcM )$. A component has an index structure in the indexing set $A$ if and only if any other component related to it by a Weyl symmetry also has an index structure in $A$. In this case, no ambiguity\footnote{There is one notable exception that will be encountered in the odd-dimensional version of the theorem, but the argument there can be adapted with no major difficulty.} arises in the application of Lemma \ref{lem-det_matrix}. In particular, one should check that the linear combination of the components involved in $K_{\mu \nu}$ and $\det (\mathbf{K}_{\kappa \lambda \rho} )$ do not lead to the vanishing of these scalars. This step can be carried out by inspection, and will be left to the reader.
\end{rem}

\paragraph{Case $k=0$:}
Assume that the Weyl tensor is a section of $\mcC^0$ so that
\begin{align}\label{eq-C0_components}
 C_{\mu \nu \kappa \lambda} = C_{\mu \nu \kappa \tilde{\lambda}} & = 0 \, ,
\end{align}
for all $\mu$, $\nu$, $\kappa$, $\lambda$. Then equations \eqref{eq-Bianchi3a} become
\begin{align}
0 & = 2 g \ind{_{\tilde{\rho} \lb{\kappa}}} A \ind{_{\rb{\lambda} \mu \nu}} + 2 \Gamma \ind{_{[\mu \nu]} ^{\tilde{\sigma}}} C \ind{_{\tilde{\rho} {\tilde{\sigma}} \kappa \lambda}} + 4 \Gamma \ind{_{\lb{\mu}| \lb{\kappa}} ^{\tilde{\sigma}}} C \ind{_{\rb{\lambda}|{\tilde{\sigma}}| \rb{\nu} \tilde{\rho}}} =: B \ind{_{\mu \nu \tilde{\rho} | \kappa \lambda}} \, . \tag{\ref{eq-Bianchi3a}}
\end{align}
Here $B \ind{_{\mu \nu \tilde{\rho} | \kappa \lambda}}$ is merely a short hand for this set of algebraic equations. Now, suppose that the Weyl tensor is otherwise generic, and the Cotton-York tensor is a section of $\mcA^{-\tfrac{1}{2}}$. Then the set of equations \eqref{eq-Bianchi3a} constitutes a homogeneous system of $\tfrac{1}{4}m^3 (m-1)^2$ equations on $\tfrac{1}{2}m^2 (m-1)$ unknowns. Pick all $m(m-1)$ equations $B \ind{_{\mu \nu \tilde{\mu} | \mu \nu}}$, and all $\tfrac{1}{2}m(m-1)(m-2)$ equations $B \ind{_{\mu \nu \tilde{\mu} | \mu \lambda}}$, which, dropping the Einstein summation convention, can be written as
\begin{multline*}
 0 =  \Gamma \ind{_{\mu \mu \nu}} ( C \ind{_{\nu \tilde{\nu} \nu \tilde{\mu}}}  + C \ind{_{\mu \tilde{\mu} \nu \tilde{\mu}}} )  
 + \Gamma \ind{_{\nu \nu \mu}} ( C \ind{_{\mu \tilde{\mu} \mu \tilde{\mu}}} + C \ind{_{\tilde{\mu} \tilde{\nu} \mu \nu}} + C \ind{_{\nu \tilde{\nu} \mu \tilde{\mu}}} )  \\
+ \sum_{\sigma \neq  \mu, \nu } \left( \Gamma \ind{_{\mu \nu \sigma}} ( C \ind{_{\tilde{\mu} \tilde{\sigma} \mu \nu}} - C \ind{_{\mu \tilde{\sigma} \nu \tilde{\mu}}} ) + \Gamma \ind{_{\mu \mu \sigma}} C \ind{_{\nu \tilde{\sigma} \nu \tilde{\mu}}} - \Gamma \ind{_{\nu \mu \sigma}} ( C \ind{_{\tilde{\mu} \tilde{\sigma} \mu \nu}}  + C \ind{_{\nu \tilde{\sigma} \mu \tilde{\mu}}} ) + \Gamma \ind{_{\nu \nu \sigma}} C \ind{_{\mu \tilde{\sigma} \mu \tilde{\mu}}} \right) \, ,
\end{multline*}
\begin{multline*}
 0 =  \Gamma \ind{_{\mu \mu \nu}}  C \ind{_{\lambda \tilde{\nu} \nu \tilde{\mu}}}
+ \Gamma \ind{_{\mu \mu \lambda}} C \ind{_{\mu \tilde{\mu} \nu \tilde{\mu}}} 
+ \Gamma \ind{_{\nu \nu \mu}} ( C \ind{_{\tilde{\mu} \tilde{\nu} \mu \lambda}}  + C \ind{_{\lambda \tilde{\nu} \mu \tilde{\mu}}} )
- \Gamma \ind{_{\nu \nu \lambda}} C \ind{_{\mu \tilde{\nu} \mu \tilde{\mu}}}   \\
+ \Gamma \ind{_{\mu \nu \lambda}} ( C \ind{_{\tilde{\mu} \tilde{\lambda} \mu \lambda}} + C \ind{_{\mu \tilde{\nu} \nu \tilde{\mu}}} )
 + \Gamma \ind{_{\nu \lambda \mu}} ( C \ind{_{\lambda \tilde{\lambda} \mu \tilde{\mu}}}  + C \ind{_{\mu \tilde{\mu} \mu \tilde{\mu}}} + C \ind{_{\tilde{\mu} \tilde{\lambda} \mu \lambda}} )  \\
+ \sum_{\sigma \neq  \lambda, \mu, \nu } \left( \Gamma \ind{_{\mu \nu \sigma}} C \ind{_{\tilde{\mu} \tilde{\sigma} \mu \lambda}} + \Gamma \ind{_{\mu \mu \sigma}} C \ind{_{\lambda \tilde{\sigma} \nu \tilde{\mu}}} - \Gamma \ind{_{\mu \lambda \sigma}} C \ind{_{\mu \tilde{\sigma} \nu \tilde{\mu}}} - \Gamma \ind{_{\nu \mu \sigma}} ( C \ind{_{\tilde{\mu} \tilde{\sigma} \mu \lambda}} + C \ind{_{\lambda \tilde{\sigma} \mu \tilde{\mu}}} ) + \Gamma \ind{_{\nu \lambda \sigma}} C \ind{_{\mu \tilde{\sigma} \mu \tilde{\mu}}} \right) \, ,
\end{multline*}
respectively. These sets of equations can be put into the matrix form \eqref{mat-GS_even}, by defining, for all $\kappa$, $\rho$ distinct, and $\mu < \nu < \lambda$,
\begin{align*}
 K_{\kappa \rho} & :=  C \ind{_{\kappa \tilde{\kappa} \kappa \tilde{\kappa}}} + C \ind{_{\tilde{\kappa} \tilde{\rho} \kappa \rho}} + C \ind{_{\rho \tilde{\rho} \kappa \tilde{\kappa}}}  \, , \\
 \mathbf{K}_{\mu \nu \lambda} & := \begin{pmatrix}
  C \ind{_{\tilde{\mu} \tilde{\lambda} \mu \lambda}} + C \ind{_{\mu \tilde{\nu} \nu \tilde{\mu}}} & C \ind{_{\lambda \tilde{\lambda} \mu \tilde{\mu}}}  + C \ind{_{\mu \tilde{\mu} \mu \tilde{\mu}}} + C \ind{_{\tilde{\mu} \tilde{\lambda} \mu \lambda}} & 0 \\
0 & C \ind{_{\tilde{\nu} \tilde{\mu} \nu \mu}} + C \ind{_{\nu \tilde{\lambda} \lambda \tilde{\nu}}} & C \ind{_{\mu \tilde{\mu} \nu \tilde{\nu}}}  + C \ind{_{\nu \tilde{\nu} \nu \tilde{\nu}}} + C \ind{_{\tilde{\nu} \tilde{\mu} \nu \mu}} \\
 C \ind{_{\nu \tilde{\nu} \lambda \tilde{\lambda}}}  + C \ind{_{\lambda \tilde{\lambda} \lambda \tilde{\lambda}}} + C \ind{_{\tilde{\lambda} \tilde{\nu} \lambda \nu}} & 0 & C \ind{_{\tilde{\lambda} \tilde{\nu} \lambda \nu}} + C \ind{_{\lambda \tilde{\mu} \mu \tilde{\lambda}}}
 \end{pmatrix} \, ,
\end{align*}
respectively. The latter has determinant
\begin{align*}
  \det (\mathbf{K}_{\mu \nu \lambda} ) & = ( C \ind{_{\tilde{\mu} \tilde{\lambda} \mu \lambda}} + C \ind{_{\mu \tilde{\nu} \nu \tilde{\mu}}} ) \cdot (C \ind{_{\tilde{\nu} \tilde{\mu} \nu \mu}} + C \ind{_{\nu \tilde{\lambda} \lambda \tilde{\nu}}} ) \cdot ( C \ind{_{\tilde{\lambda} \tilde{\nu} \lambda \nu}} + C \ind{_{\lambda \tilde{\mu} \mu \tilde{\lambda}}} ) \\
& \qquad \qquad + ( C \ind{_{\lambda \tilde{\lambda} \mu \tilde{\mu}}}  + C \ind{_{\mu \tilde{\mu} \mu \tilde{\mu}}} + C \ind{_{\tilde{\mu} \tilde{\lambda} \mu \lambda}} ) \cdot ( C \ind{_{\mu \tilde{\mu} \nu \tilde{\nu}}}  + C \ind{_{\nu \tilde{\nu} \nu \tilde{\nu}}} + C \ind{_{\tilde{\nu} \tilde{\mu} \nu \mu}} ) \cdot ( C \ind{_{\nu \tilde{\nu} \lambda \tilde{\lambda}}}  + C \ind{_{\lambda \tilde{\lambda} \lambda \tilde{\lambda}}} + C \ind{_{\tilde{\lambda} \tilde{\nu} \lambda \nu}} ) \, ,
\end{align*}
which can be seen to be non-vanishing by the genericity assumption,\footnote{It may be of concern that the tracefree property of the Weyl tensor could make $\det (\mathbf{K}_{\mu \nu \lambda} )$ vanish in principle. But if one notes that $\det (\mathbf{K}_{\mu \nu \lambda} )$ depends only on three distinct indices $\mu$, $\nu$, $\lambda$, we see that the only dimension where this issue could arise is six. To settle the issue, we expand the determinant and eliminate the dependency of the components of the Weyl tensor by choosing a select few. A judicious choice leads to
\begin{multline*}
\det (\mathbf{K}_{\mu \nu \lambda} ) = \\ \frac{1}{8} ( C \ind{_{\mu \tilde{\mu} \mu \tilde{\mu}}} + C \ind{_{\mu \tilde{\mu} \lambda \tilde{\lambda}}} - C \ind{_{\mu \tilde{\mu} \nu\tilde{\nu}}} + 4 C \ind{_{\mu \tilde{\nu} \nu \tilde{\mu}}} ) \cdot ( - C \ind{_{\nu \tilde{\nu} \nu \tilde{\nu}}} + C \ind{_{\nu \tilde{\nu} \lambda \tilde{\lambda}}} + 3 C \ind{_{\mu \tilde{\mu} \nu\tilde{\nu}}} - 4 C \ind{_{\mu \tilde{\nu} \nu \tilde{\mu}}} ) \cdot ( C \ind{_{\nu \tilde{\nu} \nu \tilde{\nu}}} - C \ind{_{\mu \tilde{\mu} \mu \tilde{\mu}}} + C \ind{_{\mu \tilde{\mu} \lambda \tilde{\lambda}}} + C \ind{_{\nu \tilde{\nu} \lambda \tilde{\lambda}}} ) \\
+ \frac{1}{4}(3 C \ind{_{\mu \tilde{\mu} \mu \tilde{\mu}}} + 5 C \ind{_{\mu \tilde{\mu} \lambda \tilde{\lambda}}} + C \ind{_{\mu \tilde{\mu} \nu \tilde{\nu}}} ) \cdot ( C \ind{_{\nu \tilde{\nu} \nu \tilde{\nu}}} + 2 C \ind{_{\mu \tilde{\mu} \nu \tilde{\nu}}} - C \ind{_{\mu \tilde{\nu} \nu \tilde{\mu}}} )  \cdot (C \ind{_{\mu \tilde{\mu} \mu \tilde{\mu}}} + C \ind{_{\nu \tilde{\nu} \nu \tilde{\nu}}} + 7 C \ind{_{\mu \tilde{\nu} \nu \tilde{\mu}}} - 5 C \ind{_{\mu \tilde{\mu} \nu \tilde{\nu}}} + C \ind{_{\nu \tilde{\nu} \lambda \tilde{\lambda}}} ) \, .
\end{multline*}
By the genericity assumption, this has to be non-vanishing.} and $K_{\mu \nu}$ and $\mathbf{K}_{\mu \nu \lambda}$ are thus non-singular. Hence, the term \eqref{eq-special_term} is non-vanishing.

Further, one can check that the components of the Weyl tensor in the entries of $K_{\mu \nu}$ and $\mathbf{K}_{\kappa \mu \nu}$ have distinct index structures from those in the remaining entries of $\mathbf{K}$. Hence, we can apply Lemma \ref{lem-det_matrix} to conclude that $\mathbf{K}$ is non-singular, thus establishing condition \eqref{eq-integrability_connection_components}.

\paragraph{Case $k=1$:}
Assume that the Weyl tensor is a section of $\mcC^1$ so that conditions \eqref{eq-C0_components} hold together with
\begin{align} \label{eq-C1_components}
 C_{\mu \tilde{\nu} \kappa \tilde{\lambda}} & = 0 \, ,
\end{align}
for all $\mu$, $\nu$, $\kappa$, $\lambda$. Then equations \eqref{eq-Bianchi1a} become
\begin{align}
0 & = - 2 g \ind{_{\lb{\mu} | \tilde{\lambda}}} A \ind{_{\kappa | \rb{\nu} \tilde{\rho}}} + g \ind{_{\tilde{\rho} \kappa}} A \ind{_{\tilde{\lambda} \mu \nu}} 
+ 2 \Gamma \ind{_{[\mu \nu]} ^{\tilde{\sigma}}} C \ind{_{\tilde{\rho} {\tilde{\sigma}} \kappa \tilde{\lambda}}}
+ 2 \Gamma \ind{_{\lb{\mu}| \kappa} ^{\tilde{\sigma}}} C \ind{_{\tilde{\lambda}|{\tilde{\sigma}}| \rb{\nu} \tilde{\rho}}} =: B_{\mu \nu \tilde{\rho} | \kappa \tilde{\lambda}} \, . \tag{\ref{eq-Bianchi1a}}
\end{align}
Now, suppose that the Weyl tensor is otherwise generic, and the Cotton-York tensor is a section of $\mcA^{\tfrac{1}{2}}$. Then the set of equations \eqref{eq-Bianchi1a} constitutes a homogeneous system of $\tfrac{1}{2}m^4 (m-1)$ equations on $\tfrac{1}{2}m^2 (m-1)$ unknowns. Pick all $m(m-1)$ equations $B_{\mu \nu \tilde{\nu} | \nu \tilde{\nu}}$, and all $\tfrac{1}{2}m(m-1)(m-2)$ equations $B_{\mu \nu \tilde{\nu} | \nu \tilde{\kappa}}$
\begin{align*}
 0 & = 
 \Gamma \ind{_{\nu \nu \mu}} ( C \ind{_{\tilde{\mu} {\tilde{\nu}} \mu \tilde{\mu}}} + C \ind{_{\tilde{\mu} \tilde{\nu} \mu \tilde{\mu}}}  )
+ \Gamma \ind{_{\mu \mu \nu}} C \ind{_{\tilde{\mu} \tilde{\nu} \nu \tilde{\mu}}}
+ \sum_{\sigma \neq \mu, \nu} \left(  \Gamma \ind{_{\mu \nu \sigma}} C \ind{_{\tilde{\mu} {\tilde{\sigma}} \mu \tilde{\mu}}}
- \Gamma \ind{_{\nu \mu \sigma}} C \ind{_{\tilde{\mu} {\tilde{\sigma}} \mu \tilde{\mu}}}
+ \Gamma \ind{_{\mu \mu \sigma}} C \ind{_{\tilde{\mu} \tilde{\sigma} \nu \tilde{\mu}}}
- \Gamma \ind{_{\nu \mu \sigma}} C \ind{_{\tilde{\mu} \tilde{\sigma} \mu \tilde{\mu}}} \right) \, ,
\end{align*}
\begin{multline*}
 0 = - \Gamma \ind{_{\mu \mu \nu}} C \ind{_{\tilde{\nu} {\tilde{\mu}} \kappa \tilde{\nu}}}
- \Gamma \ind{_{\mu \mu \kappa}} C \ind{_{\tilde{\nu} \tilde{\mu} \nu \tilde{\nu}}} 
+ \Gamma \ind{_{\nu \kappa \mu}} ( C \ind{_{\tilde{\nu} \tilde{\mu} \tilde{\nu} \mu}} - C \ind{_{\tilde{\nu} {\tilde{\kappa}} \tilde{\nu} \kappa}} ) 
+ \Gamma \ind{_{\mu \nu \kappa}} 
C \ind{_{\tilde{\nu} {\tilde{\kappa}} \kappa \tilde{\nu}}}  \\
+ \sum_{\sigma \neq \mu, \nu, \kappa} \left(  \Gamma \ind{_{\mu \nu \sigma}} C \ind{_{\tilde{\nu} {\tilde{\sigma}} \kappa \tilde{\nu}}}
- \Gamma \ind{_{\nu \mu \sigma}} C \ind{_{\tilde{\nu} {\tilde{\sigma}} \kappa \tilde{\nu}}}
+ \Gamma \ind{_{\mu \kappa \sigma}} C \ind{_{\tilde{\nu} \tilde{\sigma} \nu \tilde{\nu}}}
- \Gamma \ind{_{\nu \kappa \sigma}} C \ind{_{\tilde{\nu} \tilde{\sigma} \mu \tilde{\nu}}} \right) \, .
\end{multline*}
These sets of equations can be put into the matrix form \eqref{mat-GS_even}, where, for all $\kappa$, $\rho$ distinct, and $\mu < \nu < \lambda$,
\begin{align*}
 K_{\kappa \rho} & := C \ind{_{\tilde{\kappa} \tilde{\rho} \rho \tilde{\kappa}}}  \, ,
&
 \mathbf{K}_{\mu \nu \kappa} & := \begin{pmatrix}
  C \ind{_{\tilde{\nu} {\tilde{\kappa}} \kappa \tilde{\nu}}}  & C \ind{_{\tilde{\nu} \tilde{\mu} \tilde{\nu} \mu}} - C \ind{_{\tilde{\nu} {\tilde{\kappa}} \tilde{\nu} \kappa}} & 0 \\
 0 & C \ind{_{\tilde{\kappa} {\tilde{\mu}} \mu \tilde{\kappa}}}  & C \ind{_{\tilde{\kappa} \tilde{\nu} \tilde{\kappa} \nu}} - C \ind{_{\tilde{\kappa} {\tilde{\mu}} \tilde{\kappa} \mu}} \\
 C \ind{_{\tilde{\mu} \tilde{\kappa} \tilde{\mu} \kappa}} - C \ind{_{\tilde{\mu} {\tilde{\nu}} \tilde{\mu} \nu}} & 0 &
 C \ind{_{\tilde{\mu} {\tilde{\nu}} \nu \tilde{\mu}}}  
 \end{pmatrix} \, ,
\end{align*}
respectively. The latter has determinant
\begin{align*}
 \det (\mathbf{K}_{\mu \nu \kappa} ) & = C \ind{_{\tilde{\nu} {\tilde{\kappa}} \kappa \tilde{\nu}}}   \cdot C \ind{_{\tilde{\kappa} {\tilde{\mu}} \mu \tilde{\kappa}}}  \cdot C \ind{_{\tilde{\mu} {\tilde{\nu}} \nu \tilde{\mu}}} 
+ ( C \ind{_{\tilde{\nu} \tilde{\mu} \tilde{\nu} \mu}} - C \ind{_{\tilde{\nu} {\tilde{\kappa}} \tilde{\nu} \kappa}}  ) \cdot ( C \ind{_{\tilde{\kappa} \tilde{\nu} \tilde{\kappa} \nu}} - C \ind{_{\tilde{\kappa} {\tilde{\mu}} \tilde{\kappa} \mu}} ) \cdot (  C \ind{_{\tilde{\mu} \tilde{\kappa} \tilde{\mu} \kappa}} - C \ind{_{\tilde{\mu} {\tilde{\nu}} \tilde{\mu} \nu}}  ) \, ,
\end{align*}
which can be seen to be non-vanishing by the genericity assumption,\footnote{If anything goes wrong because of the tracefree property of the Weyl tensor, it has to happen in six dimensions. But when $m=3$, one can check that the determinant simplifies to $\det (\mathbf{K}_{\mu \nu \kappa} ) = 9 C \ind{_{\tilde{\nu} {\tilde{\kappa}} \kappa \tilde{\nu}}}   \cdot C \ind{_{\tilde{\kappa} {\tilde{\mu}} \mu \tilde{\kappa}}}  \cdot C \ind{_{\tilde{\mu} {\tilde{\nu}} \nu \tilde{\mu}}}$, which is clearly non-vanishing.} and $K_{\mu \nu}$ and $\mathbf{K}_{\mu \nu \kappa}$ are thus non-singular. Hence, the term \eqref{eq-special_term} is non-vanishing.

Further, one can check that the components of the Weyl tensor in the entries of $K_{\mu \nu}$ and $\mathbf{K}_{\kappa \mu \nu}$ have distinct index structures from those in the remaining entries of $\mathbf{K}$. Hence, we can apply Lemma \ref{lem-det_matrix} to conclude that $\mathbf{K}$ is non-singular, thus establishing conditions \eqref{eq-integrability_connection_components}.

\paragraph{Case $k=2$:}
Assume that the Weyl tensor is a section of $\mcC^{2}$ so that conditions \eqref{eq-C0_components} and \eqref{eq-C1_components} hold together with
\begin{align}\label{eq-C2_components}
C_{\tilde{\mu} \tilde{\nu} \kappa \tilde{\lambda}} & = 0  \, ,
\end{align}
for all $\mu$, $\nu$, $\kappa$, $\lambda$. Then equations \eqref{eq-Bianchi-1a} become
\begin{align}
0 &=  - 2 g \ind{_{\lb{\tilde{\mu}} | \lambda}} A \ind{_{\tilde{\kappa} | \rb{\tilde{\nu}} \rho}} + g \ind{_{\rho \tilde{\kappa}}} A \ind{_{\lambda \tilde{\mu} \tilde{\nu}}} 
- \Gamma \ind{_{\rho \lambda} ^{\tilde{\sigma}}} C \ind{_{\tilde{\kappa} {\tilde{\sigma}} \tilde{\mu} \tilde{\nu}}} =: B_{\tilde{\mu} \tilde{\nu} \rho | \lambda \tilde{\kappa}} \tag{\ref{eq-Bianchi-1a}} \, .
\end{align}
Now, suppose that the Weyl tensor is otherwise generic, and the Cotton-York tensor is a section of $\mcA^{\tfrac{3}{2}}$. Then the set of equations \eqref{eq-Bianchi-1a} constitutes a homogeneous system of $\tfrac{1}{2}m^4 (m-1)$ equations on $\tfrac{1}{2}m^2 (m-1)$ unknowns. Pick all $m(m-1)$ equations $B_{\tilde{\mu} \tilde{\nu} \nu | \nu \tilde{\nu}}$, and all $\tfrac{1}{2}m(m-1)(m-2)$ equations $B_{\tilde{\mu} \tilde{\nu} \nu | \lambda \tilde{\nu}}$
\begin{align*}
 0 &=  - \Gamma \ind{_{\nu \mu \nu}} C \ind{_{\tilde{\mu} \tilde{\nu} \tilde{\mu} \tilde{\nu}}}
- \sum_{\sigma \neq \mu, \nu } \Gamma \ind{_{\nu \nu \sigma}} C \ind{_{\tilde{\nu} \tilde{\sigma} \tilde{\mu} \tilde{\nu}}} \, ,
&
 0 &= - \Gamma \ind{_{\nu \mu \lambda }} C \ind{_{\tilde{\mu} \tilde{\nu} \tilde{\mu} \tilde{\nu}}}
- \sum_{\sigma \neq \mu, \nu } \Gamma \ind{_{\nu \lambda \sigma}} C \ind{_{\tilde{\nu} \tilde{\sigma} \tilde{\mu} \tilde{\nu}}} \, ,
\end{align*}
respectively. These can be put into the matrix form \eqref{mat-GS_even}, where, for all $\kappa$, $\rho$ distinct, and $\mu < \nu < \lambda$,
\begin{align*}
 K_{\kappa \rho} & :=
  C \ind{_{\tilde{\kappa} \tilde{\rho} \tilde{\kappa} \tilde{\rho}}} \, ,
& \mathbf{K}_{\nu \lambda \mu} & := 
\begin{pmatrix}
 C \ind{_{\tilde{\mu} \tilde{\nu} \tilde{\mu} \tilde{\nu}}} & 0 & 0 \\
 0 & C \ind{_{\tilde{\lambda} \tilde{\mu} \tilde{\lambda} \tilde{\mu}}} & 0 \\
 0 & 0 & C \ind{_{\tilde{\nu} \tilde{\lambda} \tilde{\nu} \tilde{\lambda}}}
\end{pmatrix} \, ,
\end{align*}
respectively. The latter has determinant
\begin{align*}
 \det (\mathbf{K}_{\mu \nu \lambda} ) & = C \ind{_{\tilde{\mu} \tilde{\nu} \tilde{\mu} \tilde{\nu}}} \cdot C \ind{_{\tilde{\lambda} \tilde{\mu} \tilde{\lambda} \tilde{\mu}}} \cdot C \ind{_{\tilde{\nu} \tilde{\lambda} \tilde{\nu} \tilde{\lambda}}} \, ,
\end{align*}
which can be seen to be non-vanishing by the genericity assumption, and $K_{\mu \nu}$ and $\mathbf{K}_{\mu \nu \lambda}$ are thus non-singular. Hence, the term \eqref{eq-special_term} is non-vanishing.

Further, one can check that the components of the Weyl tensor in the entries of $K_{\mu \nu}$ and $\mathbf{K}_{\kappa \mu \nu}$ have distinct index structures from those in the remaining entries of $\mathbf{K}$. Hence, we can apply Lemma \ref{lem-det_matrix} to conclude that $\mathbf{K}$ is non-singular, thus establishing conditions \eqref{eq-integrability_connection_components}.
\end{proof}

\begin{rem}
 In `low' dimensions it can be checked from the Bianchi identity that the condition of Proposition \ref{prop-integrability_condition} cannot be sufficient for the integrability of $\mcN$. It is however not clear whether this remains true in `high enough' dimensions. In this case the Cotton-York tensor would present no obstruction to the integrability of $\mcN$.
\end{rem}

\subsection{The complex Goldberg-Sachs theorem in odd dimensions}
We proceed as in even dimensions. The proof of the odd-dimensional generalisation of implication \eqref{thm-GS4b} of Theorem \ref{thm-GS4} is identical to its even-dimensional counterpart.
\begin{prop}\label{prop-GS_odd}
 Let $(\mcM, \bm{g})$ be a $(2m+1)$-dimensional complex Riemannian manifold, where $m \geq 2$. Let $\mcN$ be an almost null structure on $\mcM$, and $\mcU$ an open subset of $\mcM$. Let $k \in \{0 , 1 , 2, 3, 4\}$. Suppose that the Weyl tensor is a section of $\mcC^k$ over $\mcU$. Then the Cotton-York tensor is a section of $\mcA^{k - 2}$ over $\mcU$. Suppose further that $\mcN$ is integrable in $\mcU$. Then the Cotton-York tensor is a section of $\mcA^{k - 1}$ over $\mcU$.
\end{prop}

Next, the odd-dimensional generalisation of implication \eqref{thm-GS4b} of Theorem \ref{thm-GS4} can be expressed as follows.
\begin{thm} \label{thm-GS_odd}
 Let $(\mcM, \bm{g})$ be a $(2m+1)$-dimensional complex Riemannian manifold, where $m \geq 2$. Let $\mcN$ be an almost null structure on $\mcM$, and $\mcU$ an open subset of $\mcM$. Let $k \in \{ 0,1,2,3,4 \}$. Suppose that the Weyl tensor is a section of $\mcC^k$ over $\mcU$, and is otherwise generic. Suppose further that the Cotton-York tensor is a section of $\mcA^{k-1}$ over $\mcU$. Then $\mcN$ is integrable in $\mcU$.
\end{thm}

\begin{proof}
The odd-dimensional case follows exactly the same procedure as the even-dimensional one. Choose a local frame $\{\bm{\xi}_\mu , \tilde{\bm{\xi}}_{\tilde{\mu}} , \bm{\xi}_0 \}$ over $\mcU$  adapted to $\mcN$. Then, we have local gradings \eqref{eq-Weyl_grading_odd} and \eqref{eq-CY_grading_odd} on the bundles $\mcC$ and $\mcA$ respectively. The condition that the Weyl and Cotton-York tensor be sections of $\mcC^k$ and $\mcA^{k-1}$ respectively is equivalent to their components in $\mcC_i$ and $\mcA_{i-1}$ vanishing for all $-4 \leq i \leq k-1$.

The integrability of the almost null structure, by Lemma \ref{lem-integrability_connection}, is then equivalent to the connection components satisfying
\begin{align}
 \Gamma_{\kappa \mu \nu } & = 0 \, , \tag{\ref{eq-integrability_connection_components}} \\
 \Gamma_{\kappa \mu 0} & = 0 \, , \label{eq-integrability_connection_components_odd1} \\
 \Gamma_{0 \mu \nu} & = 0 \label{eq-integrability_connection_components_odd2} \, ,
\end{align}
for all $\kappa$, $\mu$, $\nu$. These constitute $\tfrac{1}{2}m^2(m-1)$, $m^2$, and $\tfrac{1}{2}m(m-1)$ conditions respectively.

As in the even-dimensional case, for each $k \in \{0,1, 2, 3, 4\}$, the assertion of the theorem is proved by means of the Bianchi identity, which, from the algebraic degeneracy of the Weyl and Cotton-York tensors, gives rise to a homogeneous overdetermined system of linear equations on the unknowns $\Gamma_{\kappa \mu \nu}$, $\Gamma_{\kappa \mu 0}$ and $\Gamma_{0 \mu \nu}$ for all $\kappa$, $\mu$, $\nu$. In fact, for each $k \in \{0,1, 2, 3,4\}$, we shall be able to split the proof into three steps, as it turns out that 
the relevant algebraic equations arising from the Bianchi identity can be arranged into three systems. A first system consists of equations on $\Gamma_{\kappa \mu \nu}$, a second one on $\Gamma_{\kappa \mu \nu}$, $\Gamma_{\kappa \mu 0}$, and a third one on $\Gamma_{\kappa \mu \nu}$, $\Gamma_{\kappa \mu 0}$, $\Gamma_{0 \mu \nu}$, for all $\kappa$, $\mu$, $\nu$. Hence, we will be able to conclude \eqref{eq-integrability_connection_components}, \eqref{eq-integrability_connection_components_odd1}, and \eqref{eq-integrability_connection_components_odd2} successively, by considering suitable subsystems of these systems.

Thus, to show conditions \eqref{eq-integrability_connection_components}, we can simply recycle the setup of the proof of Theorem \ref{thm-GS}, in particular matrix \eqref{mat-GS_even}.
\begin{rem}
  While Theorem \ref{thm-GS} is stated for $m\geq 3$ only, the proof can still be re-used in the case $m=2$, i.e. when $\mcM$ is five-dimensional. In this case, only the `upper half' of the system \eqref{mat-GS_even} is relevant.
\end{rem}
Once conditions \eqref{eq-integrability_connection_components} have been established, we can move on to show condition \eqref{eq-integrability_connection_components_odd1}, by considering a system of linear equations of the form
\begin{align} \label{mat-GS_odd1}
\left(
  \begin{BMAT}((rc)[2pt,6cm,6cm]{cccccc.cccccc}{cccccc.cccccc}
	L_{1} & * & \cdots & & & & & & & & \cdots & * \\
	* & L_{2} & & & & & & & & & & \vdots \\
	\vdots & & \ddots & & & & & & & & & \\
	& & & L_{\mu} & & & & & & & & \\
	& & & & \ddots & & & & & & & \\
	& & & & & L_{m} & & & & & & \\
	& & & & & & \mathbf{L}_{1 2} & & & & & \\
	& & & & & & & \mathbf{L}_{1 3} & & & & \\
	& & & & & & & & \ddots & & & \\
	& & & & & & & & & \mathbf{L}_{\mu \nu} & & \vdots \\
	\vdots & & & & & & & & & & \ddots & * \\
	* & \cdots & & & & & & & & \cdots & * & \mathbf{L}_{m-1, m} 
\end{BMAT}
\right)
\quad
\left(
\begin{BMAT}(rc)[2pt,1cm,6.5cm]{c}{cccccc.cccccc}
\Gamma_{1 1 0} \\
\Gamma_{2 2 0} \\
\vdots \\
\Gamma_{\mu \mu 0} \\
\vdots \\
\Gamma_{m m 0} \\
\mathbf{v}_{1 2} \\
\mathbf{v}_{1 3} \\
\vdots \\
\mathbf{v}_{\mu \nu} \\
\vdots \\
\mathbf{v}_{m-1, m} 
\end{BMAT}
\right)
& =
\left(
\begin{BMAT}(rc)[2pt,1cm,6.5cm]{c}{cccccc.cccccc}
0 \\
0 \\
\vdots \\
0 \\
\vdots \\
0 \\
\mathbf{0}_2 \\
\mathbf{0}_2 \\
\vdots \\
\mathbf{0}_2 \\
\vdots \\
\mathbf{0}_2
\end{BMAT}
\right) \, ,
\end{align}
or $\mathbf{L} \mathbf{v} = \mathbf{0}$ for short. Here, each entry of the $m^2 \times 1$ vector $\mathbf{v}$ corresponds to a connection component $\Gamma_{\mu \nu 0}$. Some have been arranged in pairs as defined by the $\tfrac{1}{2}m(m-1)$ column vectors
\begin{align*}
 \mathbf{v}_{\mu \nu} :=
\begin{pmatrix}
 \Gamma_{\mu \nu 0} \\ 
 - \Gamma_{\nu \mu 0} 
\end{pmatrix} \, ,
\end{align*}
for all $\mu < \nu $. As in the previous step, we have singled out the matrices $L_\mu$ and $\mathbf{L}_{\mu \nu}$ of dimensions $1 \times 1$ and $2 \times 2$ respectively. Following the same argument as presented in the proof Theorem \ref{thm-GS}, these will play a central r\^{o}le in demonstrating that $\mathbf{L}$ is non-singular. In particular, with reference to Lemma \ref{lem-det_matrix} and equation \eqref{eq-lem_mat}, the term $D$ of the determinant of matrix \eqref{mat-GS_odd1} is given by
\begin{align} \label{eq-special_term_odd1}
 D & = \left( \prod_{\kappa} L_{\kappa} \right) \cdot \left( \prod_{\mu < \nu} \det( \mathbf{L}_{\mu \nu}  ) \right) \, .
\end{align}

Once conditions \eqref{eq-integrability_connection_components_odd1} have been established, we will be able to find a system of linear equations of the form
\begin{align} \label{mat-GS_odd2}
  \begin{pmatrix}
       M_{1 2} & * & \cdots & & \cdots & * \\
	* & M_{1 3} & & & & \vdots \\
	\vdots & & \ddots & & & \\
	 & &  & M_{\mu \nu} & & \vdots \\
	\vdots & &  &  & \ddots & * \\
	* & \cdots &  & \cdots & * & M_{m-1,m}
  \end{pmatrix}
\begin{pmatrix}
  \Gamma_{0 1 2} \\
  \Gamma_{0 1 3} \\
  \vdots \\
  \Gamma_{0 \mu \nu} \\
  \vdots \\
  \Gamma_{0, m-1, m}
  \end{pmatrix}
& = 
\begin{pmatrix}
  0 \\
  0 \\
  \vdots \\
  0 \\
  \vdots \\
  0
  \end{pmatrix} \, ,
\end{align}
or $\mathbf{M} \mathbf{w} = \mathbf{0}$ for short. Again, the diagonal entries $M_{\mu \nu}$ for all $\mu$, $\nu$ have been singled out for their crucial part in the application of Lemma \ref{lem-det_matrix}, where the term $D$  of the determinant of matrix \eqref{mat-GS_odd2} is now given by
\begin{align} \label{eq-special_term_odd2}
 D & = \prod_{\mu < \nu} M_{\mu \nu} \, .
\end{align}

\paragraph{Case $k=0$:} Assume the Weyl tensor is a section of $\mcC^{0}$ so that conditions \eqref{eq-C0_components} hold together with
\begin{align}\label{eq-C0_components_odd}
 C_{\mu \nu \kappa 0} = C_{\mu \tilde{\nu} \kappa 0} & = 0  \, ,
\end{align}
for all $\mu$, $\nu$, $\kappa$, $\lambda$. Then equations \eqref{eq-Bianchi3a}, \eqref{eq-Bianchi3b}, \eqref{eq-Bianchi2a} and \eqref{eq-Bianchi2b} become
\begin{align}
0 & = 2 g \ind{_{\tilde{\rho} \lb{\kappa}}} A \ind{_{\rb{\lambda} \mu \nu}} + 2 \Gamma \ind{_{[\mu \nu]} ^{\tilde{\sigma}}} C \ind{_{\tilde{\rho} {\tilde{\sigma}} \kappa \lambda}} + 4 \Gamma \ind{_{\lb{\mu}| \lb{\kappa}} ^{\tilde{\sigma}}} C \ind{_{\rb{\lambda}|{\tilde{\sigma}}| \rb{\nu} \tilde{\rho}}} =: B \ind{_{\mu \nu \tilde{\rho} | \kappa \lambda}} \, , \tag{\ref{eq-Bianchi3a}}  \\
0 & = - A \ind{_{\kappa \mu \nu}} + 2 \Gamma \ind{_{[\mu \nu]} ^{\tilde{\sigma}}} C \ind{_{0 {\tilde{\sigma}} \kappa 0}} 
+ 2 \Gamma \ind{_{\lb{\mu}| \kappa} ^{\tilde{\sigma}}} C \ind{_{0 {\tilde{\sigma}}| \rb{\nu} 0}} =: B \ind{_{\mu \nu 0 | \kappa 0}} \, , \tag{\ref{eq-Bianchi3b}} \\
0 & = g \ind{_{\tilde{\rho} \kappa}} A \ind{_{0 \mu \nu}} + 2 \Gamma \ind{_{[\mu \nu]} ^{\tilde{\sigma}}} C \ind{_{\tilde{\rho} {\tilde{\sigma}} \kappa 0}} + 2 \Gamma \ind{_{[\mu \nu]} ^0} C \ind{_{\tilde{\rho} 0 \kappa 0}} 
  + 2 \Gamma \ind{_{\lb{\mu}| \kappa} ^{\tilde{\sigma}}} C \ind{_{0  {\tilde{\sigma}} | \rb{\nu} \tilde{\rho}}} - 2 \Gamma \ind{_{\lb{\mu}| 0} ^{\tilde{\sigma}}} C \ind{_{\kappa {\tilde{\sigma}} | \rb{\nu} \tilde{\rho}}} =: B \ind{_{\mu \nu \tilde{\rho} | \kappa 0}} \, , \tag{\ref{eq-Bianchi2a}} \\
0 & = 2 g \ind{_{\lb{\mu} | \tilde{\lambda}}} A \ind{_{\kappa | \rb{\nu} 0}} + 2 \Gamma \ind{_{[\mu \nu]} ^{\tilde{\sigma}}} C \ind{_{0 {\tilde{\sigma}} \kappa \tilde{\lambda}}} + 2 \Gamma \ind{_{0 \lb{\mu}} ^{\tilde{\sigma}}} C \ind{_{\rb{\nu} {\tilde{\sigma}} \kappa \tilde{\lambda}}}  + 2 \Gamma \ind{_{\lb{\nu}| 0|} ^{\tilde{\sigma}}} C \ind{_{\rb{\mu} {\tilde{\sigma}} \kappa \tilde{\lambda}}} \nonumber \\  & \qquad \qquad \qquad \qquad \qquad + 2 \Gamma \ind{_{\lb{\mu} | \kappa} ^{\tilde{\sigma}}} C \ind{_{\tilde{\lambda}{\tilde{\sigma}}| \rb{\nu} 0}} + 2 \Gamma \ind{_{\lb{\mu}| \kappa} ^0} C \ind{_{\tilde{\lambda} 0 | \rb{\nu} 0}} + \Gamma \ind{_{0 \kappa} ^{\tilde{\sigma}}} C \ind{_{\tilde{\lambda} {\tilde{\sigma}} \mu \nu}} =: B \ind{_{\mu \nu 0 | \kappa \tilde{\lambda}}} \, . \tag{\ref{eq-Bianchi2b}}
\end{align}
Now, suppose that the Weyl tensor is otherwise generic and the Cotton-York tensor is a section of $\mcA^{-1}$. Then, we see that equation \eqref{eq-Bianchi3a} is identical to the one used in the proof of case $k=0$ of Theorem \ref{thm-GS}. Hence, conditions \eqref{eq-integrability_connection_components} are established. We also note that the same result can be equally derived from equations \eqref{eq-Bianchi3b}.

Consequently, equations \eqref{eq-Bianchi2a} depend only on the connection components $\Gamma _{\mu \nu 0}$ for all $\mu$, $\nu$. Now, pick $m$ equations from the $m(m-1)$ equations $B_{\mu \nu \tilde{\nu} | \mu 0}$, and all $m(m-1)$ equations $B_{\mu \nu \tilde{\mu} | \mu 0}$:
\begin{multline*}
0 = 
 \Gamma _{\mu \mu 0} C _{\mu \tilde{\mu} \nu \tilde{\nu}} 
- \Gamma _{\nu \nu 0} C _{\mu \tilde{\nu} \mu \tilde{\nu}}
+ \Gamma _{\mu \nu 0} ( C _{\mu \tilde{\nu} \nu \tilde{\nu}} + C _{\tilde{\nu} 0 \mu 0} )
- \Gamma _{\nu \mu 0} ( C _{\mu \tilde{\mu} \mu \tilde{\nu}} + C _{\tilde{\nu} 0 \mu 0} ) \\
+ \sum_{\sigma \neq \mu, \nu} 
 \left( \Gamma _{\mu \sigma 0} C _{\mu \tilde{\sigma} \nu \tilde{\nu}} - \Gamma _{\nu \sigma 0} C _{\mu \tilde{\sigma} \mu \tilde{\nu}} \right) \, ,
\end{multline*}
\begin{multline*}
0 = \Gamma _{\mu \mu 0} C _{\mu \tilde{\mu} \nu \tilde{\mu}} 
- \Gamma _{\nu \nu 0} C _{\mu \tilde{\nu} \mu \tilde{\mu}}   
+ \Gamma _{\mu \nu 0} ( C _{\mu \tilde{\nu} \nu \tilde{\mu}} + C _{\tilde{\mu} 0 \mu 0} )  
- \Gamma _{\nu \mu 0} ( C _{\mu \tilde{\mu} \mu \tilde{\mu}} + C _{\tilde{\mu} 0 \mu 0} ) \\
+ \sum_{\sigma \neq \mu, \nu} \left( \Gamma _{\mu \sigma 0} C _{\mu \tilde{\sigma} \nu \tilde{\mu}} - \Gamma _{\nu \sigma 0} C _{\mu \tilde{\sigma} \mu \tilde{\mu}} \right)
\, ,
\end{multline*}
respectively. These equations can be put into the matrix form \eqref{mat-GS_odd1} by defining
\begin{align*}
 L_{\mu} & : =
\begin{cases}
 C _{\mu \tilde{\mu} m \tilde{m}} \, , & \mbox{for $\mu \neq m$} \, , \\
 C _{m, \tilde{m}, m-1, \widetilde{m-1}} \, , & \mbox{for $\mu = m$} \, ,
\end{cases} \\
 \mathbf{L}_{\mu \nu} & := \begin{pmatrix}
  C _{\mu \tilde{\nu} \nu \tilde{\mu}} + C _{\tilde{\mu} 0 \mu 0} &  C _{\mu \tilde{\mu} \mu \tilde{\mu}} + C _{\tilde{\mu} 0 \mu 0}  \\
  C _{\nu \tilde{\mu} \mu \tilde{\nu}} + C _{\tilde{\nu} 0 \nu 0} &  C _{\nu \tilde{\nu} \nu \tilde{\nu}} + C _{\tilde{\nu} 0 \nu 0}  
 \end{pmatrix} \, , & \mbox{for all $\mu <\nu$,}
\end{align*}
respectively. Now, each matrix $\mathbf{L}_{\mu \nu}$ has determinant, for all $\mu < \nu$,
\begin{align*}
 \det (\mathbf{L}_{\mu \nu} ) = ( C _{\mu \tilde{\nu} \nu \tilde{\mu}} + C _{\tilde{\mu} 0 \mu 0} ) \cdot ( C _{\nu \tilde{\nu} \nu \tilde{\nu}} + C _{\tilde{\nu} 0 \nu 0}  ) -
( C _{\mu \tilde{\mu} \mu \tilde{\mu}} + C _{\tilde{\mu} 0 \mu 0} ) \cdot ( C _{\nu \tilde{\mu} \mu \tilde{\nu}} + C _{\tilde{\nu} 0 \nu 0} ) \, ,
\end{align*}
which can be seen to be non-vanishing by the genericity assumption.\footnote{Again, the tracefree property of the Weyl tensor may seem problematic when $m=2$. But in that case, some manipulations show that $\det (\mathbf{L}_{\mu \nu} ) = 2 C \ind{_{\mu 0 \tilde{\mu} 0}}   \cdot ( C \ind{_{\mu \nu  \tilde{\mu} \tilde{\nu}}} - C \ind{_{\mu \tilde{\nu} \tilde{\mu} \nu}} )$, which is clearly non-vanishing.} Hence, the term \eqref{eq-special_term_odd1} is non-vanishing.

Further, one can check that the components of the Weyl tensor in the entries of $L_\mu$ and $\mathbf{L}_{\mu \nu}$ have distinct index structures from those in the remaining entries of $\mathbf{L}$. Hence, we can apply Lemma \ref{lem-det_matrix} to conclude that $\mathbf{L}$ is non-singular, thus establishing conditions \eqref{eq-integrability_connection_components_odd1}.

At this stage, equations \eqref{eq-Bianchi2b} only constrain the connection components $\Gamma _{0 \mu \nu}$ for all $\mu$, $\nu$. Pick all $\tfrac{1}{2}m(m-1)$ equations $B \ind{_{\mu \nu 0 | \mu \tilde{\mu}}}$:
\begin{align}
0 = \Gamma _{0 \mu \nu} ( C _{\nu \tilde{\nu} \mu \tilde{\mu}} + C _{\tilde{\mu} \tilde{\nu} \mu \nu} + C _{\mu \tilde{\mu} \mu \tilde{\mu}} ) 
+ \sum_{\sigma \neq \mu, \nu} (  \Gamma _{0 \mu \sigma} ( C _{\nu \tilde{\sigma} \mu \tilde{\mu}} + C _{\tilde{\mu} \tilde{\sigma} \mu \nu} )
- \Gamma _{0 \nu \sigma} C _{\mu \tilde{\sigma} \mu \tilde{\mu}} )
\, ,
\end{align}
which can be put in the matrix form \eqref{mat-GS_odd2} by defining, for every $\mu < \nu$,
\begin{align*}
 M_{\mu \nu} := C _{\nu \tilde{\nu} \mu \tilde{\mu}} + C _{\tilde{\mu} \tilde{\nu} \mu \nu} + C _{\mu \tilde{\mu} \mu \tilde{\mu}} \, .
\end{align*}
This is non-vanishing by the genericity assumption, and thus, the term \eqref{eq-special_term_odd2} is non-vanishing.

As in the previous step, the components of the Weyl tensor in the diagonal entries $M_{\mu \nu}$ of $\mathbf{M}$ have distinct index structures from those in the remaining entries of $\mathbf{K}$. Hence, we can apply Lemma \ref{lem-det_matrix} to conclude that $\mathbf{M}$ is non-singular, thus establishing conditions \eqref{eq-integrability_connection_components_odd1}.

\paragraph{Case $k=1$:}
Assume that the Weyl tensor is a section of $\mcC^0$, so that conditions \eqref{eq-C0_components}, \eqref{eq-C0_components_odd} and \eqref{eq-C1_components} hold. Then equations \eqref{eq-Bianchi2a}, \eqref{eq-Bianchi2b}, \eqref{eq-Bianchi1a} and  \eqref{eq-Bianchi1b} become
\begin{align*}
0 & = g \ind{_{\tilde{\rho} \kappa}} A \ind{_{0 \mu \nu}} + 2 \Gamma \ind{_{[\mu \nu]} ^{\tilde{\sigma}}} C \ind{_{\tilde{\rho} {\tilde{\sigma}} \kappa 0}} 
  + 2 \Gamma \ind{_{\lb{\mu}| \kappa} ^{\tilde{\sigma}}} C \ind{_{0  {\tilde{\sigma}} | \rb{\nu} \tilde{\rho}}} =: B_{\mu \nu \tilde{\rho} | \kappa 0} \, , \tag{\ref{eq-Bianchi2a}} \\
0 & = 2 g \ind{_{\lb{\mu} | \tilde{\lambda}}} A \ind{_{\kappa | \rb{\nu} 0}} + 2 \Gamma \ind{_{[\mu \nu]} ^{\tilde{\sigma}}} C \ind{_{0 {\tilde{\sigma}} \kappa \tilde{\lambda}}} + 2 \Gamma \ind{_{\lb{\mu} | \kappa} ^{\tilde{\sigma}}} C \ind{_{\tilde{\lambda}{\tilde{\sigma}}| \rb{\nu} 0}}  =: B_{\mu \nu 0 | \kappa \tilde{\lambda}} \tag{\ref{eq-Bianchi2b}} \, , \\ 
0 & =  - 2 g \ind{_{\lb{\mu} | \tilde{\lambda}}} A \ind{_{\kappa | \rb{\nu} \tilde{\rho}}} + g \ind{_{\tilde{\rho} \kappa}} A \ind{_{\tilde{\lambda} \mu \nu}} 
+ 2 \Gamma \ind{_{[\mu \nu]} ^{\tilde{\sigma}}} C \ind{_{\tilde{\rho} {\tilde{\sigma}} \kappa \tilde{\lambda}}}
+ 2 \Gamma \ind{_{[\mu \nu]} ^0} C \ind{_{\tilde{\rho} 0 \kappa \tilde{\lambda}}} 
+ 2 \Gamma \ind{_{\lb{\mu}| \kappa} ^{\tilde{\sigma}}} C \ind{_{\tilde{\lambda}|{\tilde{\sigma}}| \rb{\nu} \tilde{\rho}}}
+ 2 \Gamma \ind{_{\lb{\mu}| \kappa} ^0} C \ind{_{\tilde{\lambda}|0| \rb{\nu} \tilde{\rho}}} \\
& \qquad \qquad \qquad \qquad \qquad \qquad \qquad \qquad \qquad \qquad \qquad \qquad \qquad \qquad \qquad \qquad \qquad \qquad \quad =: B_{\mu \nu \tilde{\rho} | \kappa \tilde{\lambda}} \, , \tag{\ref{eq-Bianchi1a}} \\
0 &=  g \ind{_{\tilde{\nu} \kappa}} A \ind{_{0 0 \mu}} - A \ind{_{\kappa \mu \tilde{\nu}}} 
+ \Gamma \ind{_{0 \mu} ^{\tilde{\sigma}}} C \ind{_{\tilde{\nu} {\tilde{\sigma}} \kappa 0}} 
- \Gamma \ind{_{\mu 0} ^{\tilde{\sigma}}} C \ind{_{\tilde{\nu} {\tilde{\sigma}} \kappa 0}} 
+ \Gamma \ind{_{\mu \kappa} ^{\tilde{\sigma}}} C \ind{_{0 {\tilde{\sigma}} \tilde{\nu} 0}} 
- \Gamma \ind{_{\mu 0} ^{\tilde{\sigma}}} C \ind{_{\kappa {\tilde{\sigma}} \tilde{\nu} 0}}
+ \Gamma \ind{_{0 \kappa} ^{\tilde{\sigma}}} C \ind{_{0 {\tilde{\sigma}} \mu \tilde{\nu}}} =: B_{\mu \tilde{\nu} 0 | \kappa 0} \tag{\ref{eq-Bianchi1b}} \, .
\end{align*}
Now, suppose that the Weyl tensor is otherwise generic, and the Cotton-York tensor is a section of $\mcA^{0}$. Then, both sets of equations \eqref{eq-Bianchi2a} and \eqref{eq-Bianchi2b} are equations on $\Gamma_{\kappa \mu \nu}$ for all $\kappa$, $\mu$, $\nu$. Now, pick all $m(m-1)$ equations $B_{\mu \nu 0 | \mu \tilde{\nu}}$ and all $\tfrac{1}{2}m(m-1)(m-2)$ equations $B_{\mu \nu 0 | \mu \tilde{\lambda}}$
\begin{align*}
0 & =  - \Gamma \ind{_{\mu \mu \nu}} C \ind{_{0 {\tilde{\mu}} \mu \tilde{\nu}}}
+ \Gamma \ind{_{\nu \nu \mu}} C \ind{_{0 {\tilde{\nu}} \mu \tilde{\nu}}} 
+ \sum_{\sigma \neq \mu, \nu} \left( \Gamma \ind{_{\mu \nu \sigma}} C \ind{_{0 {\tilde{\sigma}} \mu \tilde{\nu}}}
- \Gamma \ind{_{\nu \mu \sigma}} ( C \ind{_{0 {\tilde{\sigma}} \mu \tilde{\nu}}} + C \ind{_{\tilde{\nu} {\tilde{\sigma}} \mu 0}}   )
+ \Gamma \ind{_{\mu \mu \sigma}} C \ind{_{\tilde{\nu} {\tilde{\sigma}} \nu 0}}  \right) \, , \\
0 & =  - \Gamma \ind{_{\mu \mu \nu}} ( C \ind{_{0 {\tilde{\mu}} \mu \tilde{\lambda}}} - C \ind{_{\tilde{\lambda} {\tilde{\nu}} \nu 0}} )
+ \Gamma \ind{_{\nu \nu \mu}} ( C \ind{_{0 {\tilde{\nu}} \mu \tilde{\lambda}}} + C \ind{_{\tilde{\lambda} {\tilde{\nu}} \mu 0}} )  
+ \Gamma \ind{_{\mu \nu \lambda}} C \ind{_{0 {\tilde{\lambda}} \mu \tilde{\lambda}}}
+ \Gamma \ind{_{\nu \lambda \mu}} C \ind{_{0 {\tilde{\lambda}} \mu \tilde{\lambda}}} \\
& \qquad \qquad \qquad \qquad \qquad \qquad
+ \sum_{\sigma \neq \mu, \nu, \lambda} \left( \Gamma \ind{_{\mu \nu \sigma}} C \ind{_{0 {\tilde{\sigma}} \mu \tilde{\lambda}}}
- \Gamma \ind{_{\nu \mu \sigma}} ( C \ind{_{0 {\tilde{\sigma}} \mu \tilde{\lambda}}} + C \ind{_{\tilde{\lambda} {\tilde{\sigma}} \mu 0}}   )
+ \Gamma \ind{_{\mu \mu \sigma}} C \ind{_{\tilde{\lambda} {\tilde{\sigma}} \nu 0}}  \right) \, ,
\end{align*}
respectively. These can be put in matrix form \eqref{mat-GS_even} by defining, for all $\kappa$, $\rho$ distinct, and $\mu < \nu < \lambda$,
\begin{align*}
 K_{\kappa \rho} & := C \ind{_{0 {\tilde{\kappa}} \rho \tilde{\kappa}}} \, ,
& \mathbf{K}_{\mu \nu \lambda} & :=
\begin{pmatrix}
 C \ind{_{0 {\tilde{\lambda}} \mu \tilde{\lambda}}} & C \ind{_{0 {\tilde{\lambda}} \mu \tilde{\lambda}}}  & 0 \\
 0 &  C \ind{_{0 {\tilde{\mu}} \nu \tilde{\mu}}} & C \ind{_{0 {\tilde{\mu}} \nu \tilde{\mu}}}  \\
 C \ind{_{0 {\tilde{\nu}} \lambda \tilde{\nu}}} & 0 & C \ind{_{0 {\tilde{\nu}} \lambda \tilde{\nu}}}  
\end{pmatrix} \, .
\end{align*}
The latter has determinant
\begin{align*}
 \det ( \mathbf{K}_{\mu \nu \lambda} ) & = 2 C \ind{_{0 {\tilde{\lambda}} \mu \tilde{\lambda}}} \cdot C \ind{_{0 {\tilde{\mu}} \nu \tilde{\mu}}} \cdot C \ind{_{0 {\tilde{\nu}} \lambda \tilde{\nu}}}  \, ,
\end{align*}
which can be seen to be non-vanishing by the genericity assumption. So, $K_{\mu \nu}$ and $\mathbf{K}_{\mu \nu \lambda}$ are non-singular, and thus, the term \eqref{eq-special_term} is non-vanishing.

Further, one can check that the components of the Weyl tensor in the entries of $K_{\mu \nu}$ and $\mathbf{K}_{\kappa \mu \nu}$ have distinct index structures from those in the remaining entries of $\mathbf{K}$. Hence, we can apply Lemma \ref{lem-det_matrix} to conclude that $\mathbf{K}$ is non-singular,  thus establishing conditions \eqref{eq-integrability_connection_components}.

Now, equations \eqref{eq-Bianchi1a} constrain only $\Gamma \ind{_{\mu \nu 0}}$ for all $\mu$, $\nu$.
Pick $m$ equations from the $m(m-1)$ equations $B_{\mu \nu \tilde{\nu} | \mu \tilde{\nu}}$, and all $m(m-1)$ equations $B_{\mu \nu \tilde{\mu} | \mu \tilde{\mu}}$
\begin{align*}
0 & =  
  \Gamma \ind{_{\mu \nu 0}} C \ind{_{\tilde{\nu} 0 \mu \tilde{\nu}}}
- 2 \Gamma \ind{_{\nu \mu 0}} C \ind{_{\tilde{\nu} 0 \mu \tilde{\nu}}} 
+ \Gamma \ind{_{\mu \mu 0}} C \ind{_{\tilde{\nu} 0 \nu \tilde{\nu}}} \, , &
0 & =  
  \Gamma \ind{_{\mu \nu 0}} C \ind{_{\tilde{\mu} 0 \mu \tilde{\mu}}}
- 2 \Gamma \ind{_{\nu \mu 0}} C \ind{_{\tilde{\mu} 0 \mu \tilde{\mu}}}  
+ \Gamma \ind{_{\mu \mu 0}} C \ind{_{\tilde{\mu} 0 \nu \tilde{\mu}}} \, ,
\end{align*}
respectively. These can be put into matrix form \eqref{mat-GS_odd1} by defining
\begin{align*}
 L_\mu & :=
\begin{cases}
  C \ind{_{\tilde{m} 0 m \tilde{m}}} \, , & \mbox{for all $\mu \neq m$,} \\ 
  C \ind{_{\widetilde{m-1}, 0 ,m-1, \widetilde{m-1}}} \, , & \mbox{for $\mu = m$,}
\end{cases} \\
\mathbf{L}_{\mu \nu} & :=
\begin{pmatrix}
 C \ind{_{\tilde{\mu} 0 \mu \tilde{\mu}}} & 2 C \ind{_{\tilde{\mu} 0 \mu \tilde{\mu}}} \\
 2 C \ind{_{\tilde{\nu} 0 \nu \tilde{\nu}}}  & C \ind{_{\tilde{\nu} 0 \nu \tilde{\nu}}} 
\end{pmatrix} \, , & \mbox{for all $\mu < \nu$,}
\end{align*}
respectively. At a glance, we see that each of $L_\mu$ and $\mathbf{L}_{\mu \nu}$, since $\det (\mathbf{L}_{\mu \nu} )= -3 C \ind{_{\tilde{\mu} 0 \mu \tilde{\mu}}} \cdot C \ind{_{\tilde{\nu} 0 \nu \tilde{\nu}}}$, are non-singular by the genericity assumption, and thus, the term \eqref{eq-special_term_odd1} is non-vanishing.

Further, one can check that the components of the Weyl tensor in the entries of $L_\mu$ and $\mathbf{L}_{\mu \nu}$ have distinct index structures from those in the remaining entries of $\mathbf{L}$. Hence, we can apply Lemma \ref{lem-det_matrix} to conclude that $\mathbf{L}$ is non-singular, thus establishing conditions \eqref{eq-integrability_connection_components_odd1}.

Finally, equations \eqref{eq-Bianchi1b} now constrain only $\Gamma \ind{_{0 \mu \nu}}$ for all $\mu$, $\nu$. Pick all $\tfrac{1}{2}m(m-1)$ equations $B_{\mu \tilde{\nu} 0 | \mu 0}$
\begin{align*}
 0 & =
  \Gamma \ind{_{0 \mu \nu}} C \ind{_{0 {\tilde{\nu}} \mu \tilde{\nu}}}
+ \sum_{\sigma \neq \mu, \nu} 
 \Gamma \ind{_{0 \mu \sigma}} \left( C \ind{_{\tilde{\nu} {\tilde{\sigma}} \mu 0}} + C \ind{_{0 {\tilde{\sigma}} \mu \tilde{\nu}}} \right) \, ,
\end{align*}
which can be put into the matrix form \eqref{mat-GS_odd2} by defining, for every $\mu < \nu$,
\begin{align*}
 M_{\mu \nu} & := C \ind{_{0 {\tilde{\nu}} \mu \tilde{\nu}}} \, .
\end{align*}
By the genericity assumption, each $M_{\mu \nu}$ is non-vanishing, and thus, the term \eqref{eq-special_term_odd2} is non-vanishing.

Further, one can check that the components of the Weyl tensor in the diagonal entries $M_{\mu \nu}$ of $\mathbf{M}$ have distinct index structures from those in the remaining entries of $\mathbf{M}$. Hence, we can apply Lemma \ref{lem-det_matrix} to conclude that $\mathbf{M}$ is non-singular, thus establishing conditions \eqref{eq-integrability_connection_components_odd2}.

\paragraph{Case $k=2$:}
Assume the Weyl tensor is a section of $\mcC^{1}$ so that conditions \eqref{eq-C0_components}, \eqref{eq-C0_components_odd} and \eqref{eq-C1_components} hold together with
\begin{align}\label{eq-C2_components_odd}
C_{\mu \tilde{\nu} \tilde{\kappa} 0} & = 0 \, ,
\end{align}
for all $\mu$, $\nu$, $\kappa$. Then equations \eqref{eq-Bianchi1a}, \eqref{eq-Bianchi1b}, \eqref{eq-Bianchi0a}, \eqref{eq-Bianchi0b}, \eqref{eq-Bianchi0c}, and \eqref{eq-Bianchi0d} become
\begin{align*}
0 & = - 2 g \ind{_{\lb{\mu} | \tilde{\lambda}}} A \ind{_{\kappa | \rb{\nu} \tilde{\rho}}} + g \ind{_{\tilde{\rho} \kappa}} A \ind{_{\tilde{\lambda} \mu \nu}} 
+ 2 \Gamma \ind{_{[\mu \nu]} ^{\tilde{\sigma}}} C \ind{_{\tilde{\rho} {\tilde{\sigma}} \kappa \tilde{\lambda}}}
+ 2 \Gamma \ind{_{\lb{\mu}| \kappa} ^{\tilde{\sigma}}} C \ind{_{\tilde{\lambda}|{\tilde{\sigma}}| \rb{\nu} \tilde{\rho}}} =: B_{\mu \nu \tilde{\rho} | \kappa \tilde{\lambda}} \tag{\ref{eq-Bianchi1a}} \, ,\\
0 & =   g \ind{_{\tilde{\nu} \kappa}} A \ind{_{0 0 \mu}} - A \ind{_{\kappa \mu \tilde{\nu}}} 
+ \Gamma \ind{_{\mu \kappa} ^{\tilde{\sigma}}} C \ind{_{0 {\tilde{\sigma}} \tilde{\nu} 0}} =: B_{\mu \tilde{\nu} 0 | \kappa 0} \tag{\ref{eq-Bianchi1b}} \, , \\
0 & =   2 g \ind{_{\lb{\tilde{\mu}} | \kappa}} A \ind{_{0 | \rb{\tilde{\nu}} \rho}} 
- \Gamma \ind{_{\rho 0} ^{\tilde{\sigma}}} C \ind{_{\kappa {\tilde{\sigma}} \tilde{\mu} \tilde{\nu}}} - \Gamma \ind{_{\rho \kappa} ^{\tilde{\sigma}}} C \ind{_{0 {\tilde{\sigma}} \tilde{\mu} \tilde{\nu}}} =: B_{\tilde{\mu} \tilde{\nu} \rho | \kappa 0} \tag{\ref{eq-Bianchi0a}} \, , \\
0 & =   4 g \ind{_{\lb{\tilde{\mu}} | \lb{\kappa}}} A \ind{_{\rb{\lambda} | \rb{\tilde{\nu}} 0}} 
+ 2 \Gamma \ind{_{0 \lb{\kappa}} ^{\tilde{\sigma}}} C \ind{_{\rb{\lambda} {\tilde{\sigma}} \tilde{\mu} \tilde{\nu}}} =: B_{\tilde{\mu} \tilde{\nu} 0 | \kappa \lambda} \tag{\ref{eq-Bianchi0b}} \, , \\
0 & =   2 g \ind{_{\lb{\mu} | \tilde{\kappa}}} A \ind{_{0 | \rb{\nu} \tilde{\rho}}} 
+ 2 \Gamma \ind{_{[\mu \nu]} ^{\tilde{\sigma}}} C \ind{_{\tilde{\rho} {\tilde{\sigma}} \tilde{\kappa} 0}} 
+ 2 \Gamma \ind{_{[\mu \nu]} ^0} C \ind{_{\tilde{\rho} 0 \tilde{\kappa} 0}} 
- 2 \Gamma \ind{_{\lb{\mu}| 0} ^{\tilde{\sigma}}} C \ind{_{\tilde{\kappa}|{\tilde{\sigma}}| \rb{\nu} \tilde{\rho}}}  =: B_{\mu \nu \tilde{\rho} | \tilde{\kappa} 0} \tag{\ref{eq-Bianchi0c}} \, , \\
0 & =   4 g \ind{_{\lb{\mu} | \lb{\tilde{\kappa}}}} A \ind{_{\rb{\tilde{\lambda}} | \rb{\nu} 0}} 
+ 2 \Gamma \ind{_{[\mu \nu]} ^{\tilde{\sigma}}} C \ind{_{0 {\tilde{\sigma}} \tilde{\kappa} \tilde{\lambda}}} 
+ 2 \Gamma \ind{_{0 \lb{\mu}} ^{\tilde{\sigma}}} C \ind{_{\rb{\nu} {\tilde{\sigma}} \tilde{\kappa} \tilde{\lambda}}} 
+ 2 \Gamma \ind{_{\lb{\nu}| 0|} ^{\tilde{\sigma}}} C \ind{_{\rb{\mu} {\tilde{\sigma}} \tilde{\kappa} \tilde{\lambda}}} =: B_{\mu \nu 0 | \tilde{\kappa} \tilde{\lambda}} \tag{\ref{eq-Bianchi0d}} \, .
\end{align*}
Now, suppose the Weyl tensor is otherwise generic, and the Cotton-York tensor is a section of $\mcA^1$. Then, referring to the proof of case $k=1$ of Theorem \ref{thm-GS}, equations \eqref{eq-Bianchi1a} lead immediately to conditions \eqref{eq-integrability_connection_components}. Alternatively, one could use equations \eqref{eq-Bianchi1b}.

Next, equations \eqref{eq-Bianchi0a} are now equations on $\Gamma_{\kappa \mu 0}$ for all $\kappa$, $\mu$. Pick a subset of $m$ equations of $B_{\tilde{\mu} \tilde{\nu} \mu | \nu 0}$ and all $m(m-1)$ equations $B_{\tilde{\mu} \tilde{\nu} \mu | \mu 0}$
\begin{align*}
0 & = \Gamma \ind{_{\mu \mu 0}} C \ind{_{\nu {\tilde{\mu}} \tilde{\mu} \tilde{\nu}}}
+ \Gamma \ind{_{\mu \nu 0}} C \ind{_{\nu {\tilde{\nu}} \tilde{\mu} \tilde{\nu}}}
+ \sum_{\sigma \neq \mu, \nu} \Gamma \ind{_{\mu \sigma 0}} C \ind{_{\nu {\tilde{\sigma}} \tilde{\mu} \tilde{\nu}}} \, , &
0 & = \Gamma \ind{_{\mu \mu 0}} C \ind{_{\mu {\tilde{\mu}} \tilde{\mu} \tilde{\nu}}}
+ \Gamma \ind{_{\mu \nu 0}} C \ind{_{\mu {\tilde{\nu}} \tilde{\mu} \tilde{\nu}}}
+ \sum_{\sigma \neq \mu, \nu} \Gamma \ind{_{\mu \sigma 0}} C \ind{_{\mu {\tilde{\sigma}} \tilde{\mu} \tilde{\nu}}} \, ,
\end{align*}
respectively, which can be put into matrix form \eqref{mat-GS_odd1} by defining
\begin{align*}
 L_{\mu} & := \begin{cases}
               - C \ind{_{m, {\tilde{\mu}}, \tilde{m}, \tilde{\mu}}} \, , & \mbox{for all $\mu \neq m$,} \\
		- C \ind{_{m-1, {\tilde{m}}, \widetilde{m-1}, \tilde{m}}} \, , & \mbox{for $\mu = m$,} 
              \end{cases} \\
\mathbf{L}_{\mu \nu} & :=
\begin{pmatrix}
C \ind{_{\mu {\tilde{\nu}} \tilde{\mu} \tilde{\nu}}} & 0 \\
0 & - C \ind{_{\nu {\tilde{\mu}} \tilde{\nu} \tilde{\mu}}}  
\end{pmatrix}
\, , & \mbox{for all $\mu < \nu$,}
\end{align*}
respectively. Each $\mathbf{L}_{\mu \nu}$ and each $L_\mu$ are obviously non-singular by the genericity assumption. Hence, the term \eqref{eq-special_term} is non-vanishing.

Further, one can check that the components of the Weyl tensor in the entries of $L_\mu$ and $\mathbf{L}_{\mu \nu}$ have distinct index structures from those in the remaining entries of $\mathbf{L}$. Hence, we can apply Lemma \ref{lem-det_matrix} to conclude that $\mathbf{L}$ is non-singular, thus establishing conditions \eqref{eq-integrability_connection_components_odd1}.

Finally, equations \eqref{eq-Bianchi0b} now constrain $\Gamma_{0 \mu \nu}$ for all $\mu$, $\nu$. Pick all $\frac{1}{2}m(m-1)$ equations $B_{\tilde{\mu} \tilde{\nu} 0| \mu \nu}$
\begin{align*}
0  & = 
\Gamma \ind{_{0 \mu \nu}} ( C \ind{_{\mu {\tilde{\mu}} \tilde{\mu} \tilde{\nu}}} + C \ind{_{\nu {\tilde{\nu}} \tilde{\mu} \tilde{\nu}}} )
+ \sum_{\sigma \neq \mu, \nu} ( \Gamma \ind{_{0 \mu \sigma}} C \ind{_{\nu {\tilde{\sigma}} \tilde{\mu} \tilde{\nu}}} - \Gamma \ind{_{0 \nu \sigma}} C \ind{_{{\mu \tilde{\sigma}} \tilde{\mu} \tilde{\nu}}} )
\end{align*}
which can be put into matrix form \eqref{mat-GS_odd2} by defining, for all $\mu < \nu$,
\begin{align*}
 M_{\mu \nu} & := C \ind{_{\mu {\tilde{\mu}} \tilde{\mu} \tilde{\nu}}} + C \ind{_{\nu {\tilde{\nu}} \tilde{\mu} \tilde{\nu}}} \, .
\end{align*}
which is non-vanishing by the genericity assumption. Hence, the term \eqref{eq-special_term_odd2} is non-vanishing.

Unlike the two previous steps, an issue regarding the index structures of the components of the Weyl tensor arises. Indeed, components of the form $C \ind{_{\mu {\tilde{\mu}} \tilde{\mu} \tilde{\nu}}}$ for all $\mu$, $\nu$ occur in both $M_{\mu \nu}$ and the remaining entries of $\mathbf{M}$. This can be seen from the tracefree property of the Weyl tensor,
\begin{align*}
 C \ind{_{\nu {\tilde{\sigma}} \tilde{\mu} \tilde{\nu}}} & = ( C \ind{_{\mu {\tilde{\mu}} \tilde{\mu} \tilde{\sigma}}} - C \ind{_{\sigma {\tilde{\sigma}} \tilde{\mu} \tilde{\sigma}}} ) - C \ind{_{\tilde{\nu} {\tilde{\sigma}} \tilde{\mu} \nu}} - \sum_{\rho \neq \mu, \nu, \sigma} ( C \ind{_{\rho {\tilde{\sigma}} \tilde{\mu} \tilde{\rho}}} - C \ind{_{\tilde{\rho} {\tilde{\sigma}} \tilde{\mu} \rho}} ) - C \ind{_{0 {\tilde{\sigma}} \tilde{\mu} 0}} \, ,
\end{align*}
for all $\mu$, $\nu$, $\sigma$. This would thus preclude the application of Lemma \ref{lem-det_matrix} as we have previously used it. However, defining new index structures by setting $C_{\tilde{\mu} \tilde{\nu} +} := C \ind{_{\mu {\tilde{\mu}} \tilde{\mu} \tilde{\nu}}} + C \ind{_{\nu {\tilde{\nu}} \tilde{\mu} \tilde{\nu}}}$ and $C_{\tilde{\mu} \tilde{\nu} -} := C \ind{_{\mu {\tilde{\mu}} \tilde{\mu} \tilde{\nu}}} - C \ind{_{\nu {\tilde{\nu}} \tilde{\mu} \tilde{\nu}}}$ for all $\mu$, $\nu$, removes this ambiguity. The hypotheses of Lemma \ref{lem-det_matrix} are now fulfilled, and yield the desired result, i.e conditions \eqref{eq-integrability_connection_components_odd2}.

\paragraph{Case $k=3$:}
Assume the Weyl tensor is a section of $\mcC^3$, so that conditions \eqref{eq-C0_components}, \eqref{eq-C0_components_odd}, \eqref{eq-C1_components}, \eqref{eq-C2_components_odd} and \eqref{eq-C2_components} hold. Then equations \eqref{eq-Bianchi0c}, \eqref{eq-Bianchi0d}, \eqref{eq-Bianchi-1a}, and \eqref{eq-Bianchi-1b} become
\begin{align*}
0 & =  4 g \ind{_{\lb{\mu} | \lb{\tilde{\kappa}}}} A \ind{_{\rb{\tilde{\lambda}} | \rb{\nu} 0}} 
+ 2 \Gamma \ind{_{[\mu \nu]} ^{\tilde{\sigma}}} C \ind{_{0 {\tilde{\sigma}} \tilde{\kappa} \tilde{\lambda}}} =: B_{\mu \nu 0 |\tilde{\kappa} \tilde{\lambda}} \tag{\ref{eq-Bianchi0c}} \, , \\
0 & =  2 g \ind{_{\lb{\mu} | \tilde{\kappa}}} A \ind{_{0 | \rb{\nu} \tilde{\rho}}} 
+ 2 \Gamma \ind{_{[\mu \nu]} ^{\tilde{\sigma}}} C \ind{_{\tilde{\rho} {\tilde{\sigma}} \tilde{\kappa} 0}} 
=: B_{\mu \nu \tilde{\rho} | \tilde{\kappa} 0} \tag{\ref{eq-Bianchi0d}} \, , \\
0 &=  - 2 g \ind{_{\lb{\tilde{\mu}} | \lambda}} A \ind{_{\tilde{\kappa} | \rb{\tilde{\nu}} \rho}} 
+ g \ind{_{\rho \tilde{\kappa}}} A \ind{_{\lambda \tilde{\mu} \tilde{\nu}}} 
- \Gamma \ind{_{\rho \lambda} ^{\tilde{\sigma}}} C \ind{_{\tilde{\kappa} {\tilde{\sigma}} \tilde{\mu} \tilde{\nu}}}
- \Gamma \ind{_{\rho \lambda} ^0} C \ind{_{\tilde{\kappa} 0 \tilde{\mu} \tilde{\nu}}} =: B_{\tilde{\mu} \tilde{\nu} \rho | \tilde{\kappa} \lambda} \tag{\ref{eq-Bianchi-1a}} \, , \\
0 & =  g \ind{_{\nu \tilde{\kappa}}} A \ind{_{0 0 \tilde{\mu}}} - A \ind{_{\tilde{\kappa} \tilde{\mu} \nu}} 
+ \Gamma \ind{_{\nu 0} ^{\tilde{\sigma}}} C \ind{_{\tilde{\mu} {\tilde{\sigma}} \tilde{\kappa} 0}} 
- \Gamma \ind{_{0 \nu} ^{\tilde{\sigma}}} C \ind{_{\tilde{\mu} {\tilde{\sigma}} \tilde{\kappa} 0}} 
- \Gamma \ind{_{\nu 0} ^{\tilde{\sigma}}} C \ind{_{\tilde{\kappa} {\tilde{\sigma}} 0 \tilde{\mu}}}  
   =: B_{\tilde{\mu} \nu 0 | \tilde{\kappa} 0} \, . \tag{\ref{eq-Bianchi-1b}}
\end{align*}
Now, suppose the Weyl tensor is otherwise generic, and the Cotton-York tensor is a section of $\mcA^{2}$. Equations \eqref{eq-Bianchi0c} and \eqref{eq-Bianchi0d} constrain $\Gamma_{\mu \nu \kappa}$ for all $\kappa$, $\mu$, $\nu$. Pick all $m(m-1)$ equations $B_{\mu \nu \tilde{\mu} | \tilde{\nu} 0}$ and all $\frac{1}{2}m(m-1)(m-2)$ equations $B_{\mu \nu \tilde{\mu} | \tilde{\kappa} 0}$
\begin{align*}
0 & =
\Gamma \ind{_{\nu \nu \mu}} C \ind{_{\tilde{\mu} {\tilde{\nu}} \tilde{\nu} 0}}
+ \sum_{\sigma \neq \mu, \nu } ( \Gamma \ind{_{\mu \nu \sigma}} C \ind{_{\tilde{\mu} {\tilde{\sigma}} \tilde{\nu} 0}} - \Gamma \ind{_{\nu \mu \sigma}} C \ind{_{\tilde{\mu} {\tilde{\sigma}} \tilde{\nu} 0}} ) \, , \\
0 & =
\Gamma \ind{_{\nu \nu \mu}} C \ind{_{\tilde{\mu} {\tilde{\nu}} \tilde{\kappa} 0}}
+ \Gamma \ind{_{\mu \nu \kappa}} C \ind{_{\tilde{\mu} {\tilde{\kappa}} \tilde{\kappa} 0}}
+ \Gamma \ind{_{\nu \kappa \mu}} C \ind{_{\tilde{\mu} {\tilde{\kappa}} \tilde{\kappa} 0}}
+ \sum_{\sigma \neq \mu, \nu, \kappa} ( \Gamma \ind{_{\mu \nu \sigma}} C \ind{_{\tilde{\mu} {\tilde{\sigma}} \tilde{\kappa} 0}} - \Gamma \ind{_{\nu \mu \sigma}} C \ind{_{\tilde{\mu} {\tilde{\sigma}} \tilde{\kappa} 0}} ) \, ,
\end{align*}
which can be put in matrix form \eqref{mat-GS_even} by defining, for all $\kappa$, $\rho$ distinct, and $\mu < \nu < \lambda$,
\begin{align*}
K_{\kappa \rho} & :=  C \ind{_{\tilde{\rho} {\tilde{\kappa}} \tilde{\kappa} 0}} \, , &
\mathbf{K}_{\mu \nu \kappa} & :=
 \begin{pmatrix}
  C \ind{_{\tilde{\mu} {\tilde{\kappa}} \tilde{\kappa} 0}} & C \ind{_{\tilde{\mu} {\tilde{\kappa}} \tilde{\kappa} 0}} & 0 \\
  0 & C \ind{_{\tilde{\nu} {\tilde{\mu}} \tilde{\mu} 0}} & C \ind{_{\tilde{\nu} {\tilde{\mu}} \tilde{\mu} 0}} \\
  C \ind{_{\tilde{\kappa} {\tilde{\nu}} \tilde{\nu} 0}}  & 0 & C \ind{_{\tilde{\kappa} {\tilde{\nu}} \tilde{\nu} 0}} \\
 \end{pmatrix} \, ,
\end{align*}
respectively. That these matrices are non-singular by the genericity assumption is clear, and so the term \eqref{eq-special_term} is non-vanishing.

Further, one can check that the components of the Weyl tensor in the entries of $K_{\mu \nu}$ and $\mathbf{K}_{\kappa \mu \nu}$ have distinct index structures from those in the remaining entries of $\mathbf{K}$. Hence, we can apply Lemma \ref{lem-det_matrix} to conclude that $\mathbf{K}$ is non-singular.

Next, equations \eqref{eq-Bianchi-1a} now constrain $\Gamma_{\mu \nu 0}$ for all $\mu, \nu$, and conditions \eqref{eq-integrability_connection_components_odd1} follow directly from any subsets of $m^2$ equations
\begin{align*}
0 &=  - \Gamma \ind{_{\rho \lambda 0}} C \ind{_{\tilde{\kappa} 0 \tilde{\mu} \tilde{\nu}}} \, .
\end{align*}

Finally, equations \eqref{eq-Bianchi-1b} now only constrain $\Gamma_{0 \nu \kappa}$ for all $\kappa$, $\nu$. Pick $\frac{1}{2}m(m-1)$ equations $B_{\tilde{\mu} \mu 0 | \tilde{\kappa} 0}$
\begin{align*}
0 & = 
- \Gamma \ind{_{0 \mu \kappa}} C \ind{_{\tilde{\mu} {\tilde{\kappa}} \tilde{\kappa} 0}} 
- \sum_{\sigma \neq \mu, \nu, \kappa} \Gamma \ind{_{0 \mu \sigma}} C \ind{_{\tilde{\mu} {\tilde{\sigma}} \tilde{\kappa} 0}} \, ,
\end{align*}
which can be put in matrix form \eqref{mat-GS_odd1} by defining, for all $\mu < \kappa$,
\begin{align*}
 M_{\mu \kappa} := - C \ind{_{\tilde{\mu} {\tilde{\kappa}} \tilde{\kappa} 0}} \, .
\end{align*}
This is non-vanishing by the genericity assumption, and thus the term \eqref{eq-special_term_odd1} is non-vanishing too.

Further, one can check that the components of the Weyl tensor in the diagonal entries $M_{\mu \nu}$ of $\mathbf{M}$ have distinct index structures from those in the remaining entries of $\mathbf{M}$. Hence, we can apply Lemma \ref{lem-det_matrix} to conclude that $\mathbf{M}$ is non-singular, thus establishing conditions \eqref{eq-integrability_connection_components_odd1}

\paragraph{Case $k=4$:}
Assume that the Weyl tensor is a section of $\mcC^4$, so that conditions \eqref{eq-C0_components}, \eqref{eq-C0_components_odd}, \eqref{eq-C1_components}, \eqref{eq-C2_components_odd} and \eqref{eq-C2_components} hold together with
\begin{align} \label{eq-C4_components_odd}
 C \ind{_{\tilde{\kappa} 0 \tilde{\mu} \tilde{\nu}}} & = 0 \, ,
\end{align}
for all $\kappa$, $\mu$, $\nu$. Then equations \eqref{eq-Bianchi-1a}, \eqref{eq-Bianchi-1b}, \eqref{eq-Bianchi-2a}, and \eqref{eq-Bianchi-2b} become
\begin{align*}
0 & =  g \ind{_{\nu \tilde{\kappa}}} A \ind{_{0 0 \tilde{\mu}}} - A \ind{_{\tilde{\kappa} \tilde{\mu} \nu}} =: B_{\tilde{\mu} \nu 0 | \tilde{\kappa} 0} \tag{\ref{eq-Bianchi-1a}} \, , \\
0 &=  - 2 g \ind{_{\lb{\tilde{\mu}} | \lambda}} A \ind{_{\tilde{\kappa} | \rb{\tilde{\nu}} \rho}} + g \ind{_{\rho \tilde{\kappa}}} A \ind{_{\lambda \tilde{\mu} \tilde{\nu}}} 
- \Gamma \ind{_{\rho \lambda} ^{\tilde{\sigma}}} C \ind{_{\tilde{\kappa} {\tilde{\sigma}} \tilde{\mu} \tilde{\nu}}} =: B_{\tilde{\mu} \nu \rho | \tilde{\kappa} \tilde{\lambda}} \tag{\ref{eq-Bianchi-1b}} \, , \\
 0 & =  g \ind{_{\rho \tilde{\kappa}}} A \ind{_{0 \tilde{\mu} \tilde{\nu}}}
- \Gamma \ind{_{\rho 0} ^{\tilde{\sigma}}} C \ind{_{\tilde{\kappa} {\tilde{\sigma}}  \tilde{\mu} \tilde{\nu}}} =: B_{\tilde{\mu} \tilde{\nu} \rho | \tilde{\kappa} 0} \tag{\ref{eq-Bianchi-2a}} \, , \\
0 & =  2 g \ind{_{\lb{\tilde{\mu}} | \lambda}} A \ind{_{\tilde{\kappa} | \rb{\tilde{\nu}} 0}}
- \Gamma \ind{_{0 \lambda} ^{\tilde{\sigma}}} C \ind{_{\tilde{\kappa} {\tilde{\sigma}} \tilde{\mu} \tilde{\nu}}} =: B_{\tilde{\mu} \tilde{\nu} 0 | \tilde{\kappa} \lambda} \tag{\ref{eq-Bianchi-2b}} \, .
\end{align*}
Now, suppose that the Weyl tensor is otherwise generic, and the Cotton-York tensor is a section of $\mcA^{3}$. Referring to the proof of case $k=2$ of Theorem \ref{thm-GS}, equations \eqref{eq-Bianchi-1b} leads immediately to conditions \eqref{eq-integrability_connection_components}.

Next, equations \eqref{eq-Bianchi-2a} now constrain only $\Gamma_{\kappa \mu 0}$ for all $\kappa$, $\mu$. Choose $m$ equations of $B_{\tilde{\mu} \tilde{\nu} \nu | \tilde{\mu} 0}$ and all $m(m-1)$ equations $B_{\tilde{\mu} \tilde{\nu} \mu | \tilde{\mu} 0}$
\begin{align*}
 0  & = 
  \Gamma \ind{_{\nu \nu 0}} C \ind{_{\tilde{\mu} {\tilde{\nu}}  \tilde{\mu} \tilde{\nu}}} 
+ \sum_{\sigma \neq \mu, \nu, \rho} \Gamma \ind{_{\rho \sigma 0}} C \ind{_{\tilde{\mu} {\tilde{\sigma}}  \tilde{\mu} \tilde{\nu}}} \, , &
 0  & = 
 \Gamma \ind{_{\mu \nu 0}} C \ind{_{\tilde{\mu} {\tilde{\nu}}  \tilde{\mu} \tilde{\nu}}} 
+ \sum_{\sigma \neq \mu, \nu, \rho} \Gamma \ind{_{\rho \sigma 0}} C \ind{_{\tilde{\mu} {\tilde{\sigma}}  \tilde{\mu} \tilde{\nu}}} \, ,
\end{align*}
respectively. These can be put into the matrix form \eqref{mat-GS_odd1} by defining
\begin{align*}
 L_\nu & := \begin{cases}
             C \ind{_{\tilde{m} {\tilde{\nu}}  \tilde{m} \tilde{\nu}}} \, , & \mbox{for all $\nu \neq m$,} \\
             C \ind{_{\widetilde{m-1} {\tilde{m}}  \widetilde{m-1} \tilde{m}}} \, , & \mbox{for $\nu = m$,} \\
            \end{cases} \\
 \mathbf{L}_{\mu \nu} & :=
\begin{pmatrix}
  C \ind{_{\tilde{\mu} {\tilde{\nu}}  \tilde{\mu} \tilde{\nu}}} & 0 \\
0 & - C \ind{_{\tilde{\nu} {\tilde{\mu}}  \tilde{\nu} \tilde{\mu}}} 
\end{pmatrix} \, , &  \mbox{for all $\mu < \nu$,}
\end{align*}
respectively, and $\det (\mathbf{L}_{\mu \nu} ) = - ( C \ind{_{\tilde{\mu} {\tilde{\nu}}  \tilde{\mu} \tilde{\nu}}} )^2$. These are clearly non-singular by the genericity assumption, and thus, the term \eqref{eq-special_term_odd1} is non-vanishing.

Further, one can check that the components of the Weyl tensor in the entries of $L_\mu$ and $\mathbf{L}_{\mu \nu}$ have distinct index structures from those in the remaining entries of $\mathbf{L}$. Hence, we can apply Lemma \ref{lem-det_matrix} to conclude that $\mathbf{L}$ is non-singular, thus establishing conditions \eqref{eq-integrability_connection_components_odd1}.

Finally, equations \eqref{eq-Bianchi-2b} now constrain $\Gamma \ind{_{0 \mu \nu}}$ only. Pick all $\tfrac{1}{2}m(m-1)$ equations $B_{\tilde{\mu} \tilde{\nu} 0 | \tilde{\mu} \mu}$
\begin{align*}
 0 & =  
- \Gamma \ind{_{0 \mu \nu}} C \ind{_{\tilde{\mu} {\tilde{\nu}} \tilde{\mu} \tilde{\nu}}} 
- \sum_{\sigma \neq \mu, \nu} \Gamma \ind{_{0 \mu \sigma}} C \ind{_{\tilde{\mu} {\tilde{\sigma}} \tilde{\mu} \tilde{\nu}}} \, ,
\end{align*}
which can be put in matrix form \eqref{mat-GS_odd2} by defining, for all $\mu < \nu$,
\begin{align*}
 M_{\mu \nu} & := - C \ind{_{\tilde{\mu} {\tilde{\nu}} \tilde{\mu} \tilde{\nu}}} \, .
\end{align*}
The genericity assumption tells us that this is non-vanishing non-vanishing, and so, the term \eqref{eq-special_term_odd2} is non-vanishing.

Further, one can check that the components of the Weyl tensor in the entries of $M_{\mu \nu}$ have distinct index structures from those in the remaining entries of $\mathbf{M}$. Hence, we can apply Lemma \ref{lem-det_matrix} to conclude that $\mathbf{M}$ is non-singular, and condition \eqref{eq-integrability_connection_components_odd2} holds true.
\end{proof}

\subsection{Conformal invariance}\label{sec-conformal}
An alternative way to relate the algebraic properties of the Weyl tensor and the Cotton-York tensor, and the integrability of the almost null structure as given by the propositions and theorems above is to consider a conformal change of metric
\begin{align}\label{eq-conformal_change}
 \hat{\bm{g}} = \Omega^2 \bm{g}
\end{align}
for some non-vanishing holomorphic function $\Omega : \mcM \rightarrow \C$. First, it is clear that under such a change, the defining property of the null distribution and its orthogonal complement is invariant, i.e. $\hat{\bm{g}}$ is degenerate on $\mcN$ if and only if $\bm{g}$ is. Further, since the involutivity of these distributions does not depend on the metric, we obtain
\begin{lem}
 The integrability of an almost null structure $\mcN$ is a conformally invariant property.
\end{lem}
It then comes as no surprise that the integrability condition of $\mcN$ given in Proposition \ref{prop-integrability_condition} is itself purely conformal since it only involves the conformally invariant Weyl tensor.

Now define $\bm{\Upsilon} := \bm{g} ( \Omega^{-1} \nabla \Omega )$ and denote by $\hat{\nabla}$ the Levi-Civita connection of $\hat{\bm{g}}$. If $\hat{\bm{A}}$ denotes the Cotton-York tensor of $\hat{\nabla}$, it is well-known (see e.g. \cite{Gover2010} and references therein) that under the conformal change \eqref{eq-conformal_change} the Cotton-York tensor transforms as
\begin{align}\label{eq-CY_conformal}
 \hat{\bm{A}} ( \bm{X}, \bm{Y}, \bm{Z} ) & = \bm{A} ( \bm{X}, \bm{Y}, \bm{Z} ) - \bm{C} ( \bm{\Upsilon} , \bm{X} , \bm{Y} , \bm{Z} )
\end{align}
for all $\bm{X}, \bm{Y}, \bm{Z} \in \Gamma (\Tgt \mcM)$. In four dimensions, similar statements can be madewith regard to the self-dual and anti-self-dual parts of the Weyl and Cotton-York tensors. By Lemma \ref{lem-type_characterisation}, we can now conclude
\begin{lem}\label{lem-CY_conformal}
 Let $(\mcM , \bm{g})$ be a $(2m+\epsilon)$-dimensional complex Riemannian manifold endowed with an almost null structure $\mcN$, where $\epsilon \in \{ 0 , 1 \}$ and $m\geq 2$. Assume $2m+\epsilon \geq 5$. When $\epsilon=0$, respectively, $\epsilon=1$, if the Weyl tensor is a section of $\mcC^k$ for $k \in \{0,1,2 \}$, respectively for $k \in \{0,1,2,3,4 \}$, then the property that the Cotton-York is a section of $\mcA^{k-\tfrac{1}{2}}$, respectively $\mcA^{k-1}$, is conformally invariant.

Assuming $2m+\epsilon = 4$ and $\mcN$ self-dual, if the self-dual part of the Weyl tensor is a section of ${}^+\mcC^k$ for $k \in \{0,1,2 \}$, then the property that the self-dual part of the Cotton-York is a section of ${}^+ \mcA^{k-\tfrac{1}{2}}$ is conformally invariant.
\end{lem}
It then follows immediately that Propositions \ref{prop-GS} and \ref{prop-GS_odd}, and Theorems \ref{thm-GS4}, \ref{thm-GS} and \ref{thm-GS_odd} are conformally invariant.

\section{Further degeneracy} \label{sec-GS_degenerate}
\subsection{On the genericity assumption of the Weyl tensor}
In general, imposing additional structures on a complex Riemannian manifold $(\mcM, \bm{g})$ and its almost null structure $\mcN$ will make the Weyl tensor degenerate further, and it is therefore important to keep track of the emerging algebraic relations between the components of the Weyl tensor to check whether Theorems \ref{thm-GS} and \ref{thm-GS_odd} remain valid. If one realises that the proofs fail under stronger assumptions, one may still have the option of making a different choice of components, and succeed in proving the assertion of the theorem, although in some cases \cite{Taghavi-Chabert2011}, no such choice may present itself.

Here, we list a number of reasons leading to further degeneracy to the Weyl tensor:
\begin{itemize}
 \item The basic framework of the results of section \ref{sec-GS} is the filtered vector bundles $(\mcC , \{ \mcC^i \} )$ and $(\mcA , \{ \mcA^i \} )$. However, as pointed out in Remarks \ref{rem-rep} and \ref{rem-refine}, it is possible to refine the classifications of the Weyl tensor and Cotton-York tensor by considering the irreducible $\prb$-modules contained in each of the quotient bundles $\mcC^i/\mcC^{i+1}$ and $\mcA^i/\mcA^{i+1}$. In particular, Propositions \ref{prop-GS} and \ref{prop-GS_odd} can certainly be made more precise. Extensions of Theorems \ref{thm-GS} and \ref{thm-GS_odd} in this setting, however, are less straightforward. Such a generalisation would remain invariant under $P$.
 \item One may consider the algebraic degeneracy of the Weyl tensor with respect to more than one almost null structure, in particular, any of the canonical almost null structures defined in section \ref{sec-null_basis}, in which case the discussion ceases to be $P$-invariant. This is a generalisation of the four-dimensional Petrov type D condition.
 \item The discussion can also be extended in a natural way to real smooth pseudo-Riemannian manifolds, in which case the almost null structure must satisfy certain reality conditions. In fact, Theorem \ref{thm-GS4} was initially stated in Lorentzian geometry.
\end{itemize}

Of course, any combinations of these further degenerate cases can be used. We shall leave the first of these considerations for now, although we shall briefly comment on it in section \ref{sec-Robinson_structure}. Instead, we focus on the last two points.

Before we proceed, some scepticism might be expressed as to whether the rather broad genericity assumption used in the proofs of Theorem \ref{thm-GS} and \ref{thm-GS_odd} is reasonable. In other words, one may ask the question: are there any `interesting' real or complex (pseudo-)Riemannian manifolds, whose conformal curvature does not degenerate so much as to make the claims invalid? But the present work has precisely been motivated by the existence of such explicit examples as the Kerr-NUT-AdS metric \cite{Chen2006}, which is, in fact, endowed with multiple null structures and a reality structure, and a certain class of higher-dimensional Kerr-Schild metrics \cite{Taghavi-Chaberta}. On the other hand, it is worth pointing out that those manifolds that do not fall into the `generic' class may well exhibit other geometric structures of interest -- a five-dimensional Lorentzian class of such manifolds is considered in \cite{Gomez-Lobo2011}.

\begin{rem}
 It would also be instructive to derive the Jordan normal forms of the Weyl tensor, regarded as a section of the bundle of endomorphisms of $\bigwedge^\bullet \Tgt^* \mcM$, corresponding to the each of the degeneracy classes $\mcC^k$. In this way, the genericity assumption could become more transparent. Further, the eigenvalue structure would provide necessary conditions for the existence of a null structure by means of curvature invariants \cite{Coley2010}. This being said, it was shown in \cite{Taghavi-Chabert2011}, at least in five dimensions, Lorentzian signature, that the Jordan normal form alone does not determine the algebraic speciality of the Weyl tensor. In fact, the existence of a certain number of null eigenforms, both simple and non-simple, appears to be a crucial factor in that matter.\footnote{One can already see that if the Weyl tensor is a section of $\mcC^0$, then any simple section of $\bigwedge^2 (\mcV^*)^1$, i.e. scalar multiples of $\tilde{\bm{\theta}}^{\tilde{\mu}} \wedge \tilde{\bm{\theta}}^{\tilde{\nu}}$ for all $\mu$, $\nu$, is an eigenform of $\bm{C}$.}
\end{rem}

\subsection{Degeneracy of the Cotton-York tensor}
The content of Propositions \ref{prop-GS} and \ref{prop-GS_odd} is really that the Cotton-York tensor should be regarded as an obstruction to the integrability of an almost null structure when the Weyl tensor is algebraically special with respect to it. In fact, Theorems \ref{thm-GS} and \ref{thm-GS_odd} do not depend on any genericity assumption on the Cotton-York tensor. Thus, one may apply stronger conditions on the Cotton-York tensor without affecting Theorems \ref{thm-GS} and \ref{thm-GS_odd}, i.e. the almost null structure remains integrable. These will in general no longer be conformally invariant by equation \eqref{eq-CY_conformal}.

An interesting issue that arises as a result of further degeneracy of the Cotton-York tensor, such as the Einstein condition, is whether one can deduce that more connection components vanish, i.e. the null structure enjoy further geometric properties, beside integrability, as determined, e.g. by the differential equations \eqref{eq-null_class_even} or \eqref{eq-null_class_odd}. In four dimensions, we know that this is not the case. In higher dimensions, if the Weyl tensor is a generic section of $\mcC^0$, then the integrability of the almost null structure is also all one can deduce. On the other hand, for the other degeneracy classes $\mcC^k$ for $k>0$, the Einstein condition yields algebraic relations between the Weyl tensor components and the connection components. Viewed as a homogeneous system of linear equations on the connection components, it is an open question as to whether these components must also vanish in high enough dimensions -- in low dimensions such a system is underdetermined. It is worth pointing out, however, \cite{Taghavi-Chabert2011} that additional reality conditions on five-dimensional Einstein manifolds do lead, in some instances, to further degeneracy of the connection components.

\subsection{Multiple null structures}\label{sec-multiple_null}
Recall from section \ref{sec-null_basis} that the normal form of the metric (locally) determines $2^m$ canonical almost null structures, the set of which is denoted $\mcB_S$. It is then pertinent to consider the algebraic properties of the Weyl and Cotton-York tensors with respect to any number of almost null structures in $\mcB_S$, and such an approach is clearly no longer $P$-invariant. Nonetheless, for a chosen almost null structure $\mcN$, one may still refer to the Weyl tensor as a section of $\mcC^k$ \emph{with respect to $\mcN$} for some $k$, and similarly for the Cotton-York tensor. In particular, one could apply Propositions \ref{prop-GS} and \ref{prop-GS_odd} repeatedly for any number of distinct almost null structures.

On the other hand, one has to be a little more cautious if one wishes to generalise Theorems \ref{thm-GS} and \ref{thm-GS_odd} in the present context. Indeed, assuming the algebraic degeneracy of the Weyl tensor with respect to two or more almost null structures will violate the genericity condition on the Weyl tensor. There is not enough space for a full treatment of this problem here. Instead, we focus on a generalisation of the four-dimensional Petrov type D condition in the sense that the self-dual part of the Weyl tensor is algebraically special with respect to two distinct self-dual almost null structures, i.e. it can be viewed as a section of ${}^+\mcC^0$ with respect to each of these. A similar definition can be made regarding the anti-self-dual part of the Weyl tensor.

In higher dimensions, the situation is analogous except for matters of self-duality. For clarity, the algebraic conditions on the Weyl tensor and Cotton-York tensor are given explicitly.
\begin{thm}\label{thm-GS_multi}
 Let $(\mcM, \bm{g})$ be a $(2m+\epsilon)$-dimensional complex Riemannian manifold, where $\epsilon \in \{ 0,1 \}$ and $2m+\epsilon \geq 5$. Let $\mcN$ be an almost null structure on $\mcM$, and let $\mcB$ be a subset of $\mcB_S$, the set of all canonical almost null structures on (some open subset of) $\mcM$ as defined in section \ref{sec-null_basis}. Suppose that the Weyl tensor and the Cotton-York tensor (locally) satisfy
\begin{align} \label{eq-multi-degenerate}
 \bm{C} (\bm{X} , \bm{Y} , \bm{Z} , \cdot ) & = 0 \, , & \bm{A} (\bm{Z} , \bm{X} , \bm{Y} ) & = 0 \, ,
\end{align}
respectively, for all $\bm{X}, \bm{Y} \in \Gamma ( \mcN_{M}^\perp )$, and $\bm{Z} \in \Gamma ( \mcN_{M} )$, for all $\mcN_{M} \in \mcB$. Suppose further that the Weyl tensor is otherwise generic. Then the almost null structures in $\mcB$ are (locally) integrable.
\end{thm}

\begin{proof}
 For definiteness, we treat the odd-dimensional case only.  With no loss of generality, we can assume that $\mcN = \mcN_{1 2 \ldots m} \in \mcB$. Let $\mcN_M$ be a canonical almost null structure in $\mcB$ distinct from $\mcN$. In particular, $M = \{ \mu_1, \ldots, \mu_p \} \subset S = \{ 1, 2, \ldots , m \}$, where $\mu_i \neq \mu_j$ for all $i \neq j$. Let $\widetilde{M} = S \setminus M$. Suppose that the Weyl tensor and Cotton-York tensors satisfy conditions \eqref{eq-multi-degenerate} with respect to both $\mcN$ and $\mcN_M$. In particular, in the latter case, we have conditions on the components of the Weyl tensor given by
\begin{align}
 C_{\mu_i \mu_j \mu_k \kappa} =
 C_{\mu_i 0 \mu_k \kappa} =
 C_{\tilde{\nu}_i \mu_j \mu_k \kappa} = C_{\mu_i \mu_j \mu_k \tilde{\kappa}} = C_{\mu_i 0 \mu_k 0} =
 C_{\tilde{\nu}_i 0 \mu_k \kappa} = C_{\mu_i 0 \mu_k \tilde{\kappa}} =
 C_{\tilde{\nu}_i \mu_j \mu_k 0} & = 0 \, , \label{eq-Weyl_condition0} \\
 C_{\tilde{\nu}_i \mu_j \mu_k \tilde{\kappa}} = C_{\mu_i \tilde{\nu}_j \tilde{\nu}_k \kappa} = C_{\tilde{\nu}_i 0 \mu_k 0} & = 0 \, ,
\label{eq-Weyl_condition1}   \\
 C_{\tilde{\nu}_i 0 \mu_k \tilde{\kappa}} = C_{\mu_i 0 \tilde{\nu}_k \kappa} =  C_{\mu_i \tilde{\nu}_j \tilde{\nu}_k 0} =
 C_{\mu_i \tilde{\nu}_j \tilde{\nu}_k \tilde{\kappa}} = C_{\tilde{\nu}_i \tilde{\nu}_j \tilde{\nu}_k \kappa} = C_{\tilde{\nu}_i 0 \tilde{\nu}_k 0} =
 C_{\tilde{\nu}_i 0 \tilde{\nu}_k \tilde{\kappa}} =
C_{\tilde{\nu}_i \tilde{\nu}_j \tilde{\nu}_k \tilde{\kappa}} & = 0 \, , \label{eq-Weyl_condition2}
\end{align}
for all $\mu_i \in M$, and $\tilde{\mu}_j \in \widetilde{M}$, and all $\kappa$. We note that there is some redundancy in the sense that conditions \eqref{eq-Weyl_condition0} are also satisfied by virtue of the algebraic degeneracy of the Weyl tensor with respect to $\mcN$. Further, conditions \eqref{eq-Weyl_condition2} are absent in the proof of Theorem \ref{thm-GS_odd} (case $k=0$). So the only issue that might arise concerns conditions \eqref{eq-Weyl_condition1}. Now, recall that the entries of the matrices $\mathbf{K}_{\mu \nu}$, $\mathbf{K}_{\kappa \mu \nu}$, $L_\mu$, $\mathbf{L}_{\mu \nu}$ and $M_{\mu \nu}$ in the proofs of Theorems \ref{thm-GS} and \ref{thm-GS_odd} are linear combinations of the components
\begin{align} \label{eq-intersection_components}
C_{\mu \nu \tilde{\mu} \tilde{\nu}} \, , & & C_{\mu \tilde{\mu} \nu \tilde{\nu}} \, , & & \mbox{and} & & C_{\mu 0 \tilde{\mu} 0} \, ,
\end{align}
for all $\mu$, $\nu$. By inspection, it is then clear that none of the conditions \eqref{eq-Weyl_condition1} violate the genericity assumption on components \eqref{eq-intersection_components}. Hence, Theorems \ref{thm-GS} and \ref{thm-GS_odd} apply to $\mcN$, i.e. $\mcN$ is integrable.

It now remains to show that $\mcN_M$ is also integrable.\footnote{In general, the fact that $\mcN$ is already integrable will imply that some of connection components obstructing the integrability of $\mcN_M$ will vanish, but this does not affect the argument.} To this end, we note that the two almost null structures $\mcN$ and $\mcN_M$ are interchanged by the symmetry
\begin{align*}
 \mu_i \leftrightarrow \tilde{\mu}_i \, ,
\end{align*}
for all $\mu_i \in \widetilde{M}$. In particular, the components \eqref{eq-intersection_components} remain invariant under this symmetry, and thus, the genericity assumption on these is not violated. We can therefore apply Theorems \ref{thm-GS} and \ref{thm-GS_odd} to conclude that $\mcN_M$ is integrable.

At this stage, since $\mcN_M \in \mcB$ was arbitrary, we can extend the above argument to any number of canonical almost null structures in $\mcB$, which proves the claim of the theorem.
\end{proof}

\begin{rem}
It is well-known that in four dimensions, the maximum number of null structures on a non-conformally flat complex Riemannian manifold is four -- both self-dual and anti-self-dual part of the Weyl tensors are then of type D. One may conjecture whether this upper bound is $2^m$ for a $(2m+\epsilon)$-dimensional manifold, where $\epsilon \in \{ 0 , 1 \}$. In \cite{Taghavi-Chabert2011}, however, a counterexample to the conjecture is presented in five dimensions -- the Myers-Perry black hole with one rotation coefficient has eight null structures. But it is not clear whether this is a feature of odd dimensions only.
\end{rem}

\begin{rem}
 In \cite{Mason2010}, it is shown that the integrability condition for the existence of a conformal Killing-Yano $2$-form $\bm{\phi}$ in normal form with distinct eigenvalues on $(\mcM, \bm{g} )$ is precisely that the Weyl tensor satisfies condition \eqref{eq-multi-degenerate} where $\mcB = \mcB_S$. Further, if the exterior derivative of $\bm{\phi}$ satisfies
\begin{align*}
 \dd \bm{\phi} (\bm{X} , \bm{Y} , \bm{Z} ) & = 0 \, ,
\end{align*}
for all $\bm{X}, \bm{Y} \in \Gamma ( \mcN_M^\perp )$, and $\bm{Z} \in \Gamma ( \mcN_M )$, for all $\mcN_M \in \mcB_S$, then $(\mcM , \bm{g} )$ locally admits $2^m$ null structures. We note that this result makes no assumption on the genericity of the Weyl tensor, and indeed, it is certainly true in the conformally flat case. On the other hand, this suggests that Theorem \ref{thm-GS_multi} together with some additional conditions on the non-vanishing components of the Weyl tensor could provide sufficient conditions for the existence of such a conformal Killing-Yano $2$-form, as in the four-dimensional case \cites{Walker1970,Penrose1986}.
\end{rem}

\subsection{Real versions}
Let $(\mcM , \bm{g})$ be a real $(2m+\epsilon)$-dimensional orientable pseudo-Riemannian smooth manifold, where $\epsilon \in \{ 0,1\}$ and $2m+\epsilon \geq 5$. We now work in the smooth real category. Thus, the tangent bundle $\Tgt \mcM$, the cotangent bundle $\Tgt^* \mcM$, and tensor products thereof, such as the bundle $\mcC$ of Weyl tensors, and the bundle $\mcA$ of Cotton-York tensors, are all smooth real vector bundles.

An almost null structure on $( \mcM , \bm{g} )$ can then be defined as a complex subbundle $\mcN$ of the complexified tangent bundle $\C \otimes \Tgt \mcM$, which is totally null with respect to the complexified metric, and of rank $m$, i.e.
\begin{align}\label{eq-null_structure_real}
 \mcN \subset \mcN^\perp \subset \C \otimes \Tgt \mcM \, .
\end{align}
To clarify the following discussion, we recall that the orthogonal complement of a real subbundle of $\Tgt \mcM$, respectively, a complex subbundle of $\C \otimes \Tgt \mcM$, is taken with respect to the real metric, respectively, the complexified metric. In both cases, it is denoted by a $\cdot ^\perp$.

The complexified tangent bundle is naturally equipped with a reality structure, induced from an involutory complex-conjugation operation $\bar{}: \C \otimes \Tgt \mcM \rightarrow \C \otimes \Tgt \mcM$, which preserves the real metric. This motivates the following definition.

\begin{defn}[\cite{Kopczy'nski1992}] Let $\mcN$ be an almost null structure on $( \mcM , \bm{g} )$. The \emph{real index} $r_p$ of the fiber $\mcN_p$ over a point $p \in \mcM$ is the dimension of the intersection of $\mcN$ and its complex conjugate $\overline{\mcN}$, i.e. $r_p = \dim \mcN_p \cap \overline{\mcN}_p$. If $r_p = r_q$ for any points $p, q$ in some open subset $\mcU$ of $\mcM$, we say that $\mcN$ has (constant) real index $r$ in $\mcU$.
\end{defn}

The signature of the metric imposes restrictions on the possible values of the real index $r$ as made precise by the next lemma.

\begin{lem}[\cite{Kopczy'nski1992}]\label{lem-real_index}
 Let $\mcN$ be an almost null structure on a pseudo-Riemannian manifold $(\mcM, \bm{g})$ where $\bm{g}$ has signature $(k, \ell)$, i.e. ($k$ positive eigenvalues, $\ell$ negative eigenvalues), with $k + \ell = 2m + \epsilon$. Then at any given point $p \in \mcM$, the real index $r_p$ of the fiber $\mcN_p$ must be a non-negative integer such that $r_p \leq \min \{k, \ell\}$ and
\begin{itemize}
 \item $r_p \in \left\{ \min \{k, \ell\} \mod 2 \right\}$ when $\epsilon=0$, and
 \item $r_p \in \left\{ \min \{k, \ell\} \mod 1 \right\}$ when $\epsilon=1$.
\end{itemize}
\end{lem}
Assuming the real index $r$ to be constant in some open subset of $\mcM$, the intersection $\mcN \cap \overline{\mcN}$ gives rise to a complexified \emph{real} totally null subbundle $\mcR$ of the tangent bundle of rank $r$.

In section \ref{sec-geometry}, we have distinguished the concept of integrability and involutivity (or formal integrability), and the Frobenius theorem tells us that these are essentially equivalent. While it may seem that this distinction is thus superfluous in the holomorphic category, in the smooth category, it becomes somewhat ambiguous. If the distributions $\mcN$, $\overline{\mcN}$, $\mcN^\perp$ and $\overline{\mcN}^\perp$ are involutive, so are the real spans of the intersections $\mcN \cap \overline{\mcN}$ and $\mcN^\perp \cap \overline{\mcN}^\perp$. Hence, by the Frobenius theorem, $\mcR$ is integrable, and, following the same arguments of Lemma \eqref{lem-integrability_connection}, is tangent to totally null and geodetic \emph{real} submanifolds of dimension $r$. In addition, each fiber of the vector bundle $\mcN + \overline{\mcN} /\mcN \cap \overline{\mcN}$ is naturally equipped with a complex structure, and the quotient manifold $\mcM / \mcR$ thus acquires the structure of a CR manifold of codimension $r+\epsilon$. However, this CR manifold is in general not embeddable, i.e. its underlying complex structure is involutive, but not integrable. Whether this CR manifold is embeddable or not, we shall nonetheless refer to such a null structure as being integrable. In the real-analytic category, on the other hand, one can simply complexify $\mcM$ and work in the holomorphic category, in which case the embeddability of the CR structure will follow. The involutivity of the complex conjugate pair of almost null structures gives rise to two holomorphic foliations of the complexified manifold, and the intersection of the leaves of these foliations are the complexification of totally null and geodetic leaves of a real foliation of the original real manifold. These analytical issues for even-dimensional Lorentzian manifolds are discussed in \cite{Trautman2002}. 

An almost null structure of constant real index on $(\mcM, \bm{g})$ is equivalent to the reduction of the structure group of the frame bundle to a real Lie group, ${}^\mcR P$ say, of $\SO(k,\ell)$. Its complexification can be viewed as the intersection of the complex parabolic subgroups ${}^\mcN P$ and ${}^{\overline{\mcN}} P$, preserving the almost null structures $\mcN$ and $\overline{\mcN}$ respectively. A description of such real Lie groups can be found at the infinitesimal level in \cite{Kopczy'nski1992}. In this context, the classification of the curvature tensors should be carried out in terms of an ${}^\mcR P$-invariant decomposition of the irreducible $\SO(k,\ell)$-modules $\mcC$ and $\mcA$. We can however bypass these representation-theoretic arguments by noting that the filtration
\eqref{eq-null_structure_real} and its complex conjugate induce two filtrations $\{{}^\mcN \mcC^i \}$ and $\{{}^{\overline{\mcN}} \mcC^i\}$ on $\C \otimes \mcC$, preserved by ${}^\mcN P$ and ${}^{\overline{\mcN}} P$ respectively. Then, for each $i$, we can consider the real span of the intersection ${}^\mcN \mcC^i_p \cap {}^{\overline{\mcN}} \mcC^i_p$ at every point $p$. This gives rise to an ${}^\mcR P$-invariant subbundle of $\mcC$, which we may reasonably describe as a real analogue of the complex subbundles ${}^\mcN \mcC^i$ and ${}^{\overline{\mcN}} \mcC^i$, i.e. it defines algebraic classes of the Weyl tensor with respect to both $\mcN$ and $\overline{\mcN}$. However, depending on the real index of $\mcN$, the fibers of ${}^\mcN \mcC^i$ and ${}^{\overline{\mcN}} \mcC^i$ will intersect trivially for some values of $i$, which precludes the existence of certain algebraic classes of Weyl tensors with respect to both $\mcN$ and $\overline{\mcN}$. The same argument applies regarding the Cotton-York tensor.

With these considerations in mind, we \emph{might} be able to apply Theorems \ref{thm-GS} and \ref{thm-GS_odd} in this real setting. Since the reality conditions on the Weyl tensor clearly violate the genericity assumption, we need to go back to the proofs of these theorems, and check whether these new assumptions undermine them. To facilitate the analysis, these reality conditions are described explicitly in the remark below.

\begin{rem}\label{rem-real_cond}
One may choose a (local) complexified frame $\{ \bm{\xi}_\mu, \tilde{\bm{\xi}}_{\tilde{\mu}}, \epsilon \bm{\xi}_0 \}$ adapted to the almost null structure $\mcN$, i.e. such that $\left\{ \bm{\xi}_\mu : \mu=1, \ldots, m \right\}$ span $\mcN$, as already described in section \ref{sec-null_basis}. This null frame will now be subject to reality conditions depending on the real index $r$ of $\mcN$ and metric signature $(k, \ell)$. For specificity, assume $k > \ell$. In $2m$ dimensions, the spanning vector fields of $\mcN$ can be chosen in such way that the complex conjugation acts as
\begin{align*}
\bar{} : ( \bm{\xi}_1 , \ldots , \bm{\xi}_r , \bm{\xi}_{r+1} , \ldots , \bm{\xi}_s , \bm{\xi}_{s+1} , \ldots , \bm{\xi}_m ) & \mapsto ( \bm{\xi}_1 , \ldots , \bm{\xi}_r , \tilde{\bm{\xi}}_{\widetilde{r+1}} , \ldots , \tilde{\bm{\xi}}_{\tilde{s}} , - \tilde{\bm{\xi}}_{\widetilde{s+1}} , \ldots , - \tilde{\bm{\xi}}_{\tilde{m}} ) \, ,
\end{align*}
where $s = \tfrac{k+r}{2}$ (this must be an integer by Lemma \ref{lem-real_index}). Here, the vector fields $\left\{ \bm{\xi}_\mu : \mu=1, \ldots, r \right\}$ span the real part of $\mcN \cap \overline{\mcN}$.

Similarly, in $2m+1$ dimensions, when $r-\ell$ is odd, we have
\begin{align*}
\bar{} : ( \bm{\xi}_1 , \ldots , \bm{\xi}_r , \bm{\xi}_{r+1} , \ldots , \bm{\xi}_s , \bm{\xi}_{s+1} , \ldots , \bm{\xi}_m , \bm{\xi}_0 ) & \mapsto ( \bm{\xi}_1 , \ldots , \bm{\xi}_r , \tilde{\bm{\xi}}_{\widetilde{r+1}} , \ldots , \tilde{\bm{\xi}}_{\tilde{s}} , - \tilde{\bm{\xi}}_{\widetilde{s+1}} , \ldots , - \tilde{\bm{\xi}}_{\tilde{m}} , - \bm{\xi}_0  ) \, ,
\end{align*}
where $s = \tfrac{k+r}{2}$. When $r-\ell$ is even, we have
\begin{align*}
\bar{} : ( \bm{\xi}_1 , \ldots , \bm{\xi}_r , \bm{\xi}_{r+1} , \ldots , \bm{\xi}_s , \bm{\xi}_{s+1} , \ldots , \bm{\xi}_m , \bm{\xi}_0 ) & \mapsto ( \bm{\xi}_1 , \ldots , \bm{\xi}_r , \tilde{\bm{\xi}}_{\widetilde{r+1}} , \ldots , \tilde{\bm{\xi}}_{\tilde{s}} , - \tilde{\bm{\xi}}_{\widetilde{s+1}} , \ldots , - \tilde{\bm{\xi}}_{\tilde{m}} , \bm{\xi}_0  ) \, , 
\end{align*}
where $s = \tfrac{k+r-1}{2}$.

In all three cases, the remaining vector fields $\left\{ \tilde{\bm{\xi}}_{\tilde{\mu}} : \mu=1, \ldots, r \right\}$ of the frame must evidently be real since the metric is real.

As in section \ref{sec-null_basis}, we can also consider the set $\mcB_S$ of all canonical almost null structures on some open set, each of which inherits the real index of the `primary' almost null structure $\mcN$. It is then more appropriate to take the quotient of  $\mcB_S$ by the equivalence relation
\begin{align*}
 \mcN_M & \sim \mcN_N & \Leftrightarrow & & \overline{\mcN_M} & = \mcN_N \, ,
\end{align*}
where $\mcN_M, \mcN_N \in \mcB_S$. This quotient will be denoted $\mcB_S / \sim$.
\end{rem}

We are now in the position of extending Theorem \ref{thm-GS_multi} to the real smooth category. Indeed, the theorem provides the right setting for the case at hand since if the Weyl tensor is degenerate with respect to an almost null structure $\mcN$, so must it be with respect to its complex conjugate $\overline{\mcN}$. Also, the complex conjugate pair $\mcN$ and $\overline{\mcN}$ will just be two of the canonical almost null structures given in section \ref{sec-null_basis}. In fact, we can consider any number of complex conjugate pairs of these, or in the above notation, any number of almost null structures in $\mcB_S / \sim$. Further, the reality conditions described in Remark \ref{rem-real_cond} on the non-vanishing components of the Weyl tensor is of no serious consequence on the genericity assumption. Finally, while Theorem \ref{thm-GS_multi} is formulated in the holomorphic category, real analyticity need not be imposed if one now regards the components of the connection and curvature, in the proofs of Theorems \ref{thm-GS}, \ref{thm-GS_odd} and \ref{thm-GS_multi}, as being complex-valued smooth functions on an open set. From these considerations, we can conclude

\begin{thm}\label{thm-real_GS_0}
 Let $(\mcM, \bm{g})$ be a $(2m+\epsilon)$-dimensional pseudo-Riemannian smooth manifold of arbitrary signature, where $\epsilon \in \{ 0,1 \}$ and $2m + \epsilon \geq 5$. Let $\mcN$ be an almost null structure on $\mcM$ of any real index allowable by Lemma \ref{lem-real_index}. Let $\mcB_S/\sim$ be the set of all canonical almost null structures on (some open subset of) $\mcM$ modulo complex conjugation as defined in Remark \ref{rem-real_cond}. Let $\mcB \subset \mcB_S / \sim$. Suppose that the Weyl and Cotton-York tensors (locally) satisfy
\begin{align*}
 \bm{C} (\bm{X} , \bm{Y} , \bm{Z} , \cdot ) & = 0 \, , & \bm{A} (\bm{Z} , \bm{X} , \bm{Y} ) & = 0 \, ,
\end{align*}
respectively, for all $\bm{X}, \bm{Y} \in \Gamma ( \mcN_{M}^\perp )$, and $\bm{Z} \in \Gamma ( \mcN_{M} )$, for all $\mcN_{M} \in \mcB$.
Assume further that the Weyl tensor is otherwise generic. Then the almost null structures in $\mcB$ are (locally) integrable.
\end{thm}

\begin{rem}
Incidentally, from its signature-independent formulation, this theorem may, in some instances, be regarded as a criterion as to whether a pseudo-Riemannian manifold of a given signature can be Wick-transformed to a different signature. Indeed, the Kerr-NUT-(A)dS metric, which has been presented in Euclidean, Lorentzian, and split signatures \cites{Chen2006,Chen2008}, is known \cite{Mason2010} to satisfy the algebraic degeneracy of Theorem \ref{thm-real_GS_0} where $\mcB = \mcB_S / \sim$.
\end{rem}

The full extensions of Theorems \ref{thm-GS} and \ref{thm-GS_odd} to the real category deserve separate treatments specific to each real index, and as they stand, the proofs must be adapted to the underlying real structure. Nonetheless, it must be emphasised that this is no tragedy. In fact, the arguments are greatly simplified by the fact that more components of the Weyl tensor vanish.

In the next three sections, we comment briefly on the Euclidean, split signature and Lorentzian cases.

\subsubsection{Hermitian structures}
We now assume that $\bm{g}$ is positive definite, i.e. $k = 2m+\epsilon$, $\ell = 0$. Then, the real index of an almost null structure must have constant real index $r=0$, and
the complexified tangent bundle splits according to the direct sum
\begin{align}\label{eq-Hermitian_grading}
 \C \otimes \Tgt \mcM & = \mcN \oplus \overline{\mcN} \oplus \epsilon ( \mcN^\perp \cap \overline{\mcN}^\perp ) \, .
\end{align}
When $\epsilon = 0$, $\mcN$ defines a metric compatible almost complex structure, and when $\epsilon = 1$, a metric compatible almost CR structure. In both cases, $\mcN$ and $\overline{\mcN}$ define the distributions of $(0,1)$-vectors and $(1,0)$-vectors respectively, or in other words, the $+\ii$- and $-\ii$-eigensubbundles of an almost complex structure of the complexified tangent bundle $\C \otimes \Tgt \mcM$. For specificity, we assume $\epsilon = 0$, in which case the structure of the frame bundle reduces from $\SO(2m)$ to the unitary group $\UU(m)$. In this case, we remark that the question of the integrability of the almost Hermitian structure does not,  by the Newlander-Niremberg theorem, necessitate real analyticity. 

Now, the direct sum \eqref{eq-Hermitian_grading} yields decompositions of the bundles $\mcC$ and $\mcA$ into irreducible $\UU(m)$-modules, and we recover the classification given in references \cites{Falcitelli1994,Tricerri1981}. In the present context, it suffices to note that in the complexification $\gr_i ( {}^\mcN \mcC ) \cong \overline{\gr_{-i} ( {}^\mcN \mcC )} \cong \gr_{-i} ( {}^{\overline{\mcN}} \mcC )$ for $i=0,1,2$, and $\gr_i ( {}^\mcN \mcA ) \cong \overline{\gr_{-i} ( {}^\mcN \mcA )} \cong \gr_{-i} ( {}^{\overline{\mcN}} \mcA )$, for  $i=\tfrac{1}{2},\tfrac{3}{2}$. One can then write $\mcC = \mcC_0 \oplus \mcC_1 \oplus \mcC_2$ and $\mcA = \mcA_{\tfrac{1}{2}} \oplus \mcA_{\tfrac{3}{2}}$, where each of the (not necessarily irreducible) $\UU(m)$-modules $\mcC_i$ and $\mcA_i$ can be identified with the real span of $\gr_i ( {}^\mcN \mcC )  \oplus \gr_{-i} ( {}^{\mcN} \mcC )$ and $\gr_i ( {}^\mcN \mcA )  \oplus \gr_{-i} ( {}^{\mcN} \mcA )$ respectively. A similar analysis can be done when $\epsilon = 1$. Thus, Theorem \ref{thm-real_GS_0} covers all possible algebraic special classes of the Weyl tensor with respect to one or more canonical almost null structures.

\subsubsection{Real null structures on pseudo-Riemannian manifolds of split signature}
The other extreme is the case where $\bm{g}$ has signature $(m,m+\epsilon)$ and the almost null structure has real index $m$. In this case, we literally have a real version of the classification of the Weyl tensor and Cotton-York tensor, and Theorems \ref{thm-GS}, \ref{thm-GS_odd} and \ref{thm-GS_multi} all apply. When integrable, the almost null structure gives rise to a foliation of $\mcM$ by $m$-dimensional totally null and geodetic leaves.

\subsubsection{Robinson structures}\label{sec-Robinson_structure}
In Lorenzian signature, i.e. $k = 2m-1 +\epsilon$, $\ell = 1$, one needs to distinguish between the cases $\epsilon = 0$ and $\epsilon=1$. When $\epsilon=0$, the real index of an almost null structure $\mcN$ must have constant real index $r=1$, and $\mcN$ defines an \emph{almost Robinson structure $(\mcN , \mcK)$} on $(\mcM , \bm{g})$, where $\mcK$ is a real null line bundle whose complexification is the intersection of $\mcN$ and its complex conjugate $\overline{\mcN}$. In particular, we have a filtration of vector bundles
\begin{align}\label{eq-K_filtration}
 \mcK \subset \mcK^\perp \subset \Tgt \mcM \, ,
\end{align}
where $\mcK^\perp$ is the orthogonal complement of $\mcK$ with respect to the \emph{real} metric. When $\mcN$ is integrable, the integral curves of the generators of $\mcK$ are null geodesics. At every point $p$, the fiber of the \emph{screen space} $\mcK^\perp/\mcK$ is naturally equipped with a complex structure, which is preserved along the flow of $\mcK$. Further, the quotient manifold $\mcM/\mcK$ acquires the structure of a CR manifold.

In the odd-dimensional case ($\epsilon=1$), the real index of an almost structure can be either $0$ or $1$. In fact, it may not even be constant throughout the manifold: an example is afforded by the five-dimensional black ring \cite{Taghavi-Chabert2011}. Nonetheless, in a small enough open set, the real index will remain constant. In the case $r=0$, the almost null structure defines an almost CR structure. On the other hand, when $r=1$, if $\mcN$ and $\mcN^\perp$ are integrable, the real null line bundle $\mcK$ arising from the intersection of the null distributions generates a congruence of null geodesics. Each fiber of the screen space $\mcK^\perp/\mcK$ is endowed with a CR structure, and the quotient manifold $\mcM/\mcK$ acquires the structure of a CR manifold of codimension $2$. The remaining part of the discussion focuses on Robinson structures.

Geometrically, the existence of a prefered null direction is equivalent to a reduction of the structure group of the frame bundle to the group $\Sim(2m-2+\epsilon)$, which preserves the filtration \eqref{eq-K_filtration}, and which has Lie algebra $\simalg(2m-2+\epsilon):=(\R \oplus \so (2m-2+\epsilon) ) \oplus \R^{2m-2 + \epsilon}$, a parabolic Lie subalgebra of $\so(1,2m-1+\epsilon)$. In the language of relativity, the summands of $\simalg(2m-2+\epsilon)$ generate \emph{boosts}, \emph{screen space rotations}, and \emph{null rotations} respectively. Clearly, the center $\mathfrak{z}$ of $\simalg(2m-2+\epsilon)$ lies in its $\R$-summand, and the grading element $\mathbf{E} \in \mfz$ induces a $|1|$-grading on both $\so(1,2m-1+\epsilon)$ and its standard representation. Thus, the tangent bundle (locally) admits the grading $\Tgt \mcM = \mcK_1 \oplus \mcK_0 \oplus \mcK_{-1}$ where $\mcK_1 \cong \mcK$, $\mcK_0 \cong \mcK^\perp / \mcK$, and $\mcK_{-1} \cong \Tgt \mcM / \mcK^\perp$. In the relativity literature \cite{Coley2004}, sections of $\mcK_i$ are said to be of \emph{boost weight $i$}. The bundles $\mcC$ and $\mcA$ now admit $\Sim(2m-2+\epsilon)$-invariant filtrations of vector bundles
\begin{align} \label{eq-Sim_filtrations}
 {}^\mcK \mcC^2 \subset {}^\mcK \mcC^1 \subset {}^\mcK \mcC^0 \subset {}^\mcK \mcC^{-1} \subset {}^\mcK \mcC^{-2} = \mcC \, , & &
 {}^\mcK \mcA^2 \subset {}^\mcK \mcA^1 \subset {}^\mcK \mcA^0 \subset {}^\mcK \mcA^{-1} \subset {}^\mcK \mcA^{-2} = \mcA \, ,
\end{align}
respectively, and each of the quotient bundles ${}^\mcK \mcC^i / {}^\mcK \mcC^{i+1}$ and ${}^\mcK \mcA^i / {}^\mcK \mcA^{i+1}$ is a completely reducible $\so(2m-2+\epsilon)$-module.

The existence of an almost Robinson structure is equivalent to a reduction of the structure group to the Lie group ${}^{\mcK} P$ with Lie algebra $(\R \oplus \uu (m-1) ) \oplus \R^{2m-2 + \epsilon} \subset \simalg(2m-2+\epsilon)$. This reduction induces further splitting of each of the $\so(2m-2+\epsilon)$-irreducible components of the associated graded bundle of any $\Sim(2m-2+\epsilon)$-invariant filtrations. In particular, in even dimensions, the complexification of the screen space splits as a direct sum
\begin{align*}
\C \otimes \mcK^\perp / \mcK & = 
(\mcK^\perp / \mcK )^{1,0} 
 \oplus 
(\mcK^\perp / \mcK )^{0,1} \oplus \C \, ,
\end{align*}
where $(\mcK^\perp / \mcK )^{1,0}$ and $(\mcK^\perp / \mcK )^{0,1}$ denote the $+\ii$- and $-\ii$-eigenbundles of the center $\mfz (\uu (m-1))$ respectively, and one can identify the null distribution $\mcN$ with $(\C \otimes \mcK )\oplus ( \mcK^\perp / \mcK )^{0,1}$.

The complexification of each of the quotient bundles ${}^\mcK \mcC^i / {}^\mcK \mcC^{i+1}$ is now a completely reducible $\uu(m-1)$-module. The $\Sim(2m-2+\epsilon)$-invariant filtrations \eqref{eq-Sim_filtrations}  decompose further into ${}^{\mcK} P$-invariant subfiltrations, and it is these filtrations that are relevant in the Lorentzian version of the Goldberg-Sachs theorem. The precise details will be given in a future publication. At this stage, it suffices to say that the curvature conditions for the existence of an Robinson structure can be read off from Lemma \ref{lem-type_characterisation}, provided that appropriate reality conditions are imposed. It is straightforward to check that under these reality conditions, the complex subbundles ${}^{\mcN} \mcC^2$ and ${}^{\overline{\mcN}} \mcC^2$ of $\C \otimes \mcC$ intersect trivially. Similarly, in odd dimensions, the bundles ${}^{\mcN} \mcC^k$ and ${}^{\overline{\mcN}} \mcC^k$ intersect trivially when $k=3,4$. In particular, in the light of Remark \ref{rem-Petrov_types}, there are no higher-dimensional Lorentzian analogue of the Petrov type N condition in this context.

In four dimensions, since $\so(2) \cong \uu(1)$, it is clear that there is no further reduction of the structure group. Geometrically, it simply means that singling out a null direction is equivalent to singling out an almost Robinson structure. Further, the isomorphism $\so(1,3) \cong \slie(2,\C)$ tells us that the complex conjugate of a self-dual almost null structures is anti-self-dual. In particular, the self-dual part of the Weyl tensor and an anti-self-dual part of the Weyl tensor are now complex conjugate of one another. Thus, the algebraic degeneracy of the self-dual part of the Weyl tensor is always mirrored by that of the anti-self-dual part of the Weyl tensor via complex conjugation. As a result, each of the self-dual complex Petrov types has a (non-trivial) real Lorentzian counterpart.

Theorem \ref{thm-real_GS_0} already gives a (partial) Lorentzian version of the Goldberg-Sachs theorem. In fact, it tells us more. One may have \emph{multiple Robinson structures}, i.e. a Lorentzian analogue of multiple null structures, whereby at most two distinct null directions are distinguished, each having up to $2^{m-1}$ complex structures associated to its screen space. For other degeneracy classes, however, one must alter the proofs of Theorems \ref{thm-GS} and \ref{thm-GS_odd} in order to demonstrate the integrability of the almost Robinson structure.\footnote{For a five-dimensional Einstein Lorentzian manifold, computations have shown that if the Weyl tensor degenerates to a generic section of the real span of ${}^{\mcN} \mcC^k \cap {}^{\overline{\mcN}} \mcC^k$, for $k=1,2$, then $(\mcN, \mcK)$ is integrable. The latter case is presented in \cites{Godazgar2010,Taghavi-Chabert2011}, and corresponds to the Weyl tensor being determined solely by a spinor field of real index $1$.} The full analysis in higher dimensions will be presented elsewhere.

Finally, as in the holomorphic category, one may wish to find refinements of the Goldberg-Sachs theorem in terms of irreducible $\uu(m-1)$-modules of the bundle ${}^\mcK \mcC$. However, it is pointed out in section 3.4.2 of reference \cite{Taghavi-Chabert2011} that certain algebraic classes of the Weyl tensor do not necessarily lead to the integrability of the underlying almost Robinson structure, and these classes are in fact defined by irreducible $\uu(m-1)$-modules.\footnote{To see this, we note that in five dimensions, Lorentzian signature, each of the irreducible ${}^{\mcK} P$-modules of the graded vector bundle $\gr(\mcC)$ is either one-real-dimensional or one-complex-dimensional, and each can be identified with a real or complex independent component of the Weyl tensor in a (spinor) frame adapted to $(\mcN,\mcK)$. Thus, the vanishing of one such component -- modulo ${}^{\mcK} P$-gauge transformations -- is tantamount to the projection of the Weyl tensor to the corresponding irreducible ${}^{\mcK} P$-module being zero.} This seems to indicate that the ${}^\mcN P$-invariant filtration on the complexification $\C \otimes \mcC$ is more relevant than the ${}^{\mcK} P$-irreducible modules of the real bundle $\mcC$ in determining the integrability of $(\mcN, \mcK)$.

\section{Conclusion and outlook}
The main thesis of this paper was to deduce the integrability of a given holomorphic maximal totally null distribution on a complex Riemannian manifold with prescribed Weyl and Cotton-York tensors in arbitrary dimensions. This can be viewed as a generalisation of what is known as the complex Goldberg-Sachs theorem. For this purpose, we introduced a higher-dimensional generalisation of the complex Petrov classification of the Weyl tensor. We also gave an extension of these results to the case of multiple almost null structures, which were then applied to the category of real smooth pseudo-Riemannian manifolds.

The discussion used a minimal amount of theory. But at its core was the parabolic Lie algebra stabilising the almost null structure.
There is another parabolic Lie algebra in the story, and it is related to conformal geometry. Indeed it was pointed out in section \ref{sec-conformal} that the Goldberg-Sachs theorem in any dimensions is a conformally invariant statement. This strongly suggests that the present results should be cast in the elegant language of parabolic geometry: the conditions on the Weyl tensor and Cotton-York tensor with respect to an almost null structure on a base conformal manifold can be re-expressed as conditions on the curvature of the projective pure spinor bundle fibered over it, and a section thereof. The content of the Goldberg-Sachs theorem is that certain curvature prescriptions imply a differential condition on this section, which can ultimately be translated by a foliation by maximal totally null leaves on the base manifold.

\appendix
\section{The Bianchi identity} \label{sec-Bianchi_identity}
In this appendix, we give the Bianchi identity in component form in the null basis $\{ \bm{\xi}_\mu , \tilde{\bm{\xi}}_{\tilde{\mu}} , \epsilon \bm{\xi}_\mu \}$ introduced in section \ref{sec-null_basis} with the following conventions:
\begin{itemize}
 \item the directional derivatives with respect to this frame are denoted
\begin{align*}
 \partial_\mu f & := \bm{\xi}_\mu f \, , & \partial_{\tilde{\mu}} f & := \tilde{\bm{\xi}}_{\tilde{\mu}} f \, , &  \partial_0 f & := \bm{\xi}_0 f \, , 
\end{align*}
for all $\mu$, and any for any holomorphic function $f$;
 \item the index notation follows the convention of \cite{Penrose1984}. In particular, the Einstein summation convention is used throughout, and square brackets around a set of indices denotes skew symmetrisation;
 \item only the Bianchi equations in odd dimensions are given -- the even-dimensional case can be recovered by ignoring any term containing an index $0$.
\end{itemize}

\begin{multline} \label{eq-Bianchi5}
\partial \ind{_{\lb{\mu}}} C \ind{_{\nu \rb{\rho} \kappa \lambda}}  = 
- 2 \Gamma \ind{_{\lb{\mu} \nu} ^\sigma} C \ind{_{\rb{\rho} \sigma \kappa \lambda}}
- 2 \Gamma \ind{_{\lb{\mu} \nu} ^{\tilde{\sigma}}} C \ind{_{\rb{\rho} \tilde{\sigma} \kappa \lambda}}
- 2 \Gamma \ind{_{\lb{\mu} \nu} ^0} C \ind{_{\rb{\rho} 0 \kappa \lambda}} \\
- 2 \Gamma \ind{_{\lb{\mu}| \lb{\kappa}} ^\sigma} C \ind{_{\rb{\lambda}|\sigma| \nu \rb{\rho}}} 
- 2 \Gamma \ind{_{\lb{\mu}| \lb{\kappa}} ^{\tilde{\sigma}}} C \ind{_{\rb{\lambda}|\tilde{\sigma}| \nu \rb{\rho}}}
- 2 \Gamma \ind{_{\lb{\mu}| \lb{\kappa}} ^0} C \ind{_{\rb{\lambda}|0| \nu \rb{\rho}}} \, ,
\end{multline}
\begin{multline}  \label{eq-Bianchi4}
2 \partial \ind{_{\lb{\mu}}} C \ind{_{\rb{\nu} 0 \kappa \lambda}} + \partial \ind{_{0}} C \ind{_{\mu \nu \kappa \lambda}} =  
 - 2 \Gamma \ind{_{[\mu \nu]} ^\sigma} C \ind{_{0 \sigma \kappa \lambda}}
- 2 \Gamma \ind{_{[\mu \nu]} ^{\tilde{\sigma}}} C \ind{_{0 {\tilde{\sigma}} \kappa \lambda}}
- 2 \Gamma \ind{_{0 \lb{\mu}} ^\sigma} C \ind{_{\rb{\nu} \sigma \kappa \lambda}}
- 2 \Gamma \ind{_{0 \lb{\mu}} ^{\tilde{\sigma}}} C \ind{_{\rb{\nu} {\tilde{\sigma}} \kappa \lambda}}
- 2 \Gamma \ind{_{0 \lb{\mu}} ^0} C \ind{_{\rb{\nu} 0 \kappa \lambda}} \\
- 2 \Gamma \ind{_{\lb{\nu}| 0|} ^\sigma} C \ind{_{\rb{\mu} \sigma \kappa \lambda}}
- 2 \Gamma \ind{_{\lb{\nu}| 0|} ^{\tilde{\sigma}}} C \ind{_{\rb{\mu} {\tilde{\sigma}} \kappa \lambda}}
- 4 \Gamma \ind{_{\lb{\mu}| \lb{\kappa}} ^\sigma} C \ind{_{\rb{\lambda} | \sigma | \rb{\nu} 0}}
- 4 \Gamma \ind{_{\lb{\mu}| \lb{\kappa}} ^{\tilde{\sigma}}} C \ind{_{\rb{\lambda}|{\tilde{\sigma}}| \rb{\nu} 0}}
- 4 \Gamma \ind{_{\lb{\mu}| \lb{\kappa}} ^0} C \ind{_{\rb{\lambda}|0| \rb{\nu} 0}} \\
- 2 \Gamma \ind{_{0 \lb{\kappa}} ^\sigma} C \ind{_{\rb{\lambda} \sigma \mu \nu}}
- 2 \Gamma \ind{_{0 \lb{\kappa}} ^{\tilde{\sigma}}} C \ind{_{\rb{\lambda} {\tilde{\sigma}}  \mu \nu}} 
- 2 \Gamma \ind{_{0 \lb{\kappa}} ^0} C \ind{_{\rb{\lambda} 0  \mu \nu}} 
\, ,
\end{multline}
\begin{multline}  \label{eq-Bianchi3a}
2 \partial \ind{_{\lb{\mu}}} C \ind{_{\rb{\nu} \tilde{\rho} \kappa \lambda}} + \partial \ind{_{\tilde{\rho}}} C \ind{_{\mu \nu \kappa \lambda}} = - 2 g \ind{_{\tilde{\rho} \lb{\kappa}}} A \ind{_{\rb{\lambda} \mu \nu}} - 2 \Gamma \ind{_{[\mu \nu]} ^\sigma} C \ind{_{\tilde{\rho} \sigma \kappa \lambda}} - 2 \Gamma \ind{_{[\mu \nu]} ^{\tilde{\sigma}}} C \ind{_{\tilde{\rho} {\tilde{\sigma}} \kappa \lambda}} - 2 \Gamma \ind{_{[\mu \nu]} ^0} C \ind{_{\tilde{\rho} 0 \kappa \lambda}} \\
- 2 \Gamma \ind{_{\tilde{\rho} \lb{\mu}} ^\sigma} C \ind{_{\rb{\nu} \sigma \kappa \lambda}} - 2 \Gamma \ind{_{\tilde{\rho} \lb{\mu}} ^{\tilde{\sigma}}} C \ind{_{\rb{\nu} {\tilde{\sigma}} \kappa \lambda}} - 2 \Gamma \ind{_{\tilde{\rho} \lb{\mu}} ^0} C \ind{_{\rb{\nu} 0 \kappa \lambda}} - 2 \Gamma \ind{_{\lb{\nu}| \tilde{\rho}|} ^\sigma} C \ind{_{\rb{\mu} \sigma \kappa \lambda}} - 2 \Gamma \ind{_{\lb{\nu}| \tilde{\rho}|} ^{\tilde{\sigma}}} C \ind{_{\rb{\mu} {\tilde{\sigma}} \kappa \lambda}} - 2 \Gamma \ind{_{\lb{\nu}| \tilde{\rho}|} ^0} C \ind{_{\rb{\mu} 0 \kappa \lambda}} \\ - 4 \Gamma \ind{_{\lb{\mu}| \lb{\kappa}} ^\sigma} C \ind{_{\rb{\lambda}|\sigma| \rb{\nu} \tilde{\rho}}} - 4 \Gamma \ind{_{\lb{\mu}| \lb{\kappa}} ^{\tilde{\sigma}}} C \ind{_{\rb{\lambda}|{\tilde{\sigma}}| \rb{\nu} \tilde{\rho}}} - 4 \Gamma \ind{_{\lb{\mu}| \lb{\kappa}} ^0} C \ind{_{\rb{\lambda}| 0 | \rb{\nu} \tilde{\rho}}} - 2 \Gamma \ind{_{\tilde{\rho} \lb{\kappa}} ^\sigma} C \ind{_{\rb{\lambda} \sigma \mu \nu}} - 2 \Gamma \ind{_{\tilde{\rho} \lb{\kappa}} ^{\tilde{\sigma}}} C \ind{_{\rb{\lambda} {\tilde{\sigma}} \mu \nu}} - 2 \Gamma \ind{_{\tilde{\rho} \lb{\kappa}} ^0} C \ind{_{\rb{\lambda}  0 \mu \nu}} \, ,
\end{multline}
\begin{multline} \label{eq-Bianchi3b}
2 \partial \ind{_{\lb{\mu}}} C \ind{_{\rb{\nu} 0 \kappa 0}} + \partial \ind{_{0}} C \ind{_{\mu \nu \kappa 0}} = A \ind{_{\kappa \mu \nu}} - 2 \Gamma \ind{_{[\mu \nu]} ^\sigma} C \ind{_{0 \sigma \kappa 0}} - 2 \Gamma \ind{_{[\mu \nu]} ^{\tilde{\sigma}}} C \ind{_{0 {\tilde{\sigma}} \kappa 0}} \\
- 2 \Gamma \ind{_{0 \lb{\mu}} ^\sigma} C \ind{_{\rb{\nu} \sigma \kappa 0}} - 2 \Gamma \ind{_{0 \lb{\mu}} ^{\tilde{\sigma}}} C \ind{_{\rb{\nu} {\tilde{\sigma}} \kappa 0}} - 2 \Gamma \ind{_{0 \lb{\mu}} ^0} C \ind{_{\rb{\nu} 0 \kappa 0}} - 2 \Gamma \ind{_{\lb{\nu}| 0|} ^\sigma} C \ind{_{\rb{\mu} \sigma \kappa 0}} - 2 \Gamma \ind{_{\lb{\nu}| 0|} ^{\tilde{\sigma}}} C \ind{_{\rb{\mu} {\tilde{\sigma}} \kappa 0}}  \\ - 2 \Gamma \ind{_{\lb{\mu}| \kappa} ^\sigma} C \ind{_{0 \sigma| \rb{\nu} 0}}  + 2 \Gamma \ind{_{\lb{\mu}| 0} ^\sigma} C \ind{_{\kappa \sigma| \rb{\nu} 0}} - 2 \Gamma \ind{_{\lb{\mu}| \kappa} ^{\tilde{\sigma}}} C \ind{_{0 {\tilde{\sigma}}| \rb{\nu} 0}} + 2 \Gamma \ind{_{\lb{\mu}| 0} ^{\tilde{\sigma}}} C \ind{_{\kappa {\tilde{\sigma}}| \rb{\nu} 0}}  \\ - \Gamma \ind{_{0 \kappa} ^\sigma} C \ind{_{0 \sigma \mu \nu}} + \Gamma \ind{_{0 0} ^\sigma} C \ind{_{\kappa \sigma \mu \nu}} -  \Gamma \ind{_{0 \kappa} ^{\tilde{\sigma}}} C \ind{_{0 {\tilde{\sigma}} \mu \nu}}  +  \Gamma \ind{_{0 0} ^{\tilde{\sigma}}} C \ind{_{\kappa {\tilde{\sigma}} \mu \nu}}  \, ,
\end{multline}
\begin{multline} \label{eq-Bianchi2a}
2 \partial \ind{_{\lb{\mu}}} C \ind{_{\rb{\nu} \tilde{\rho} \kappa 0}} + \partial \ind{_{\tilde{\rho}}} C \ind{_{\mu \nu \kappa 0}} = - g \ind{_{\tilde{\rho} \kappa}} A \ind{_{0 \mu \nu}} - 2 \Gamma \ind{_{[\mu \nu]} ^\sigma} C \ind{_{\tilde{\rho} \sigma \kappa 0}} - 2 \Gamma \ind{_{[\mu \nu]} ^{\tilde{\sigma}}} C \ind{_{\tilde{\rho} {\tilde{\sigma}} \kappa 0}} - 2 \Gamma \ind{_{[\mu \nu]} ^0} C \ind{_{\tilde{\rho} 0 \kappa 0}} \\
- 2 \Gamma \ind{_{\tilde{\rho} \lb{\mu}} ^\sigma} C \ind{_{\rb{\nu} \sigma \kappa 0}} - 2 \Gamma \ind{_{\tilde{\rho} \lb{\mu}} ^{\tilde{\sigma}}} C \ind{_{\rb{\nu} {\tilde{\sigma}} \kappa 0}} - 2 \Gamma \ind{_{\tilde{\rho} \lb{\mu}} ^0} C \ind{_{\rb{\nu} 0 \kappa 0}} - 2 \Gamma \ind{_{\lb{\nu}| \tilde{\rho}|} ^\sigma} C \ind{_{\rb{\mu} \sigma \kappa 0}} - 2 \Gamma \ind{_{\lb{\nu}| \tilde{\rho}|} ^{\tilde{\sigma}}} C \ind{_{\rb{\mu} {\tilde{\sigma}} \kappa 0}} - 2 \Gamma \ind{_{\lb{\nu}| \tilde{\rho}|} ^0} C \ind{_{\rb{\mu} 0 \kappa 0}} \\ - 2 \Gamma \ind{_{\lb{\mu}| \kappa} ^\sigma} C \ind{_{0 \sigma | \rb{\nu} \tilde{\rho}}} + 2 \Gamma \ind{_{\lb{\mu}| 0} ^\sigma} C \ind{_{\kappa \sigma | \rb{\nu} \tilde{\rho}}} - 2 \Gamma \ind{_{\lb{\mu}| \kappa} ^{\tilde{\sigma}}} C \ind{_{0  {\tilde{\sigma}} | \rb{\nu} \tilde{\rho}}} + 2 \Gamma \ind{_{\lb{\mu}| 0} ^{\tilde{\sigma}}} C \ind{_{\kappa {\tilde{\sigma}} | \rb{\nu} \tilde{\rho}}} \\
- \Gamma \ind{_{\tilde{\rho} \kappa} ^\sigma} C \ind{_{0 \sigma \mu \nu}} + \Gamma \ind{_{\tilde{\rho} 0} ^\sigma} C \ind{_{\kappa \sigma \mu \nu}} - \Gamma \ind{_{\tilde{\rho} \kappa} ^{\tilde{\sigma}}} C \ind{_{0 {\tilde{\sigma}}  \mu \nu}} + \Gamma \ind{_{\tilde{\rho} 0} ^{\tilde{\sigma}}} C \ind{_{\kappa {\tilde{\sigma}}  \mu \nu}} \, ,
\end{multline}
\begin{multline} \label{eq-Bianchi2b}
2 \partial \ind{_{\lb{\mu}}} C \ind{_{\rb{\nu} 0 \kappa \tilde{\lambda}}} + \partial \ind{_{0}} C \ind{_{\mu \nu \kappa \tilde{\lambda}}} = - 2 g \ind{_{\lb{\mu} | \tilde{\lambda}}} A \ind{_{\kappa | \rb{\nu} 0}} - 2 \Gamma \ind{_{[\mu \nu]} ^\sigma} C \ind{_{0 \sigma \kappa \tilde{\lambda}}} - 2 \Gamma \ind{_{[\mu \nu]} ^{\tilde{\sigma}}} C \ind{_{0 {\tilde{\sigma}} \kappa \tilde{\lambda}}} \\
- 2 \Gamma \ind{_{0 \lb{\mu}} ^\sigma} C \ind{_{\rb{\nu} \sigma \kappa \tilde{\lambda}}} - 2 \Gamma \ind{_{0 \lb{\mu}} ^{\tilde{\sigma}}} C \ind{_{\rb{\nu} {\tilde{\sigma}} \kappa \tilde{\lambda}}} - 2 \Gamma \ind{_{0 \lb{\mu}} ^0} C \ind{_{\rb{\nu} 0 \kappa \tilde{\lambda}}} - 2 \Gamma \ind{_{\lb{\nu}| 0|} ^\sigma} C \ind{_{\rb{\mu} \sigma \kappa \tilde{\lambda}}} - 2 \Gamma \ind{_{\lb{\nu}| 0|} ^{\tilde{\sigma}}} C \ind{_{\rb{\mu} {\tilde{\sigma}} \kappa \tilde{\lambda}}} \\
- 2 \Gamma \ind{_{\lb{\mu}| \kappa} ^\sigma} C \ind{_{\tilde{\lambda} \sigma| \rb{\nu} 0}} + 2 \Gamma \ind{_{\lb{\mu}| \tilde{\lambda}} ^\sigma} C \ind{_{\kappa \sigma| \rb{\nu} 0}} - 2 \Gamma \ind{_{\lb{\mu} | \kappa} ^{\tilde{\sigma}}} C \ind{_{\tilde{\lambda}{\tilde{\sigma}}| \rb{\nu} 0}} + 2 \Gamma \ind{_{\lb{\mu} | \tilde{\lambda}} ^{\tilde{\sigma}}} C \ind{_{\kappa {\tilde{\sigma}}| \rb{\nu} 0}} - 2 \Gamma \ind{_{\lb{\mu}| \kappa} ^0} C \ind{_{\tilde{\lambda} 0 | \rb{\nu} 0}} + 2 \Gamma \ind{_{\lb{\mu}| \tilde{\lambda}} ^0} C \ind{_{\kappa 0 | \rb{\nu} 0}} \\ - \Gamma \ind{_{0 \kappa} ^\sigma} C \ind{_{\tilde{\lambda} \sigma \mu \nu}} + \Gamma \ind{_{0 \tilde{\lambda}} ^\sigma} C \ind{_{\kappa \sigma \mu \nu}} - \Gamma \ind{_{0 \kappa} ^{\tilde{\sigma}}} C \ind{_{\tilde{\lambda} {\tilde{\sigma}} \mu \nu}} + \Gamma \ind{_{0 \tilde{\lambda}} ^{\tilde{\sigma}}} C \ind{_{\kappa {\tilde{\sigma}} \mu \nu}} - \Gamma \ind{_{0 \kappa} ^0} C \ind{_{\tilde{\lambda}  0 \mu \nu}} + \Gamma \ind{_{0 \tilde{\lambda}} ^0} C \ind{_{\kappa 0 \mu \nu}} \, ,
\end{multline}
\begin{multline} \label{eq-Bianchi1a}
2 \partial \ind{_{\lb{\mu}}} C \ind{_{\rb{\nu} \tilde{\rho} \kappa \tilde{\lambda}}} + \partial \ind{_{\tilde{\rho}}} C \ind{_{\mu \nu \kappa \tilde{\lambda}}} =  2 g \ind{_{\lb{\mu} | \tilde{\lambda}}} A \ind{_{\kappa | \rb{\nu} \tilde{\rho}}} - g \ind{_{\tilde{\rho} \kappa}} A \ind{_{\tilde{\lambda} \mu \nu}} \\
- 2 \Gamma \ind{_{[\mu \nu]} ^\sigma} C \ind{_{\tilde{\rho} \sigma \kappa \tilde{\lambda}}}
- 2 \Gamma \ind{_{[\mu \nu]} ^{\tilde{\sigma}}} C \ind{_{\tilde{\rho} {\tilde{\sigma}} \kappa \tilde{\lambda}}}
- 2 \Gamma \ind{_{[\mu \nu]} ^0} C \ind{_{\tilde{\rho} 0 \kappa \tilde{\lambda}}} 
- 2 \Gamma \ind{_{\tilde{\rho} \lb{\mu}} ^\sigma} C \ind{_{\rb{\nu} \sigma \kappa \tilde{\lambda}}}
- 2 \Gamma \ind{_{\tilde{\rho} \lb{\mu}} ^{\tilde{\sigma}}} C \ind{_{\rb{\nu} {\tilde{\sigma}} \kappa \tilde{\lambda}}} 
- 2 \Gamma \ind{_{\tilde{\rho} \lb{\mu}} ^0} C \ind{_{\rb{\nu} 0 \kappa \tilde{\lambda}}} \\
- 2 \Gamma \ind{_{\lb{\nu}| \tilde{\rho}|} ^\sigma} C \ind{_{\rb{\mu} \sigma \kappa \tilde{\lambda}}}
- 2 \Gamma \ind{_{\lb{\nu}| \tilde{\rho}|} ^{\tilde{\sigma}}} C \ind{_{\rb{\mu} {\tilde{\sigma}} \kappa \tilde{\lambda}}} 
- 2 \Gamma \ind{_{\lb{\nu}| \tilde{\rho}|} ^0} C \ind{_{\rb{\mu} 0 \kappa \tilde{\lambda}}}
- 2 \Gamma \ind{_{\lb{\mu}| \kappa} ^\sigma} C \ind{_{\tilde{\lambda} |\sigma| \rb{\nu} \tilde{\rho}}}
- 2 \Gamma \ind{_{\lb{\mu}| \kappa} ^{\tilde{\sigma}}} C \ind{_{\tilde{\lambda}|{\tilde{\sigma}}| \rb{\nu} \tilde{\rho}}}
- 2 \Gamma \ind{_{\lb{\mu}| \kappa} ^0} C \ind{_{\tilde{\lambda}|0| \rb{\nu} \tilde{\rho}}} \\
+ 2 \Gamma \ind{_{\lb{\mu}| \tilde{\lambda}} ^\sigma} C \ind{_{\kappa |\sigma| \rb{\nu} \tilde{\rho}}}
+ 2 \Gamma \ind{_{\lb{\mu}| \tilde{\lambda}} ^{\tilde{\sigma}}} C \ind{_{\kappa|{\tilde{\sigma}}| \rb{\nu} \tilde{\rho}}}
+ 2 \Gamma \ind{_{\lb{\mu}| \tilde{\lambda}} ^0} C \ind{_{\kappa |0| \rb{\nu} \tilde{\rho}}}
- \Gamma \ind{_{\tilde{\rho} \kappa} ^\sigma} C \ind{_{\tilde{\lambda} \sigma \mu \nu}}
- \Gamma \ind{_{\tilde{\rho} \kappa} ^{\tilde{\sigma}}} C \ind{_{\tilde{\lambda} {\tilde{\sigma}} \mu \nu}} 
- \Gamma \ind{_{\tilde{\rho} \kappa} ^0} C \ind{_{\tilde{\lambda} 0 \mu \nu}} \\
+ \Gamma \ind{_{\tilde{\rho} \tilde{\lambda}} ^\sigma} C \ind{_{\kappa \sigma \mu \nu}}
+ \Gamma \ind{_{\tilde{\rho} \tilde{\lambda}} ^{\tilde{\sigma}}} C \ind{_{\kappa {\tilde{\sigma}} \mu \nu}}
+ \Gamma \ind{_{\tilde{\rho} \tilde{\lambda}} ^0} C \ind{_{\kappa 0 \mu \nu}} \, ,
\end{multline}
\begin{multline} \label{eq-Bianchi1b}
\partial \ind{_{\mu}} C \ind{_{\tilde{\nu} 0 \kappa 0}} + \partial \ind{_{\tilde{\nu}}} C \ind{_{0 \mu \kappa 0}} + \partial \ind{_{0}} C \ind{_{\mu \tilde{\nu} \kappa 0}} =  - g \ind{_{\tilde{\nu} \kappa}} A \ind{_{0 0 \mu}} + A \ind{_{\kappa \mu \tilde{\nu}}} \\
- \Gamma \ind{_{\mu \tilde{\nu}} ^\sigma} C \ind{_{0 \sigma \kappa 0}}
- \Gamma \ind{_{\mu \tilde{\nu}} ^{\tilde{\sigma}}} C \ind{_{0 {\tilde{\sigma}} \kappa 0}} 
+ \Gamma \ind{_{\tilde{\nu} \mu} ^\sigma} C \ind{_{0 \sigma \kappa 0}}
+ \Gamma \ind{_{\tilde{\nu} \mu} ^{\tilde{\sigma}}} C \ind{_{0 {\tilde{\sigma}} \kappa 0}} 
- \Gamma \ind{_{\tilde{\nu} 0} ^\sigma} C \ind{_{\mu \sigma \kappa 0}}
- \Gamma \ind{_{\tilde{\nu} 0} ^{\tilde{\sigma}}} C \ind{_{\mu {\tilde{\sigma}} \kappa 0}} \\
+ \Gamma \ind{_{0 \tilde{\nu}} ^\sigma} C \ind{_{\mu \sigma \kappa 0}}
+ \Gamma \ind{_{0 \tilde{\nu}} ^{\tilde{\sigma}}} C \ind{_{\mu {\tilde{\sigma}} \kappa 0}} 
+ \Gamma \ind{_{0 \tilde{\nu}} ^0} C \ind{_{\mu 0 \kappa 0}} 
- \Gamma \ind{_{0 \mu} ^\sigma} C \ind{_{\tilde{\nu} \sigma \kappa 0}}
- \Gamma \ind{_{0 \mu} ^{\tilde{\sigma}}} C \ind{_{\tilde{\nu} {\tilde{\sigma}} \kappa 0}} 
- \Gamma \ind{_{0 \mu} ^0} C \ind{_{\tilde{\nu} 0 \kappa 0}} \\
+ \Gamma \ind{_{\mu 0} ^\sigma} C \ind{_{\tilde{\nu} \sigma \kappa 0}}
+ \Gamma \ind{_{\mu 0} ^{\tilde{\sigma}}} C \ind{_{\tilde{\nu} {\tilde{\sigma}} \kappa 0}} 
- \Gamma \ind{_{\mu \kappa} ^\sigma} C \ind{_{0 \sigma \tilde{\nu} 0}}
- \Gamma \ind{_{\mu \kappa} ^{\tilde{\sigma}}} C \ind{_{0 {\tilde{\sigma}} \tilde{\nu} 0}} 
+ \Gamma \ind{_{\mu 0} ^\sigma} C \ind{_{\kappa \sigma \tilde{\nu} 0}} 
+ \Gamma \ind{_{\mu 0} ^{\tilde{\sigma}}} C \ind{_{\kappa {\tilde{\sigma}} \tilde{\nu} 0}} \\
- \Gamma \ind{_{\tilde{\nu} \kappa} ^\sigma} C \ind{_{0 \sigma 0 \mu}}
- \Gamma \ind{_{\tilde{\nu} \kappa} ^{\tilde{\sigma}}} C \ind{_{0 {\tilde{\sigma}} 0 \mu}}
+ \Gamma \ind{_{\tilde{\nu} 0} ^\sigma} C \ind{_{\kappa \sigma 0 \mu}} 
+ \Gamma \ind{_{\tilde{\nu} 0} ^{\tilde{\sigma}}} C \ind{_{\kappa {\tilde{\sigma}} 0 \mu}}  
- \Gamma \ind{_{0 \kappa} ^\sigma} C \ind{_{0 \sigma \mu \tilde{\nu}}}
- \Gamma \ind{_{0 \kappa} ^{\tilde{\sigma}}} C \ind{_{0 {\tilde{\sigma}} \mu \tilde{\nu}}} \\
+ \Gamma \ind{_{0 0} ^\sigma} C \ind{_{\kappa \sigma \mu \tilde{\nu}}}
+ \Gamma \ind{_{0 0} ^{\tilde{\sigma}}} C \ind{_{\kappa {\tilde{\sigma}} \mu \tilde{\nu}}} 
   \, ,
\end{multline}
\begin{multline} \label{eq-Bianchi0a}
2 \partial \ind{_{\lb{\tilde{\mu}}}} C \ind{_{\rb{\tilde{\nu}} \rho \kappa 0}} + \partial \ind{_{\rho}} C \ind{_{\tilde{\mu} \tilde{\nu} \kappa 0}} =  - 2 g \ind{_{\lb{\tilde{\mu}} | \kappa}} A \ind{_{0 | \rb{\tilde{\nu}} \rho}}   
- 2 \Gamma \ind{_{[\tilde{\mu} \tilde{\nu}]} ^\sigma} C \ind{_{\rho \sigma \kappa 0}}
- 2 \Gamma \ind{_{[\tilde{\mu} \tilde{\nu}]} ^{\tilde{\sigma}}} C \ind{_{\rho {\tilde{\sigma}} \kappa 0}} 
- 2 \Gamma \ind{_{[\tilde{\mu} \tilde{\nu}]} ^0} C \ind{_{\rho 0 \kappa 0}} \\
- 2 \Gamma \ind{_{\rho \lb{\tilde{\mu}}} ^\sigma} C \ind{_{\rb{\tilde{\nu}} \sigma \kappa 0}}
- 2 \Gamma \ind{_{\rho \lb{\tilde{\mu}}} ^{\tilde{\sigma}}} C \ind{_{\rb{\tilde{\nu}} {\tilde{\sigma}} \kappa 0}} 
- 2 \Gamma \ind{_{\rho \lb{\tilde{\mu}}} ^0} C \ind{_{\rb{\tilde{\nu}} 0 \kappa 0}} 
- 2 \Gamma \ind{_{\lb{\tilde{\nu}}| \rho|} ^\sigma} C \ind{_{\rb{\tilde{\mu}} \sigma \kappa 0}}
- 2 \Gamma \ind{_{\lb{\tilde{\nu}}| \rho|} ^{\tilde{\sigma}}} C \ind{_{\rb{\tilde{\mu}} {\tilde{\sigma}} \kappa 0}} 
- 2 \Gamma \ind{_{\lb{\tilde{\nu}}| \rho|} ^0} C \ind{_{\rb{\tilde{\mu}} 0 \kappa 0}} \\
- 2 \Gamma \ind{_{\lb{\tilde{\mu}}| \kappa} ^\sigma} C \ind{_{0|\sigma| \rb{\tilde{\nu}} \rho}}
+ 2 \Gamma \ind{_{\lb{\tilde{\mu}}| 0} ^\sigma} C \ind{_{\kappa|\sigma| \rb{\tilde{\nu}} \rho}}
- 2 \Gamma \ind{_{\lb{\tilde{\mu}}| \kappa} ^{\tilde{\sigma}}} C \ind{_{0|{\tilde{\sigma}}| \rb{\tilde{\nu}} \rho}}
+ 2 \Gamma \ind{_{\lb{\tilde{\mu}}| 0} ^{\tilde{\sigma}}} C \ind{_{\kappa|{\tilde{\sigma}}| \rb{\tilde{\nu}} \rho}}  
\\
- \Gamma \ind{_{\rho \kappa} ^\sigma} C \ind{_{0 \sigma \tilde{\mu} \tilde{\nu}}}
+ \Gamma \ind{_{\rho 0} ^\sigma} C \ind{_{\kappa \sigma \tilde{\mu} \tilde{\nu}}}  
- \Gamma \ind{_{\rho \kappa} ^{\tilde{\sigma}}} C \ind{_{0 {\tilde{\sigma}} \tilde{\mu} \tilde{\nu}}}
+ \Gamma \ind{_{\rho 0} ^{\tilde{\sigma}}} C \ind{_{\kappa {\tilde{\sigma}} \tilde{\mu} \tilde{\nu}}}
\end{multline}
\begin{multline} \label{eq-Bianchi0b}
2 \partial \ind{_{\lb{\tilde{\mu}}}} C \ind{_{\rb{\tilde{\nu}} 0 \kappa \lambda}} + \partial \ind{_{0}} C \ind{_{\tilde{\mu} \tilde{\nu} \kappa \lambda}} =  - 4 g \ind{_{\lb{\tilde{\mu}} | \lb{\kappa}}} A \ind{_{\rb{\lambda} | \rb{\tilde{\nu}} 0}}  
- 2 \Gamma \ind{_{[\tilde{\mu} \tilde{\nu}]} ^\sigma} C \ind{_{0 \sigma \kappa \lambda}}
- 2 \Gamma \ind{_{[\tilde{\mu} \tilde{\nu}]} ^{\tilde{\sigma}}} C \ind{_{0 {\tilde{\sigma}} \kappa \lambda}} 
\\
- 2 \Gamma \ind{_{0 \lb{\tilde{\mu}}} ^\sigma} C \ind{_{\rb{\tilde{\nu}} \sigma \kappa \lambda}}
- 2 \Gamma \ind{_{0 \lb{\tilde{\mu}}} ^{\tilde{\sigma}}} C \ind{_{\rb{\tilde{\nu}} {\tilde{\sigma}} \kappa \lambda}} 
- 2 \Gamma \ind{_{0 \lb{\tilde{\mu}}} ^0} C \ind{_{\rb{\tilde{\nu}} 0 \kappa \lambda}} 
- 2 \Gamma \ind{_{\lb{\tilde{\nu}}| 0|} ^\sigma} C \ind{_{\rb{\tilde{\mu}} \sigma \kappa \lambda}}
- 2 \Gamma \ind{_{\lb{\tilde{\nu}}| 0|} ^{\tilde{\sigma}}} C \ind{_{\rb{\tilde{\mu}} {\tilde{\sigma}} \kappa \lambda}} 
\\
- 4 \Gamma \ind{_{\lb{\tilde{\mu}}| \lb{\kappa}} ^\sigma} C \ind{_{\rb{\lambda}|\sigma| \rb{\tilde{\nu}} 0}}
- 4 \Gamma \ind{_{\lb{\tilde{\mu}}| \lb{\kappa}} ^{\tilde{\sigma}}} C \ind{_{\rb{\lambda}|{\tilde{\sigma}}| \rb{\tilde{\nu}} 0}} 
- 4 \Gamma \ind{_{\lb{\tilde{\mu}}| \lb{\kappa}} ^0} C \ind{_{\rb{\lambda}|0| \rb{\tilde{\nu}} 0}} \\
- 2 \Gamma \ind{_{0 \lb{\kappa}} ^\sigma} C \ind{_{\rb{\lambda} \sigma \tilde{\mu} \tilde{\nu}}} 
- 2 \Gamma \ind{_{0 \lb{\kappa}} ^{\tilde{\sigma}}} C \ind{_{\rb{\lambda} {\tilde{\sigma}} \tilde{\mu} \tilde{\nu}}}
- 2 \Gamma \ind{_{0 \lb{\kappa}} ^0} C \ind{_{\rb{\lambda} 0 \tilde{\mu} \tilde{\nu}}} \, ,
\end{multline}
\begin{multline}  \label{eq-Bianchi0c}
2 \partial \ind{_{\lb{\mu}}} C \ind{_{\rb{\nu} 0 \tilde{\kappa} \tilde{\lambda}}} + \partial \ind{_{0}} C \ind{_{\mu \nu \tilde{\kappa} \tilde{\lambda}}} =  - 4 g \ind{_{\lb{\mu} | \lb{\tilde{\kappa}}}} A \ind{_{\rb{\tilde{\lambda}} | \rb{\nu} 0}} 
- 2 \Gamma \ind{_{[\mu \nu]} ^\sigma} C \ind{_{0 \sigma \tilde{\kappa} \tilde{\lambda}}}
- 2 \Gamma \ind{_{[\mu \nu]} ^{\tilde{\sigma}}} C \ind{_{0 {\tilde{\sigma}} \tilde{\kappa} \tilde{\lambda}}} 
\\
- 2 \Gamma \ind{_{0 \lb{\mu}} ^\sigma} C \ind{_{\rb{\nu} \sigma \tilde{\kappa} \tilde{\lambda}}}
- 2 \Gamma \ind{_{0 \lb{\mu}} ^{\tilde{\sigma}}} C \ind{_{\rb{\nu} {\tilde{\sigma}} \tilde{\kappa} \tilde{\lambda}}} 
- 2 \Gamma \ind{_{0 \lb{\mu}} ^0} C \ind{_{\rb{\nu} 0 \tilde{\kappa} \tilde{\lambda}}} 
- 2 \Gamma \ind{_{\lb{\nu}| 0|} ^\sigma} C \ind{_{\rb{\mu} \sigma \tilde{\kappa} \tilde{\lambda}}}
- 2 \Gamma \ind{_{\lb{\nu}| 0|} ^{\tilde{\sigma}}} C \ind{_{\rb{\mu} {\tilde{\sigma}} \tilde{\kappa} \tilde{\lambda}}} 
\\
- 4 \Gamma \ind{_{\lb{\mu}| \lb{\tilde{\kappa}}} ^\sigma} C \ind{_{\rb{\tilde{\lambda}}|\sigma| \rb{\nu} 0}}
- 4 \Gamma \ind{_{\lb{\mu}| \lb{\tilde{\kappa}}} ^{\tilde{\sigma}}} C \ind{_{\rb{\tilde{\lambda}}|{\tilde{\sigma}}| \rb{\nu} 0}} 
- 4 \Gamma \ind{_{\lb{\mu}| \lb{\tilde{\kappa}}} ^0} C \ind{_{\rb{\tilde{\lambda}}|0| \rb{\nu} 0}} \\
- 2 \Gamma \ind{_{0 \lb{\tilde{\kappa}}} ^\sigma} C \ind{_{\rb{\tilde{\lambda}} \sigma \mu \nu}} 
- 2 \Gamma \ind{_{0 \lb{\tilde{\kappa}}} ^{\tilde{\sigma}}} C \ind{_{\rb{\tilde{\lambda}} {\tilde{\sigma}} \mu \nu}}
- 2 \Gamma \ind{_{0 \lb{\tilde{\kappa}}} ^0} C \ind{_{\rb{\tilde{\lambda}} 0 \mu \nu}} \, ,
\end{multline}
\begin{multline}  \label{eq-Bianchi0d}
2 \partial \ind{_{\lb{\mu}}} C \ind{_{\rb{\nu} \tilde{\rho} \tilde{\kappa} 0}} + \partial \ind{_{\tilde{\rho}}} C \ind{_{\mu \nu \tilde{\kappa} 0}} =  - 2 g \ind{_{\lb{\mu} | \tilde{\kappa}}} A \ind{_{0 | \rb{\nu} \tilde{\rho}}}  
- 2 \Gamma \ind{_{[\mu \nu]} ^\sigma} C \ind{_{\tilde{\rho} \sigma \tilde{\kappa} 0}}
- 2 \Gamma \ind{_{[\mu \nu]} ^{\tilde{\sigma}}} C \ind{_{\tilde{\rho} {\tilde{\sigma}} \tilde{\kappa} 0}} 
- 2 \Gamma \ind{_{[\mu \nu]} ^0} C \ind{_{\tilde{\rho} 0 \tilde{\kappa} 0}} \\
- 2 \Gamma \ind{_{\tilde{\rho} \lb{\mu}} ^\sigma} C \ind{_{\rb{\nu} \sigma \tilde{\kappa} 0}}
- 2 \Gamma \ind{_{\tilde{\rho} \lb{\mu}} ^{\tilde{\sigma}}} C \ind{_{\rb{\nu} {\tilde{\sigma}} \tilde{\kappa} 0}} 
- 2 \Gamma \ind{_{\tilde{\rho} \lb{\mu}} ^0} C \ind{_{\rb{\nu} 0 \tilde{\kappa} 0}} 
- 2 \Gamma \ind{_{\lb{\nu}| \tilde{\rho}|} ^\sigma} C \ind{_{\rb{\mu} \sigma \tilde{\kappa} 0}}
- 2 \Gamma \ind{_{\lb{\nu}| \tilde{\rho}|} ^{\tilde{\sigma}}} C \ind{_{\rb{\mu} {\tilde{\sigma}} \tilde{\kappa} 0}} 
- 2 \Gamma \ind{_{\lb{\nu}| \tilde{\rho}|} ^0} C \ind{_{\rb{\mu} 0 \tilde{\kappa} 0}} \\
- 2 \Gamma \ind{_{\lb{\mu}| \tilde{\kappa}} ^\sigma} C \ind{_{0|\sigma| \rb{\nu} \tilde{\rho}}}
+ 2 \Gamma \ind{_{\lb{\mu}| 0} ^\sigma} C \ind{_{\tilde{\kappa}|\sigma| \rb{\nu} \tilde{\rho}}}
- 2 \Gamma \ind{_{\lb{\mu}| \tilde{\kappa}} ^{\tilde{\sigma}}} C \ind{_{0|{\tilde{\sigma}}| \rb{\nu} \tilde{\rho}}}
+ 2 \Gamma \ind{_{\lb{\mu}| 0} ^{\tilde{\sigma}}} C \ind{_{\tilde{\kappa}|{\tilde{\sigma}}| \rb{\nu} \tilde{\rho}}}  
\\
- \Gamma \ind{_{\tilde{\rho} \tilde{\kappa}} ^\sigma} C \ind{_{0 \sigma \mu \nu}}
+ \Gamma \ind{_{\tilde{\rho} 0} ^\sigma} C \ind{_{\tilde{\kappa} \sigma \mu \nu}}  
- \Gamma \ind{_{\tilde{\rho} \tilde{\kappa}} ^{\tilde{\sigma}}} C \ind{_{0 {\tilde{\sigma}} \mu \nu}}
+ \Gamma \ind{_{\tilde{\rho} 0} ^{\tilde{\sigma}}} C \ind{_{\tilde{\kappa} {\tilde{\sigma}} \mu \nu}}
\end{multline}
\begin{multline} \label{eq-Bianchi-1a}
2 \partial \ind{_{\lb{\tilde{\mu}}}} C \ind{_{\rb{\tilde{\nu}} \rho \tilde{\kappa} \lambda}} + \partial \ind{_{\rho}} C \ind{_{\tilde{\mu} \tilde{\nu} \tilde{\kappa} \lambda}} =  2 g \ind{_{\lb{\tilde{\mu}} | \lambda}} A \ind{_{\tilde{\kappa} | \rb{\tilde{\nu}} \rho}} - g \ind{_{\rho \tilde{\kappa}}} A \ind{_{\lambda \tilde{\mu} \tilde{\nu}}}  
- 2 \Gamma \ind{_{[\tilde{\mu} \tilde{\nu}]} ^\sigma} C \ind{_{\rho \sigma \tilde{\kappa} \lambda}}
- 2 \Gamma \ind{_{[\tilde{\mu} \tilde{\nu}]} ^{\tilde{\sigma}}} C \ind{_{\rho {\tilde{\sigma}} \tilde{\kappa} \lambda}}
- 2 \Gamma \ind{_{[\tilde{\mu} \tilde{\nu}]} ^0} C \ind{_{\rho 0 \tilde{\kappa} \lambda}} \\
- 2 \Gamma \ind{_{\rho \lb{\tilde{\mu}}} ^\sigma} C \ind{_{\rb{\tilde{\nu}} \sigma \tilde{\kappa} \lambda}}
- 2 \Gamma \ind{_{\rho \lb{\tilde{\mu}}} ^{\tilde{\sigma}}} C \ind{_{\rb{\tilde{\nu}} {\tilde{\sigma}} \tilde{\kappa} \lambda}} 
- 2 \Gamma \ind{_{\rho \lb{\tilde{\mu}}} ^0} C \ind{_{\rb{\tilde{\nu}} 0 \tilde{\kappa} \lambda}}  
- 2 \Gamma \ind{_{\lb{\tilde{\nu}}| \rho|} ^\sigma} C \ind{_{\rb{\tilde{\mu}} \sigma \tilde{\kappa} \lambda}}
- 2 \Gamma \ind{_{\lb{\tilde{\nu}}| \rho|} ^{\tilde{\sigma}}} C \ind{_{\rb{\tilde{\mu}} {\tilde{\sigma}} \tilde{\kappa} \lambda}} 
- 2 \Gamma \ind{_{\lb{\tilde{\nu}}| \rho|} ^0} C \ind{_{\rb{\tilde{\mu}} 0 \tilde{\kappa} \lambda}} \\
- 2 \Gamma \ind{_{\lb{\tilde{\mu}}| \tilde{\kappa}} ^\sigma} C \ind{_{\lambda |\sigma| \rb{\tilde{\nu}} \rho}}
- 2 \Gamma \ind{_{\lb{\tilde{\mu}}| \tilde{\kappa}} ^{\tilde{\sigma}}} C \ind{_{\lambda|{\tilde{\sigma}}| \rb{\tilde{\nu}} \rho}}
- 2 \Gamma \ind{_{\lb{\tilde{\mu}}| \tilde{\kappa}} ^0} C \ind{_{\lambda|0| \rb{\tilde{\nu}} \rho}} 
+ 2 \Gamma \ind{_{\lb{\tilde{\mu}}| \lambda} ^\sigma} C \ind{_{\tilde{\kappa} |\sigma| \rb{\tilde{\nu}} \rho}}
+ 2 \Gamma \ind{_{\lb{\tilde{\mu}}| \lambda} ^{\tilde{\sigma}}} C \ind{_{\tilde{\kappa}|{\tilde{\sigma}}| \rb{\tilde{\nu}} \rho}}
+ 2 \Gamma \ind{_{\lb{\tilde{\mu}}| \lambda} ^0} C \ind{_{\tilde{\kappa} |0| \rb{\tilde{\nu}} \rho}} \\
- \Gamma \ind{_{\rho \tilde{\kappa}} ^\sigma} C \ind{_{\lambda \sigma \tilde{\mu} \tilde{\nu}}}
- \Gamma \ind{_{\rho \tilde{\kappa}} ^{\tilde{\sigma}}} C \ind{_{\lambda {\tilde{\sigma}} \tilde{\mu} \tilde{\nu}}} 
- \Gamma \ind{_{\rho \tilde{\kappa}} ^0} C \ind{_{\lambda 0 \tilde{\mu} \tilde{\nu}}} 
+ \Gamma \ind{_{\rho \lambda} ^\sigma} C \ind{_{\tilde{\kappa} \sigma \tilde{\mu} \tilde{\nu}}}
+ \Gamma \ind{_{\rho \lambda} ^{\tilde{\sigma}}} C \ind{_{\tilde{\kappa} {\tilde{\sigma}} \tilde{\mu} \tilde{\nu}}}
+ \Gamma \ind{_{\rho \lambda} ^0} C \ind{_{\tilde{\kappa} 0 \tilde{\mu} \tilde{\nu}}} \, ,
\end{multline}
\begin{multline}  \label{eq-Bianchi-1b}
\partial \ind{_{\tilde{\mu}}} C \ind{_{\nu 0 \tilde{\kappa} 0}} + \partial \ind{_{\nu}} C \ind{_{0 \tilde{\mu} \tilde{\kappa} 0}} + \partial \ind{_{0}} C \ind{_{\tilde{\mu} \nu \tilde{\kappa} 0}} =  - g \ind{_{\nu \tilde{\kappa}}} A \ind{_{0 0 \tilde{\mu}}} + A \ind{_{\tilde{\kappa} \tilde{\mu} \nu}} 
- \Gamma \ind{_{\tilde{\mu} \nu} ^\sigma} C \ind{_{0 \sigma \tilde{\kappa} 0}}
- \Gamma \ind{_{\tilde{\mu} \nu} ^{\tilde{\sigma}}} C \ind{_{0 {\tilde{\sigma}} \tilde{\kappa} 0}} 
+ \Gamma \ind{_{\nu \tilde{\mu}} ^\sigma} C \ind{_{0 \sigma \tilde{\kappa} 0}}
+ \Gamma \ind{_{\nu \tilde{\mu}} ^{\tilde{\sigma}}} C \ind{_{0 {\tilde{\sigma}} \tilde{\kappa} 0}} \\
- \Gamma \ind{_{\nu 0} ^\sigma} C \ind{_{\tilde{\mu} \sigma \tilde{\kappa} 0}}
- \Gamma \ind{_{\nu 0} ^{\tilde{\sigma}}} C \ind{_{\tilde{\mu} {\tilde{\sigma}} \tilde{\kappa} 0}} 
+ \Gamma \ind{_{0 \nu} ^\sigma} C \ind{_{\tilde{\mu} \sigma \tilde{\kappa} 0}}
+ \Gamma \ind{_{0 \nu} ^{\tilde{\sigma}}} C \ind{_{\tilde{\mu} {\tilde{\sigma}} \tilde{\kappa} 0}} 
+ \Gamma \ind{_{0 \nu} ^0} C \ind{_{\tilde{\mu} 0 \tilde{\kappa} 0}} 
- \Gamma \ind{_{0 \tilde{\mu}} ^\sigma} C \ind{_{\nu \sigma \tilde{\kappa} 0}}
- \Gamma \ind{_{0 \tilde{\mu}} ^{\tilde{\sigma}}} C \ind{_{\nu {\tilde{\sigma}} \tilde{\kappa} 0}} 
- \Gamma \ind{_{0 \tilde{\mu}} ^0} C \ind{_{\nu 0 \tilde{\kappa} 0}} \\
+ \Gamma \ind{_{\tilde{\mu} 0} ^\sigma} C \ind{_{\nu \sigma \tilde{\kappa} 0}}
+ \Gamma \ind{_{\tilde{\mu} 0} ^{\tilde{\sigma}}} C \ind{_{\nu {\tilde{\sigma}} \tilde{\kappa} 0}} 
- \Gamma \ind{_{\tilde{\mu} \tilde{\kappa}} ^\sigma} C \ind{_{0 \sigma \nu 0}}
- \Gamma \ind{_{\tilde{\mu} \tilde{\kappa}} ^{\tilde{\sigma}}} C \ind{_{0 {\tilde{\sigma}} \nu 0}} 
+ \Gamma \ind{_{\tilde{\mu} 0} ^\sigma} C \ind{_{\tilde{\kappa} \sigma \nu 0}} 
+ \Gamma \ind{_{\tilde{\mu} 0} ^{\tilde{\sigma}}} C \ind{_{\tilde{\kappa} {\tilde{\sigma}} \nu 0}} \\
- \Gamma \ind{_{\nu \tilde{\kappa}} ^\sigma} C \ind{_{0 \sigma 0 \tilde{\mu}}}
- \Gamma \ind{_{\nu \tilde{\kappa}} ^{\tilde{\sigma}}} C \ind{_{0 {\tilde{\sigma}} 0 \tilde{\mu}}}
+ \Gamma \ind{_{\nu 0} ^\sigma} C \ind{_{\tilde{\kappa} \sigma 0 \tilde{\mu}}} 
+ \Gamma \ind{_{\nu 0} ^{\tilde{\sigma}}} C \ind{_{\tilde{\kappa} {\tilde{\sigma}} 0 \tilde{\mu}}}  
- \Gamma \ind{_{0 \tilde{\kappa}} ^\sigma} C \ind{_{0 \sigma \tilde{\mu} \nu}}
- \Gamma \ind{_{0 \tilde{\kappa}} ^{\tilde{\sigma}}} C \ind{_{0 {\tilde{\sigma}} \tilde{\mu} \nu}} \\
+ \Gamma \ind{_{0 0} ^\sigma} C \ind{_{\tilde{\kappa} \sigma \tilde{\mu} \nu}}
+ \Gamma \ind{_{0 0} ^{\tilde{\sigma}}} C \ind{_{\tilde{\kappa} {\tilde{\sigma}} \tilde{\mu} \nu}} 
   \, ,
\end{multline}
\begin{multline}  \label{eq-Bianchi-2a}
2 \partial \ind{_{\lb{\tilde{\mu}}}} C \ind{_{\rb{\tilde{\nu}} \rho \tilde{\kappa} 0}} + \partial \ind{_{\rho}} C \ind{_{\tilde{\mu} \tilde{\nu} \tilde{\kappa} 0}} = - g \ind{_{\rho \tilde{\kappa}}} A \ind{_{0 \tilde{\mu} \tilde{\nu}}} 
- 2 \Gamma \ind{_{[\tilde{\mu} \tilde{\nu}]} ^\sigma} C \ind{_{\rho \sigma \tilde{\kappa} 0}}
- 2 \Gamma \ind{_{[\tilde{\mu} \tilde{\nu}]} ^{\tilde{\sigma}}} C \ind{_{\rho {\tilde{\sigma}} \tilde{\kappa} 0}}
- 2 \Gamma \ind{_{[\tilde{\mu} \tilde{\nu}]} ^0} C \ind{_{\rho 0 \tilde{\kappa} 0}} \\
- 2 \Gamma \ind{_{\rho \lb{\tilde{\mu}}} ^\sigma} C \ind{_{\rb{\tilde{\nu}} \sigma \tilde{\kappa} 0}} - 2 \Gamma \ind{_{\rho \lb{\tilde{\mu}}} ^{\tilde{\sigma}}} C \ind{_{\rb{\tilde{\nu}} {\tilde{\sigma}} \tilde{\kappa} 0}} - 2 \Gamma \ind{_{\rho \lb{\tilde{\mu}}} ^0} C \ind{_{\rb{\tilde{\nu}} 0 \tilde{\kappa} 0}} - 2 \Gamma \ind{_{\lb{\tilde{\nu}}| \rho|} ^\sigma} C \ind{_{\rb{\tilde{\mu}} \sigma \tilde{\kappa} 0}} - 2 \Gamma \ind{_{\lb{\tilde{\nu}}| \rho|} ^{\tilde{\sigma}}} C \ind{_{\rb{\tilde{\mu}} {\tilde{\sigma}} \tilde{\kappa} 0}} - 2 \Gamma \ind{_{\lb{\tilde{\nu}}| \rho|} ^0} C \ind{_{\rb{\tilde{\mu}} 0 \tilde{\kappa} 0}} \\ - 2 \Gamma \ind{_{\lb{\tilde{\mu}}| \tilde{\kappa}} ^\sigma} C \ind{_{0 \sigma | \rb{\tilde{\nu}} \rho}} + 2 \Gamma \ind{_{\lb{\tilde{\mu}}| 0} ^\sigma} C \ind{_{\tilde{\kappa} \sigma | \rb{\tilde{\nu}} \rho}} - 2 \Gamma \ind{_{\lb{\tilde{\mu}}| \tilde{\kappa}} ^{\tilde{\sigma}}} C \ind{_{0  {\tilde{\sigma}} | \rb{\tilde{\nu}} \rho}} + 2 \Gamma \ind{_{\lb{\tilde{\mu}}| 0} ^{\tilde{\sigma}}} C \ind{_{\tilde{\kappa} {\tilde{\sigma}} | \rb{\tilde{\nu}} \rho}} \\
- \Gamma \ind{_{\rho \tilde{\kappa}} ^\sigma} C \ind{_{0 \sigma \tilde{\mu} \tilde{\nu}}} + \Gamma \ind{_{\rho 0} ^\sigma} C \ind{_{\tilde{\kappa} \sigma \tilde{\mu} \tilde{\nu}}} - \Gamma \ind{_{\rho \tilde{\kappa}} ^{\tilde{\sigma}}} C \ind{_{0 {\tilde{\sigma}}  \tilde{\mu} \tilde{\nu}}} + \Gamma \ind{_{\rho 0} ^{\tilde{\sigma}}} C \ind{_{\tilde{\kappa} {\tilde{\sigma}}  \tilde{\mu} \tilde{\nu}}} \, ,
\end{multline}
\begin{multline} \label{eq-Bianchi-2b}
2 \partial \ind{_{\lb{\tilde{\mu}}}} C \ind{_{\rb{\tilde{\nu}} 0 \tilde{\kappa} \lambda}} + \partial \ind{_{0}} C \ind{_{\tilde{\mu} \tilde{\nu} \tilde{\kappa} \lambda}} = - 2 g \ind{_{\lb{\tilde{\mu}} | \lambda}} A \ind{_{\tilde{\kappa} | \rb{\tilde{\nu}} 0}} - 2 \Gamma \ind{_{[\tilde{\mu} \tilde{\nu}]} ^\sigma} C \ind{_{0 \sigma \tilde{\kappa} \lambda}} - 2 \Gamma \ind{_{[\tilde{\mu} \tilde{\nu}]} ^{\tilde{\sigma}}} C \ind{_{0 {\tilde{\sigma}} \tilde{\kappa} \lambda}} \\
- 2 \Gamma \ind{_{0 \lb{\tilde{\mu}}} ^\sigma} C \ind{_{\rb{\tilde{\nu}} \sigma \tilde{\kappa} \lambda}} - 2 \Gamma \ind{_{0 \lb{\tilde{\mu}}} ^{\tilde{\sigma}}} C \ind{_{\rb{\tilde{\nu}} {\tilde{\sigma}} \tilde{\kappa} \lambda}} - 2 \Gamma \ind{_{0 \lb{\tilde{\mu}}} ^0} C \ind{_{\rb{\tilde{\nu}} 0 \tilde{\kappa} \lambda}} - 2 \Gamma \ind{_{\lb{\tilde{\nu}}| 0|} ^\sigma} C \ind{_{\rb{\tilde{\mu}} \sigma \tilde{\kappa} \lambda}} - 2 \Gamma \ind{_{\lb{\tilde{\nu}}| 0|} ^{\tilde{\sigma}}} C \ind{_{\rb{\tilde{\mu}} {\tilde{\sigma}} \tilde{\kappa} \lambda}} \\
- 2 \Gamma \ind{_{\lb{\tilde{\mu}}| \tilde{\kappa}} ^\sigma} C \ind{_{\lambda \sigma| \rb{\tilde{\nu}} 0}} + 2 \Gamma \ind{_{\lb{\tilde{\mu}}| \lambda} ^\sigma} C \ind{_{\tilde{\kappa} \sigma| \rb{\tilde{\nu}} 0}} - 2 \Gamma \ind{_{\lb{\tilde{\mu}} | \tilde{\kappa}} ^{\tilde{\sigma}}} C \ind{_{\lambda{\tilde{\sigma}}| \rb{\tilde{\nu}} 0}} + 2 \Gamma \ind{_{\lb{\tilde{\mu}} | \lambda} ^{\tilde{\sigma}}} C \ind{_{\tilde{\kappa} {\tilde{\sigma}}| \rb{\tilde{\nu}} 0}} - 2 \Gamma \ind{_{\lb{\tilde{\mu}}| \tilde{\kappa}} ^0} C \ind{_{\lambda 0 | \rb{\tilde{\nu}} 0}} + 2 \Gamma \ind{_{\lb{\tilde{\mu}}| \lambda} ^0} C \ind{_{\tilde{\kappa} 0 | \rb{\tilde{\nu}} 0}} \\ - \Gamma \ind{_{0 \tilde{\kappa}} ^\sigma} C \ind{_{\lambda \sigma \tilde{\mu} \tilde{\nu}}} + \Gamma \ind{_{0 \lambda} ^\sigma} C \ind{_{\tilde{\kappa} \sigma \tilde{\mu} \tilde{\nu}}} - \Gamma \ind{_{0 \tilde{\kappa}} ^{\tilde{\sigma}}} C \ind{_{\lambda {\tilde{\sigma}} \tilde{\mu} \tilde{\nu}}} + \Gamma \ind{_{0 \lambda} ^{\tilde{\sigma}}} C \ind{_{\tilde{\kappa} {\tilde{\sigma}} \tilde{\mu} \tilde{\nu}}} - \Gamma \ind{_{0 \tilde{\kappa}} ^0} C \ind{_{\lambda  0 \tilde{\mu} \tilde{\nu}}} + \Gamma \ind{_{0 \lambda} ^0} C \ind{_{\tilde{\kappa} 0 \tilde{\mu} \tilde{\nu}}} \, ,
\end{multline}
\begin{multline} \label{eq-Bianchi-3a}
2 \partial \ind{_{\lb{\tilde{\mu}}}} C \ind{_{\rb{\tilde{\nu}} \rho \tilde{\kappa} \tilde{\lambda}}} + \partial \ind{_{\rho}} C \ind{_{\tilde{\mu} \tilde{\nu} \tilde{\kappa} \tilde{\lambda}}} = - 2 g \ind{_{\rho \lb{\tilde{\kappa}}}} A \ind{_{\rb{\tilde{\lambda}} \tilde{\mu} \tilde{\nu}}} - 2 \Gamma \ind{_{[\tilde{\mu} \tilde{\nu}]} ^\sigma} C \ind{_{\rho \sigma \tilde{\kappa} \tilde{\lambda}}} - 2 \Gamma \ind{_{[\tilde{\mu} \tilde{\nu}]} ^{\tilde{\sigma}}} C \ind{_{\rho {\tilde{\sigma}} \tilde{\kappa} \tilde{\lambda}}} - 2 \Gamma \ind{_{[\tilde{\mu} \tilde{\nu}]} ^0} C \ind{_{\rho 0 \tilde{\kappa} \tilde{\lambda}}} \\
- 2 \Gamma \ind{_{\rho \lb{\tilde{\mu}}} ^\sigma} C \ind{_{\rb{\tilde{\nu}} \sigma \tilde{\kappa} \tilde{\lambda}}} - 2 \Gamma \ind{_{\rho \lb{\tilde{\mu}}} ^{\tilde{\sigma}}} C \ind{_{\rb{\tilde{\nu}} {\tilde{\sigma}} \tilde{\kappa} \tilde{\lambda}}} - 2 \Gamma \ind{_{\rho \lb{\tilde{\mu}}} ^0} C \ind{_{\rb{\tilde{\nu}} 0 \tilde{\kappa} \tilde{\lambda}}} - 2 \Gamma \ind{_{\lb{\tilde{\nu}}| \rho|} ^\sigma} C \ind{_{\rb{\tilde{\mu}} \sigma \tilde{\kappa} \tilde{\lambda}}} - 2 \Gamma \ind{_{\lb{\tilde{\nu}}| \rho|} ^{\tilde{\sigma}}} C \ind{_{\rb{\tilde{\mu}} {\tilde{\sigma}} \tilde{\kappa} \tilde{\lambda}}} - 2 \Gamma \ind{_{\lb{\tilde{\nu}}| \rho|} ^0} C \ind{_{\rb{\tilde{\mu}} 0 \tilde{\kappa} \tilde{\lambda}}} \\ - 4 \Gamma \ind{_{\lb{\tilde{\mu}}| \lb{\tilde{\kappa}}} ^\sigma} C \ind{_{\rb{\tilde{\lambda}}|\sigma| \rb{\tilde{\nu}} \rho}} - 4 \Gamma \ind{_{\lb{\tilde{\mu}}| \lb{\tilde{\kappa}}} ^{\tilde{\sigma}}} C \ind{_{\rb{\tilde{\lambda}}|{\tilde{\sigma}}| \rb{\tilde{\nu}} \rho}} - 4 \Gamma \ind{_{\lb{\tilde{\mu}}| \lb{\tilde{\kappa}}} ^0} C \ind{_{\rb{\tilde{\lambda}}| 0 | \rb{\tilde{\nu}} \rho}} - 2 \Gamma \ind{_{\rho \lb{\tilde{\kappa}}} ^\sigma} C \ind{_{\rb{\tilde{\lambda}} \sigma \tilde{\mu} \tilde{\nu}}} - 2 \Gamma \ind{_{\rho \lb{\tilde{\kappa}}} ^{\tilde{\sigma}}} C \ind{_{\rb{\tilde{\lambda}} {\tilde{\sigma}} \tilde{\mu} \tilde{\nu}}} - 2 \Gamma \ind{_{\rho \lb{\tilde{\kappa}}} ^0} C \ind{_{\rb{\tilde{\lambda}}  0 \tilde{\mu} \tilde{\nu}}} \, ,
\end{multline}
\begin{multline} \label{eq-Bianchi-3b}
2 \partial \ind{_{\lb{\tilde{\mu}}}} C \ind{_{\rb{\tilde{\nu}} 0 \tilde{\kappa} 0}} + \partial \ind{_{0}} C \ind{_{\tilde{\mu} \tilde{\nu} \tilde{\kappa} 0}} = A \ind{_{\tilde{\kappa} \tilde{\mu} \tilde{\nu}}} - 2 \Gamma \ind{_{[\tilde{\mu} \tilde{\nu}]} ^\sigma} C \ind{_{0 \sigma \tilde{\kappa} 0}} - 2 \Gamma \ind{_{[\tilde{\mu} \tilde{\nu}]} ^{\tilde{\sigma}}} C \ind{_{0 {\tilde{\sigma}} \tilde{\kappa} 0}} \\
- 2 \Gamma \ind{_{0 \lb{\tilde{\mu}}} ^\sigma} C \ind{_{\rb{\tilde{\nu}} \sigma \tilde{\kappa} 0}} - 2 \Gamma \ind{_{0 \lb{\tilde{\mu}}} ^{\tilde{\sigma}}} C \ind{_{\rb{\tilde{\nu}} {\tilde{\sigma}} \tilde{\kappa} 0}} - 2 \Gamma \ind{_{0 \lb{\tilde{\mu}}} ^0} C \ind{_{\rb{\tilde{\nu}} 0 \tilde{\kappa} 0}} - 2 \Gamma \ind{_{\lb{\tilde{\nu}}| 0|} ^\sigma} C \ind{_{\rb{\tilde{\mu}} \sigma \tilde{\kappa} 0}} - 2 \Gamma \ind{_{\lb{\tilde{\nu}}| 0|} ^{\tilde{\sigma}}} C \ind{_{\rb{\tilde{\mu}} {\tilde{\sigma}} \tilde{\kappa} 0}}  \\ - 2 \Gamma \ind{_{\lb{\tilde{\mu}}| \tilde{\kappa}} ^\sigma} C \ind{_{0 \sigma| \rb{\tilde{\nu}} 0}}  + 2 \Gamma \ind{_{\lb{\tilde{\mu}}| 0} ^\sigma} C \ind{_{\tilde{\kappa} \sigma| \rb{\tilde{\nu}} 0}} - 2 \Gamma \ind{_{\lb{\tilde{\mu}}| \tilde{\kappa}} ^{\tilde{\sigma}}} C \ind{_{0 {\tilde{\sigma}}| \rb{\tilde{\nu}} 0}} + 2 \Gamma \ind{_{\lb{\tilde{\mu}}| 0} ^{\tilde{\sigma}}} C \ind{_{\tilde{\kappa} {\tilde{\sigma}}| \rb{\tilde{\nu}} 0}}  \\ - \Gamma \ind{_{0 \tilde{\kappa}} ^\sigma} C \ind{_{0 \sigma \tilde{\mu} \tilde{\nu}}} + \Gamma \ind{_{0 0} ^\sigma} C \ind{_{\tilde{\kappa} \sigma \tilde{\mu} \tilde{\nu}}} -  \Gamma \ind{_{0 \tilde{\kappa}} ^{\tilde{\sigma}}} C \ind{_{0 {\tilde{\sigma}} \tilde{\mu} \tilde{\nu}}}  +  \Gamma \ind{_{0 0} ^{\tilde{\sigma}}} C \ind{_{\tilde{\kappa} {\tilde{\sigma}} \tilde{\mu} \tilde{\nu}}}  \, ,
\end{multline}
\begin{multline} \label{eq-Bianchi-4}
2 \partial \ind{_{\lb{\tilde{\mu}}}} C \ind{_{\rb{\tilde{\nu}} 0 \tilde{\kappa} \tilde{\lambda}}} + \partial \ind{_{0}} C \ind{_{\tilde{\mu} \tilde{\nu} \tilde{\kappa} \tilde{\lambda}}} =  
 - 2 \Gamma \ind{_{[\tilde{\mu} \tilde{\nu}]} ^\sigma} C \ind{_{0 \sigma \tilde{\kappa} \tilde{\lambda}}}
- 2 \Gamma \ind{_{[\tilde{\mu} \tilde{\nu}]} ^{\tilde{\sigma}}} C \ind{_{0 {\tilde{\sigma}} \tilde{\kappa} \tilde{\lambda}}}
- 2 \Gamma \ind{_{0 \lb{\tilde{\mu}}} ^\sigma} C \ind{_{\rb{\tilde{\nu}} \sigma \tilde{\kappa} \tilde{\lambda}}}
- 2 \Gamma \ind{_{0 \lb{\tilde{\mu}}} ^{\tilde{\sigma}}} C \ind{_{\rb{\tilde{\nu}} {\tilde{\sigma}} \tilde{\kappa} \tilde{\lambda}}}
- 2 \Gamma \ind{_{0 \lb{\tilde{\mu}}} ^0} C \ind{_{\rb{\tilde{\nu}} 0 \tilde{\kappa} \tilde{\lambda}}} \\
- 2 \Gamma \ind{_{\lb{\tilde{\nu}}| 0|} ^\sigma} C \ind{_{\rb{\tilde{\mu}} \sigma \tilde{\kappa} \tilde{\lambda}}}
- 2 \Gamma \ind{_{\lb{\tilde{\nu}}| 0|} ^{\tilde{\sigma}}} C \ind{_{\rb{\tilde{\mu}} {\tilde{\sigma}} \tilde{\kappa} \tilde{\lambda}}}
- 4 \Gamma \ind{_{\lb{\tilde{\mu}}| \lb{\tilde{\kappa}}} ^\sigma} C \ind{_{\rb{\tilde{\lambda}}|\sigma| \rb{\tilde{\nu}} 0}}
- 4 \Gamma \ind{_{\lb{\tilde{\mu}}| \lb{\tilde{\kappa}}} ^{\tilde{\sigma}}} C \ind{_{\rb{\tilde{\lambda}}|{\tilde{\sigma}}| \rb{\tilde{\nu}} 0}}
- 4 \Gamma \ind{_{\lb{\tilde{\mu}}| \lb{\tilde{\kappa}}} ^0} C \ind{_{\rb{\tilde{\lambda}}|0| \rb{\tilde{\nu}} 0}} \\
- 2 \Gamma \ind{_{0 \lb{\tilde{\kappa}}} ^\sigma} C \ind{_{\rb{\tilde{\lambda}}|\sigma| \tilde{\mu} \tilde{\nu}}}
- 2 \Gamma \ind{_{0 \lb{\tilde{\kappa}}} ^{\tilde{\sigma}}} C \ind{_{\rb{\tilde{\lambda}}|{\tilde{\sigma}}| \tilde{\mu} \tilde{\nu}}} 
- 2 \Gamma \ind{_{0 \lb{\tilde{\kappa}}} ^0} C \ind{_{\rb{\tilde{\lambda}}|0| \tilde{\mu} \tilde{\nu}}} 
\, ,
\end{multline}
\begin{multline}  \label{eq-Bianchi-5}
\partial \ind{_{\lb{\tilde{\mu}}}} C \ind{_{\tilde{\nu} \rb{\tilde{\rho}} \tilde{\kappa} \tilde{\lambda}}}  = 
- 2 \Gamma \ind{_{\lb{\tilde{\mu}} \tilde{\nu}} ^\sigma} C \ind{_{\rb{\tilde{\rho}} \sigma \tilde{\kappa} \tilde{\lambda}}}
- 2 \Gamma \ind{_{\lb{\tilde{\mu}} \tilde{\nu}} ^{\tilde{\sigma}}} C \ind{_{\rb{\tilde{\rho}} \tilde{\sigma} \tilde{\kappa} \tilde{\lambda}}}
- 2 \Gamma \ind{_{\lb{\tilde{\mu}} \tilde{\nu}} ^0} C \ind{_{\rb{\tilde{\rho}} 0 \tilde{\kappa} \tilde{\lambda}}} \\
- 2 \Gamma \ind{_{\lb{\tilde{\mu}}| \lb{\tilde{\kappa}}} ^\sigma} C \ind{_{\rb{\tilde{\lambda}}|\sigma| \tilde{\nu} \rb{\tilde{\rho}}}} 
- 2 \Gamma \ind{_{\lb{\tilde{\mu}}| \lb{\tilde{\kappa}}} ^{\tilde{\sigma}}} C \ind{_{\rb{\tilde{\lambda}}|\tilde{\sigma}| \tilde{\nu} \rb{\tilde{\rho}}}}
- 2 \Gamma \ind{_{\lb{\tilde{\mu}}| \lb{\tilde{\kappa}}} ^0} C \ind{_{\rb{\tilde{\lambda}}|0| \tilde{\nu} \rb{\tilde{\rho}}}} \, ,
\end{multline}

\bibliography{biblio}

\end{document}